\documentclass{article}
\usepackage{amsmath,amsfonts,amssymb,amsthm,latexsym,graphicx}

\allowdisplaybreaks[4]

\usepackage{enumerate}
\usepackage{color}
\usepackage{natbib}
\usepackage{algorithm}
\usepackage[noend]{algpseudocode}

\usepackage[colorlinks=true,citecolor=blue,pdfpagemode=UseNone,pdfstartview=FitH]{hyperref}

\renewcommand{\d}{\,\mathrm{d}}
\newcommand{\dd}{\mathrm{d}}
           
\newcommand{\q}{q}
\newcommand{\p}{Q}

\newcommand{\E}{\mathbb{E}}

\newcommand{\R}{\mathbb{R}}
\newcommand{\N}{\mathbb{N}}

\DeclareMathOperator{\ucp}{\mathbb{P}}

\DeclareMathOperator{\DM}{DM}
\newcommand{\id}{1}

\newcommand{\PPP}{\mathcal{P}}

\newcommand{\U}{\mathrm{U}}

\theoremstyle{plain}
\newtheorem{theorem}{Theorem}[section]
\newtheorem{corollary}[theorem]{Corollary}
\newtheorem{lemma}[theorem]{Lemma}
\newtheorem{proposition}[theorem]{Proposition}

\theoremstyle{definition}

\newtheorem{example}[theorem]{Example}

\theoremstyle{remark}
\newtheorem{remark}[theorem]{Remark}

\usepackage[misc]{ifsym}

\usepackage[margin=1.5in]{geometry}

\begin{document}
  \title{Admissible ways of merging p-values under arbitrary dependence}
  \author{Vladimir Vovk\thanks%
    {Department of Computer Science,
    Royal Holloway, University of London,
    Egham, Surrey, UK.
    E-mail: \href{mailto:v.vovk@rhul.ac.uk}{v.vovk@rhul.ac.uk}.}
  \and
    Bin Wang\thanks{RCSDS, NCMIS, Academy of Mathematics and Systems Science,
    Chinese Academy of Sciences, Beijing, China.
    Email: \href{mailto:wangbin@amss.ac.cn}{wangbin@amss.ac.cn}.}
  \and
    Ruodu Wang\thanks%
    {Department of Statistics and Actuarial Science,
    University of Waterloo,
    Waterloo, Ontario, Canada.
    E-mail: \href{mailto:wang@uwaterloo.ca}{wang@uwaterloo.ca}.}
  }

  \maketitle

  \begin{abstract}
    Methods of merging several p-values into a single p-value are important in their own right and widely used in multiple hypothesis testing.
    This paper is the first to systematically study the admissibility (in Wald's sense) of p-merging functions and their domination structure,
    without any information on the dependence structure of the input p-values.
    As a technical tool we use the notion of e-values, which are alternatives to p-values recently promoted by several authors.
    We obtain several results on the representation of admissible p-merging functions via e-values and on (in)admissibility of existing p-merging functions.
    By introducing new admissible p-merging functions,
    we show that some classic merging methods can be strictly improved to enhance power without compromising validity under arbitrary dependence.
  \end{abstract}


\section{Introduction}
 
A common task in multiple testing of a single hypothesis and testing multiple hypotheses
is to combine several p-values into one p-value
(without using the underlying data).
If one assumes independence (or another specific dependence structure) among p-values testing a scientific hypothesis $H_0$,
then the combined p-value is effectively testing a composition of $H_0$ and the independence assumption.
A rejection obtained from such a test may be due to statistical evidence against either independence or the scientific hypothesis of interest (or both).
As we typically only have one realization of a bunch of p-values, it is not possible to identify the source of rejection.
Hence, such a method cannot be justified unless convincing evidence of independence is supplied;
however, as argued by \citet[pp.~50--51]{Efron:2010}, neither independence nor positive regression dependence,
which is often assumed in literature,
is realistic in large-scale inference. 
Therefore, it is important to consider merging methods that are valid
without available information on the dependence structure.
In general, dropping the assumption of independence makes the problem of merging p-values more difficult:
see, e.g., \citet[Section 1]{Vovk/Wang:ie}.

Several valid merging methods are known for arbitrary dependence structure among p-values;
these methods do not make any other assumptions about the input p-values
(such as assumptions about their support; those p-values can be continuous or discrete),
and their validity is exact (and not, e.g., asymptotic or approximate).
Of course, such methods, which we will call \emph{universally valid}, come at a cost of power.
The most well-known one is arguably the Bonferroni correction, which uses the minimum of p-values times the number of tests.
Several other methods include those of \citet{Ruger:1978} and \citet{Hommel:1983}, based on order statistics of p-values,
and those of \citet{Vovk/Wang:2020Biometrika}, based on generalized means of p-values; see Section \ref{sec:3} for details of these merging methods.
These methods include versions of the method of \citet{Simes:1986} and the harmonic mean of \citet{Wilson:2019}
that are adjusted to be valid under arbitrary dependence.

Our study gives rise to new universally valid merging methods
(in particular, free of any dependence assumptions)
that are more powerful than the ones in the existing literature.
Perhaps the main of these methods is what we call the grid harmonic method $H^*_K$,
which improves on the method of \citet{Hommel:1983}.
Our simulation studies demonstrate that the improvement is very substantial,
which shows in applications that are important in practice,
such as multiple hypothesis testing.
See Sections \ref{sec:o} and~\ref{sec:9}.

The main objective of this paper is to study the domination structure among universally valid functions for merging p-values,
henceforth \emph{p-merging functions}. 
In particular, we do not discuss methods that are valid for specific classes of dependence structures;
for the latter, see e.g., \citet{sarkar1998some}, \cite{Wilson:2019}, and \cite{Liu/Xie:2020},
as well as \cite{Chen:2020} for a summary.
A p-merging function is \emph{admissible} if it is not strictly dominated by any other p-merging function.
Ideally, ceteris paribus, only admissible p-merging functions should be used, as other methods can be strictly improved.
It turns out that admissibility and domination structure among p-merging functions give rise to highly non-trivial mathematical challenges. 
We are mainly interested in homogeneous and symmetric p-merging functions,
as most  p-merging functions used in practice are of this kind.

Let us briefly summarize our main contributions.
First, the merging function of \citet{Simes:1986} (valid under the assumption of independence)
is the minimum of all symmetric p-merging functions (Theorem \ref{thm:Simes}).
Second, we give two representation results (Theorems \ref{th:e} and \ref{pr:e}) of admissible p-merging functions
which are naturally connected to e-values \citep{Vovk/Wang:E,Shafer:2019,Grunwald/etal:2019},
our important technical tool, via a duality argument.
Third, we provide an analytical condition for a calibrator to induce an admissible p-merging function (Theorem \ref{th:admissible}).
Fourth, we proceed to show that the classic p-merging function of \citet{Hommel:1983} and the scaled averaging functions of \citet{Vovk/Wang:2020Biometrika}
can be strictly improved to their more powerful versions (Theorems \ref{th:o1} and \ref{th:m1}),
whereas the scaled order statistics of \citet{Ruger:1978} are generally admissible after a trivial modification (Theorem \ref{pr:o1}).
Various other smaller results on properties and comparisons of  p-merging functions are obtained during our scientific journey.

Our p-merging functions can be directly applied to any procedures for multiple hypothesis testing,
such as those of \citet{Genovese/Wasserman:2004} and \citet{Goeman/Solari:2011}; see Section \ref{sec:9} for simulation studies.
In addition to the grid harmonic p-merging function $H_K^*$, strictly dominating the merging function of \citet{Hommel:1983},
we design an admissible merging function $F_{-1,K}^*$ strictly dominating the harmonic merging function of \citet{Vovk/Wang:2020Biometrika}.
The Hommel and harmonic merging functions have been shown to be special among two general families (see Section 4 of \citet{Chen:2020})
with wide applications, attractive properties, and good empirical performance (e.g., \citet{Wilson:2020}).

Several mathematical results in this paper are quite sophisticated and surprising.
In Theorem \ref{th:o1}, we find the unexpected result that $H_K^*$, while admissible for non-prime numbers $K$ of the input p-values,
is not admissible in general for prime $K$.
For a given p-merging function, it is generally difficult to prove or disprove its admissibility, or to construct a dominating admissible p-merging function.
The proofs of our results rely on recent techniques in robust risk aggregation and dependence modeling.
In particular, advanced results on joint mixability in \citet{WW11, WW16} play a crucial role in proving Theorem \ref{th:admissible},
and many other results in the paper require complicated constructions of specific dependence structure among p-variables.
Some open questions are presented in concluding Section \ref{sec:10} for the interested reader.

\begin{remark}
  A useful distinction, introduced in \citet{Good:1958},
  is between statistical tests in parallel and in series.
  In the former case the input p-values are all based on the same evidence,
  and we are mostly interested in this case.
  In testing in series the input p-values may be based on bodies of evidence
  that we may judge to be independent,
  and then the assumption of independence of p-values may be justified.
  More generally, one may consider sequentially dependent (or \emph{sequential}) p-values;
  cf. \citet[Section 2]{Vovk/Wang:ie}.
\end{remark}

\section{P-merging functions and basic properties}
\label{sec:2}

Without loss of generality we fix an atomless probability space $(\Omega,\mathcal A,Q)$
(see, e.g., \citet[Proposition~A.27]{Follmer/Schied:2011} or \citet[Appendix D]{Vovk/Wang:E}).
A \emph{p-variable} is a random variable $P:\Omega\to[0,\infty)$ satisfying
\[
  Q(P\le \epsilon) \le \epsilon \text{ for all }\epsilon \in (0,1).
\]
The set of all p-variables is denoted by $\mathcal{P}_Q$.  
Throughout, $K\ge 2$ is an integer. 
A \emph{p-merging function} of $K$ p-values is an increasing Borel function $F:[0,\infty)^K\to[0,\infty)$
such that $F(P_1,\dots,P_K)\in \mathcal{P}_Q$ whenever $P_1,\dots,P_K \in \mathcal{P}_Q$.
(Notice that the joint distribution of $P_1,\dots,P_K\in \mathcal{P}_Q$ can be arbitrary.)
A p-merging function $F$ is \emph{symmetric} if $F(\mathbf p)$ is invariant under any permutation of $\mathbf p$,
and it is \emph{homogeneous} if $F(\lambda\mathbf p)=\lambda F(\mathbf p)$ for all $\lambda \in (0,1]$ and $\mathbf p$ with $F(\mathbf p) \le 1$.
All p-merging functions that we encounter in this paper are homogeneous and symmetric.
Although we allow the domain of $F$ to be $[0,\infty)^K$ in order to simplify presentation,
the informative part of $F$ is its restriction to $[0,1]^K$.
Throughout, $\mathbf 0$ is the $K$-vector of zeros, $\mathbf 1$ is the $K$-vector of ones,
and all vector inequalities and the operation $\wedge$ of taking the minimum of two vectors are component-wise.
For  $a,b,x,y\in \R$, $ax \wedge by $ should be understood as $ (ax)\wedge (by)$.

We say that a p-merging function $F$ \emph{dominates} a p-merging function $G$ if $F\le G$.
The domination is \emph{strict} if, in addition, $F(\mathbf{p})<G(\mathbf{p})$ for at least one $\mathbf{p}$.
We say that a p-merging function is \emph{admissible} if it is not strictly dominated by any p-merging function.
Analogously, we can define admissibility within smaller classes of p-merging functions, such as the class of symmetric p-merging functions.
Finally, a p-merging function $F$ is said to be \emph{precise} if
\[
  \sup_{\mathbf P\in \mathcal P_Q^K} Q(F (\mathbf P)\le \epsilon) = \epsilon \text{ for all }\epsilon \in (0,1).
\]
In other words, $\epsilon$ by $\epsilon$, $F$ attains the largest possible probability allowed for $F(\mathbf P)$ to be a p-value.
Precise p-merging functions are the main object studied by \citet{Vovk/Wang:2020Biometrika},
where p-values are combined via averaging.

We collect some basic properties of admissible p-merging functions,
which will be useful in our analysis later.
In particular,
an admissible p-merging function is always precise and lower semi-continuous,
the limit of p-merging functions is again a p-merging function,
and any p-merging function is dominated by an admissible p-merging function.
The proofs of these results are put in Supplemental Article, Section~\ref{app:a2}.

\begin{proposition}\label{prop:precise-p}
  An admissible p-merging function is always precise.
\end{proposition}

For an increasing Borel function $F:[0,\infty)^K\to[0,\infty)$,
its lower semicontinuous version $F'$ is given by
\begin{equation}\label{eq:lsc}
  F'(\mathbf{p}):
  =
  \lim_{\lambda\uparrow 1}
  F(\lambda \mathbf{p} ),
  \quad
  \mathbf{p}\in[0,\infty)^K.
\end{equation} 
Clearly, $F'$ is increasing, lower semicontinuous, and $F'\le F$.
Moreover, we define the \emph{zero-one adjusted} version $\widetilde F$ of $F$ by
\begin{equation}\label{eq:zero}
  \widetilde F(\mathbf{p})
  :=
  \begin{cases}
    F(\mathbf{p}\wedge\mathbf{1}) \wedge 1 & \text{if $\mathbf{p}\in(0,\infty)^K$}\\
    0 & \text{otherwise}.
  \end{cases}
\end{equation}

\begin{proposition}\label{prop:lsc}
  If $F$ is a p-merging function,
  then both its lower semicontinuous version $F'$ in \eqref{eq:lsc}
  and its zero-one adjusted version $\widetilde F$ in \eqref{eq:zero}
  are p-merging functions.
  In particular,
  an admissible p-merging function is always lower semicontinuous,
  takes value $0$ on $[0,\infty)^K\setminus(0,\infty)^K$,
  and satisfies $F(\mathbf{p})=F(\mathbf{p}\wedge\mathbf{1})\wedge1$
  for all $\mathbf{p}\in[0,\infty)^K$.
\end{proposition} 

The next result addresses the closure property of the set of p-merging functions.
 
\begin{proposition}\label{prop:limit}
  The point-wise limit of a sequence of p-merging functions is a p-merging function. 
\end{proposition}
Combining the above results, we are able to show that any p-merging function is dominated by an admissible one.
\begin{proposition}\label{prop:dominated}
  Any p-merging function is dominated by an admissible p-merging function.
\end{proposition}
\begin{remark}\label{rem:symmetry}
  Using the same proof as for Proposition \ref{prop:dominated},
  we can show that any symmetric p-merging function is dominated by a p-merging function that is admissible among symmetric p-merging functions.
  The same holds true if ``symmetric'' is replaced by ``homogeneous'' or ``symmetric and homogeneous''.
\end{remark}

\section{Some classes of p-merging functions}
\label{sec:3}

Similarly to \citet{Vovk/Wang:E},
we pay special attention to two families of  p-merging functions:
the family based on order statistics introduced by \citet{Ruger:1978}, henceforth the \emph{O-family}, where ``O'' stands for ``order'',
and the new family introduced by  \citet{Vovk/Wang:2020Biometrika}, henceforth the \emph{M-family}, where ``M'' stands for ``mean''.
The O-family is parameterized by $k\in\{1,\dots,K\}$,
and its $k$th element is the function
(shown by \citet{Ruger:1978} to be a p-merging function)
\begin{equation}\label{eq:Ruger}
  G_{k,K}: (p_1,\dots,p_K)
  \mapsto
  \frac{K}{k} 
  p_{(k)}
  \wedge
  1,
\end{equation}
where $p_{(k)}$ is the $k$th order statistic of $p_1,\dots,p_K$.
The $M$-family is parameterized by $r\in[-\infty,\infty]$,
and its element with index $r$ has the form
\begin{equation}\label{eq:merge1}
  F_{r,K}: (p_1,\dots,p_K)
  \mapsto
  b_{r,K}M_{r,K}(p_1,\dots,p_K)\wedge1,
\end{equation} 
where
 \[
  M_{r,K}(p_1,\dots,p_K)
  :=
  \left(
    \frac{p_1^r + \dots + p_K^r }{K}
  \right)^{1/r}
 \]
and $b_{r,K}\ge1$ is a suitable constant making $F_{r,K}$ a precise merging function (its value will be specified in Section \ref{sec:brk}).
The average $M_{r,K}$ is also defined for $r\in\{0,\infty,-\infty\}$ as the limiting cases of \eqref{eq:merge1},
which correspond to the geometric average, the maximum, and the minimum, respectively.
All members of both families are precise p-merging functions.

The initial and final elements of the M- and O-families coincide:
the initial element is the Bonferroni p-merging function
\begin{equation}\label{eq:Bonferroni}
  G_{1,K} = F_{-\infty,K}: (p_1,\dots,p_K)
  \mapsto
  K\min(p_1,\dots,p_K)
  \wedge
  1,
\end{equation}
and the final element is the maximum p-merging function
\[
  G_{K,K} = F_{\infty,K}:
  (p_1,\dots,p_K)
  \mapsto
  \max(p_1,\dots,p_K).
\]
While the Bonferroni p-merging function is constantly used in practice,
the maximum p-merging function is obviously useless.
For the intermediate values of $k$, $1<k<K$,
$G_{k,K}$ appear to be an arbitrary choice.
Another prominent element of the M-family is the multiple $F_{-1,K}$ of the harmonic mean $M_{-1,K}$
\citep{Good:1958,Wilson:2019},
variations of which have been used in bioinformatics and other sciences.
More generally, choosing a good value of $r$ is discussed in detail in Section~6 of \citet{Vovk/Wang:2020Biometrika}.

Another important p-merging function  is that of \citet{Hommel:1983}, given by
\[
  H_K := \left(\sum_{k=1}^K \frac 1k \right) \bigwedge_{k=1}^K G_{k,K}.
\]
To some degree it solves the problem of choosing $k$.
The Hommel function $H_K$ (or $H_K\wedge 1$, since a truncation at $1$ is trivial)
is a precise p-merging function,
and it equals a constant $\ell_K := \sum_{k=1}^K k^{-1}$ times the function
\[
  S_K := \bigwedge_{k=1}^K G_{k,K} = \frac{1}{\ell_K}H_K,
\]
used by \citet{Simes:1986}.
The Simes function $S_K$ is a valid merging function for independent p-variables
(or under some other dependence assumptions, as in, e.g., \citet{sarkar1998some}).

Admissibility of the above p-merging functions will be studied in Sections \ref{sec:o} and \ref{sec:M}.
In the case of inadmissibility,
a function can be strictly improved to another p-merging function without losing validity (Proposition \ref{prop:dominated}).
We will explicitly construct new merging functions that strictly dominate the existing ones.
In one of the two extreme special cases,
the Bonferroni p-merging function is shown to be admissible in \citet[Proposition 6.1]{Vovk/Wang:E}.
On the contrary, the maximum p-merging function $G_{K,K}$ ($F_{\infty,K}$) is not admissible for any $K\ge 2$,
since it is strictly dominated by, for instance, $(p_1,\dots,p_K)\mapsto p_1$.
Nevertheless, after a trivial modification, $G_{K,K}$ is admissible within the class of symmetric p-merging functions;
see Theorem \ref{pr:o1} in Section \ref{sec:o}.

Next, we present a result showing that the Simes function $S_K$ has a very special role in the context of p-merging,
as it is a lower bound for any symmetric p-merging functions.
Therefore, $S_K(p_1,\dots,p_K)$ can be seen as the best achievable p-value obtained via symmetric merging of $p_1,\dots,p_K$,
although the function $S_K$ itself is not a valid p-merging function.

\begin{theorem}\label{thm:Simes}
  The Simes function $S_K$ is the minimum of all symmetric p-merging functions.
\end{theorem}

\begin{proof}
  Take any symmetric p-merging function $F$ and $\mathbf{p} = (p_1,\dots,p_K)$.
  Let $\alpha:= S_K (\mathbf{p}) /K$.
  Note that $K \alpha \le 1$ and $p_{(k)}\ge k \alpha$ for each $k=1,\dots,K$.
  By the symmetry of $F$,
  \[
    F(\mathbf p)
    =
    F(p_{(1)},\dots,p_{(K)})
    \ge
    F(\alpha, 2 \alpha, \dots ,K\alpha )
    =:
    \beta.
  \]  
Let $\Pi$ be the set of all permutations of the vector $(\alpha,2\alpha,\dots,K\alpha)$, and  $\mu$ be the discrete uniform distribution over $\Pi$. 
Take a random vector $(P_1,\dots,P_K)$ following the distribution $K\alpha \mu+(1-K\alpha) \delta_{(1,\dots,1)}$. 
For each $k$, the distribution of $P_k$ is given by $\sum_{k=1}^{K} \alpha \delta_{k\alpha} + (1-K\alpha) \delta_1$, and hence $P_k$ is a p-variable.  
Since $F$ is a p-merging function, we have 
\[
  \beta \ge \p(F(P_1,\dots,P_K) \le \beta)
  \ge
  \p((P_1,\dots,P_K)\in \Pi)
  =
  K \alpha.
\] 
This implies $F( \mathbf p) \ge K\alpha =  S_K(\mathbf p)$,
and hence $S_K$ dominates all symmetric p-merging functions.
Finally, the statement of $S_K$ as a minimum follows from $S_K=\bigwedge_{k=1}^K G_{k,K}$, noting that each $G_{k,K}$ is a symmetric p-merging function.
\end{proof}

In the main part of the paper we will focus on the case $K>2$.
The case $K=2$ is very different but simpler;
it is treated separately in Supplemental Article, Section~\ref{app:K2}.
In this case, the Bonferroni p-merging function 
$(p_1,p_2)\mapsto \min(2p_1,2p_2,1)$
is the only admissible symmetric p-merging function.

\section{Duality and p-to-e merging}
\label{sec:e}

As a prelude to studying the problem of merging p-values into a p-value,
we will discuss the notion of e-values and the much easier problem of merging p-values into an e-value
\citep[Appendix G]{Vovk/Wang:E}.
As already mentioned, in this paper we are only interested in e-values as a technical tool.

An \emph{e-variable} is a non-negative extended random variable $E:\Omega\to[0,\infty]$ with $\E^Q[E]\le 1$.
A \emph{calibrator} (or, more fully, ``p-to-e calibrator'') is a decreasing function $f:[0,\infty)\to[0,\infty]$
satisfying $ f=0$ on $(1,\infty)$ and $\int_0^1 f(x)\d x\le 1$.
A calibrator transforms any p-variable to an e-variable.
It is \emph{admissible} if it is upper semicontinuous, $f(0)=\infty$, and $\int_0^1 f(x) \d x = 1$
(equivalently \citep[Propositions~2.1 and~2.2]{Vovk/Wang:E}, it is not strictly dominated, in a natural sense,
by any other calibrator).

A function $F:[0,\infty)^K\to[0,\infty]$ is a \emph{p-to-e merging function}
if $F(P_1,\dots,P_K)$ is an e-variable for any p-variables $P_1,\dots,P_K$.
A p-to-e merging function $F$ \emph{dominates} a p-to-e-merging function $G$ if $F\ge G$,
and the domination is \emph{strict} if $F\ne G$;
$F$ is \emph{admissible} if it is not strictly dominated by any other p-to-e merging function.

Below, $\Delta_K$ is the standard $K$-simplex, that is,
$\Delta_K := \{(\lambda_1,\dots,\lambda_K)\in[0,1]^K:\lambda_1+\dots+\lambda_K=1\}$,
and we always write $\mathbf p:=(p_1,\dots,p_K)$.

It is clear that a convex mixture of e-variables is an e-variable.
In this sense  a convex mixture is an ``e-merging function'';
and in the symmetric case, the arithmetic average essentially dominates any other e-merging function
\citep[Proposition 3.1]{Vovk/Wang:E}. 
Therefore, for any calibrators $f_1,\dots,f_K$ and any $(\lambda_1,\dots,\lambda_K)\in\Delta_K$,
the function
\begin{equation}\label{eq:general}
  G(\mathbf{p})
  :=
  \lambda_1 f_1(p_1)
  + \dots +
  \lambda_K f_K(p_K)
\end{equation}
is a p-to-e merging function.

The following corollary of a duality theorem for optimal transport
says that this procedure of p-to-e merging is general.

\begin{proposition}\label{pr:dual}
  For any calibrators $f_1,\dots,f_K$ and any $(\lambda_1,\dots,\lambda_K)\in\Delta_K$,
  \eqref{eq:general} is a p-to-e merging function.
  Conversely, any p-to-e merging function $F$ is dominated by the p-to-e merging function \eqref{eq:general}
  for some calibrators $f_1,\dots,f_K$ and some $(\lambda_1,\dots,\lambda_K)\in\Delta_K$.
\end{proposition}

\begin{proof}
  The non-trivial statement is the second one.
  Let $F$ be a p-to-e merging function.
  Denote by $\mathcal F$ the set of decreasing real functions on $[0,\infty)$,
  and define the operator $\bigoplus$ as
  \[
    \left(
      \bigoplus_{k=1}^K g_k
    \right)
    (x_1,\dots,x_K)
    :=
    \sum_{k=1}^K g_k(x_k),
    \quad
    (g_1,\dots,g_K)\in\mathcal F^K,
    \quad(x_1,\dots,x_K)\in[0,\infty)^K.
  \]
  Using a classic duality theorem (see, e.g., \citet[Theorem 2.3]{R13}), we have
  \begin{align}
    \min
    \left\{
      \sum_{k=1}^K \int_0^1 g_k(x)\d x: (g_1,\dots,g_K)\in\mathcal F^K,~\bigoplus_{k=1}^K g_k \ge F
    \right\}
    =
    \sup_{\mathbf P\in\mathcal P_Q^K}
    \E^Q[F(\mathbf P)]
    \le
    1.\label{eq:duality-eq}
  \end{align}
  Indeed, part (a) of Theorem 2.3 in \citet{R13} gives the equality with $\inf$ in place of $\min$
  and with $\mathbf{P}$ ranging over the probability measures on $[0,1]^K$ with the uniform marginals.
  Part (d) of that theorem gives $\inf$, and it remains to notice that every p-variable $P$
  is dominated, in the sense of $U\le P$,
  by a random variable $U$ (perhaps on an extended probability space) uniformly distributed on $[0,1]$
  (see, e.g., \citet[Theorem 2.3]{R09}).

  Choose $g_1,\dots,g_K$ at which the minimum is attained in \eqref{eq:duality-eq}.
  It is clear that we can define calibrators $f_1,\dots,f_K$ and $(\lambda_1,\dots,\lambda_K)\in\Delta_K$
  in such a way that $\lambda_k f_k\ge g_k$ for all $k$,
  e.g.,
  $
    \lambda_k := \int_0^1 g_k(x)\d x / \sum_{i=1}^K \int_0^1 g_i(x)\d x
  $
  and
  $
    f_k := g_k / \int_0^1 g_k(x)\d x
  $
  (the simple cases where one or both of the denominators vanish should be considered separately).
  With this choice $F$ will be dominated by the p-to-e merging function \eqref{eq:general}.
\end{proof}

By the Markov inequality, $1/F$ is a p-merging function for any p-to-e merging function $F$.
Such a ``naive procedure'' for merging p-values is generally not admissible.
Nevertheless, for a fixed $\epsilon\in(0,1)$ and any admissible p-merging function $G$,
we can find a p-to-e merging function $F$ such that $G\le\epsilon \Leftrightarrow F\ge 1/\epsilon$.
These statements are discussed and put in a more general context
in Section~\ref{app:naive} of Supplemental Article.

\section{Rejection regions of admissible p-merging functions}
 
A p-merging function can be characterized by its rejection regions.
The \emph{rejection region} of a p-merging function $F$ at level $\epsilon>0$ is defined as
\begin{equation}\label{eq:region}
  R_\epsilon(F) := \left\{\mathbf p \in [0,\infty)^K: F(\mathbf p)\le \epsilon \right\}.
\end{equation}
If $F$ is homogeneous, then $R_\epsilon(F)$, $\epsilon\in(0,1)$,
takes the form $R_\epsilon(F)=\epsilon A$ for some $A\subseteq [0,\infty)^K$.

Conversely, any increasing collection of Borel lower sets
$\{R_\epsilon \subseteq [0,\infty)^K: \epsilon \in (0,1)\}$
determines an increasing Borel function $F: [0,\infty)^K\to[0,1]$ by the equation
\begin{equation}\label{eq:inf}
  F(\mathbf p) = \inf\{\epsilon\in(0,1): \mathbf p\in R_{\epsilon}\},
\end{equation}
with the convention $\inf\varnothing = 1$.
It is immediate that $F$ is a p-merging function if and only if
$Q(\mathbf P \in R_\epsilon) \le \epsilon$ for all $\epsilon\in (0,1)$ and $\mathbf P\in \mathcal P_Q^K$.

The main result in this section is a representation of rejection regions of admissible p-merging functions.
It turns out that calibrating p-values into e-values as in Proposition \ref{pr:dual} is a useful technical tool for studying such rejection regions.

\begin{theorem}\label{th:e}
  For any admissible homogeneous p-merging function $F$,
  there exist $(\lambda_1,\dots,\lambda_K)\in\Delta_K$ and admissible calibrators $f_1,\dots,f_K$
  such that
  \begin{equation}\label{eq:calibrator}
    R_\epsilon(F)
    =
    \epsilon
    \left\{
      \mathbf p \in [0,\infty)^K:
      \sum_{k=1}^K \lambda_k f_k(p_k) \ge 1
    \right\}
    \qquad
    \text{for each $\epsilon \in (0,1)$}.
  \end{equation}
  Conversely, for any $(\lambda_1,\dots,\lambda_K)\in\Delta_K$ and calibrators $f_1,\dots,f_K$,
  \eqref{eq:calibrator} determines a homogeneous p-merging function.
\end{theorem}

\begin{proof}
Fix an arbitrary $\epsilon \in (0,1)$. 
Note that the set $R_\epsilon(F)$ is a lower set, and it is closed due to Proposition \ref{prop:lsc}.  
We use the same notation as in the proof of Proposition \ref{pr:dual}.
Using the duality relation \eqref{eq:duality-eq},
\begin{align*}
 & \min_{(g_1,\dots,g_K)\in\mathcal F^K}
  \left\{
    \sum_{k=1}^K \int_0^1 g_k(x)\d x:  \bigoplus_{k=1}^K g_k \ge \id_{R_\epsilon(F)}
  \right\}
   =
  \max_{\mathbf P\in\mathcal P_Q^K}
  Q(\mathbf P\in R_\epsilon(F))
  =
  \epsilon,
\end{align*}
where the last equality holds because $F$ is precise (Proposition \ref{prop:precise-p}).
Take $(g^\epsilon_1,\dots,g^\epsilon_K)\in \mathcal F^K$ such that
$\bigoplus_{k=1}^K g^\epsilon_k \ge \id_{R_\epsilon(F)}$
and
$\sum_{k=1}^K \int_0^1 g^\epsilon_k(x)\d x = \epsilon$.
Obviously we can choose each $g^\epsilon_k$ to be non-negative and left-continuous.
Using the fact that $R_\epsilon(F)$ is a closed lower set, we have
\begin{align}\label{eq:jm}
  \max_{\mathbf P\in \mathcal P_Q^K} Q( \mathbf P\in R_\epsilon (F)) = \epsilon
  ~~\Longrightarrow~~
  \max_{\mathbf P\in \mathcal P_Q^K} Q(\epsilon \mathbf P\in  R_\epsilon ( F)  ) = 1.
\end{align}
Therefore, using duality again,
\[
  \min_{(g_1,\dots,g_K)\in\mathcal F^K}
  \left\{
    \sum_{k=1}^K \frac1\epsilon \int_0^\epsilon g_k(x)\d x:\bigoplus_{k=1}^K g_k \ge \id_{R_\epsilon(F)}
  \right\}
  =
  1,
\]
implying $\sum_{k=1}^K \int_0^\epsilon g^\epsilon_k(x) \d x\ge\epsilon$.
As $g_k\ge0$ for each $k$ and $\sum_{k=1}^K \int_0^1 g^\epsilon_k(x)\d x = \epsilon$,
we know $g^\epsilon_k(x)=0$ for $x>\epsilon$.

Define the set $A_\epsilon := \{\mathbf p \in [0,\infty)^K:  \sum_{k=1}^K g^\epsilon_k(p_k) \ge  1\}$.
Since $\bigoplus_{k=1}^K g^\epsilon_k \ge \id_{R_\epsilon (F)}$,
we have $R_\epsilon ( F)  \subseteq  A_\epsilon$.
Note that $A_\epsilon $ is a closed lower set.
By Markov's inequality, 
\[
  \sup_{\mathbf P\in \mathcal P_Q^K}
  Q \left( \bigoplus_{k=1}^K g^\epsilon_k(\mathbf P) \ge 1 \right)
  \le
  \sup_{P\in \mathcal P_Q}
  \sum_{k=1}^K\E^Q[g^\epsilon_k(P)]
  =
  \epsilon. 
\]
Hence, we can define a function $F':[0,\infty)^K\to \R$ via $R_\epsilon ( F')  =  A_\epsilon $
and  $R_\delta ( F') = \delta\epsilon^{-1} A_\epsilon  $ 
for all $\delta \in(0,1)$. 
By the above properties of $A_\epsilon$, $F'$ is a valid homogeneous p-merging function.
Moreover, $F'$ dominates $F$ since $R_\delta(F) \subseteq A_\delta  $ for all $\delta\in (0,1)$ due to homogeneity of $F$.
The admissibility of $F$ now gives  $F=F'$, and thus
\[
  R_\epsilon(F) = A_\epsilon =\epsilon \left\{ 
  \mathbf p \in [0,\infty)^K:  \sum_{k=1}^K g^\epsilon_k(\epsilon p_k) \ge  1  \right\}
  \qquad
  \text{for each $\epsilon \in (0,1)$}.
\]
Note that $A:=\epsilon^{-1} R_{\epsilon}(F) = \epsilon^{-1} A_\epsilon$ does not depend on $\epsilon\in (0,1)$.
For a fixed $\epsilon\in(0,1)$,
let $\lambda_k:= \epsilon^{-1} \int_0^\epsilon g^\epsilon(x) \d x$
and
$f_k: (0,\infty)\to \R$, $x\mapsto  g^\epsilon_k(\epsilon x)/\lambda_k $ for each $k=1,\dots,K$
(if $\lambda_k=0$, then let $f_k:=1$),
and further set $f_k(0)=\infty$.
It is clear that for each $k$ with $\lambda_k\ne 0$,
\[\int_0^1 f_k(x) \d x  = \frac{ \int_0^1 \epsilon g^\epsilon_k(\epsilon x) \d x}{\int_0^1 g^\epsilon_k(  x) \d x} = \frac{ \int_0^  \epsilon g^\epsilon_k(  x) \d x}{\int_0^1 g^\epsilon_k(  x) \d x} =1.\]
The conditions that $f_k$ is decreasing and left-continuous, $\int_0^1 f_k(x) \d x =1$, $f_k(0)=\infty$, and $f_k(x)=0$ for $x>1$
imply that $f_k$ is an admissible calibrator.
Therefore, \eqref{eq:calibrator} holds. 

For the last statement, let 
$f_1,\dots,f_K$ be calibrators and $(\lambda_1,\dots,\lambda_K)\in \Delta_K$. 
Note that for each $\epsilon \in (0,1)$, \eqref{eq:calibrator} gives  
\[
  R_{\epsilon}(F)
  =
  \left\{
    \mathbf p \in [0,\infty)^K:
    \sum_{k=1}^K \lambda_k f_k \left(\frac{p_k}{\epsilon }\right) \ge 1
  \right\},
\]
and since $f(x)=0$ for $x>1$, it holds
\[
  \sum_{k=1}^K \lambda_k \int_0^1  f_k \left(\frac{x}{\epsilon }\right) \d x
  =
  \sum_{k=1}^K \lambda_k \int_0^{1/\epsilon} f_k(y) \d y
  =
  \epsilon \sum_{k=1}^K \lambda_k
  =
  \epsilon.
\]
Hence, Markov's inequality gives
\[
  \sup_{\mathbf P\in \mathcal P_Q^K} Q \left(\mathbf P \in R_{\epsilon}(F)\right)
  =
  \sup_{\mathbf P\in \mathcal P_Q^K} Q \left(\bigoplus_{k=1}^K \lambda_k f_k\left( \frac{ \mathbf P}{\epsilon}\right) \ge 1 \right)
  \le
  \epsilon.
\]
Thus, \eqref{eq:calibrator} determines a homogeneous p-merging function.
\end{proof}

As an immediate consequence of \eqref{eq:calibrator},
for an admissible homogeneous p-merging function $F$ and $\epsilon\in(0,1)$,
$F(p_1,\dots,p_K)\le\epsilon$ if and only if $F(p_1\wedge\epsilon,\dots,p_K\wedge\epsilon)\le\epsilon$.
Therefore, for a rejection region of $F$ at level $\epsilon$, there is no dependence on  input p-values larger than $\epsilon$.

If the homogeneous p-merging function $F$ is symmetric,
then $f_1,\dots,f_K$, as well as $\lambda_1,\dots,\lambda_K$, in Theorem \ref{th:e} can be chosen identical.

\begin{theorem}\label{pr:e}
  For any $F$ that is admissible within the family of homogeneous symmetric p-merging functions,
  there exists an admissible calibrator $f$ such that
  \begin{equation}\label{eq:calibrator2}
    R_\epsilon (F)
    =
    \epsilon
    \left\{\mathbf p \in [0,\infty)^K: \frac 1 K \sum_{k=1}^K f(p_k) \ge 1 \right\}
    \qquad\text{for each $\epsilon \in (0,1)$.}
  \end{equation}  
  Conversely, for any  calibrator $f$, \eqref{eq:calibrator2} determines a homogeneous symmetric p-merging function.
\end{theorem}

\begin{proof}
  The proof is similar to that of Theorem \ref{th:e} and we only mention the differences.
  For the first statement, it suffices to notice two facts.
  First, if $R_{\epsilon}$ is symmetric, then $g^\epsilon_1,\dots,g^\epsilon_K$ in the proof of Theorem \ref{th:e} can be chosen as identical;
  for instance, one can choose the average of them (see, e.g., Proposition 2.5 of \citet{R13}).
  Second, the symmetry of $R_\epsilon (F)$ guarantees that $F'$ in the proof of Theorem \ref{th:e} is symmetric,
  and hence it is sufficient to require the admissibility of $F$  within homogeneous symmetric p-merging functions in this proposition.
  The last statement in the proposition follows from Theorem \ref{th:e}
  by noting that \eqref{eq:calibrator2} defines a symmetric rejection region.
\end{proof}

\begin{remark}
  In the converse statements of Theorems \ref{th:e} and \ref{pr:e},
  a p-merging function induced by admissible calibrators is not necessarily admissible (see Example \ref{ex:ex-th:e3}),
  although admissibility is indispensable in the proof of the forward direction.
  Using \eqref{eq:jm} and a compactness argument,
  a necessary and sufficient condition for a calibrator $f$ to induce a precise p-merging function
  (a weaker requirement than admissibility)
  via \eqref{eq:calibrator2} is
  \begin{equation}\label{eq:existp}
    \p\left( \frac 1K  \sum_{k=1}^K f(P_k) = 1\right)=1
    \quad
    \text{for some $P_1,\dots,P_K\sim \mathrm U[0,1]$}.
  \end{equation}
  Condition \eqref{eq:existp} may be difficult to check for a given $f$ in general.
  For a convex $f$, as shown by \citet[Theorem 2.4]{WW11},
  \eqref{eq:existp} holds if and only if $f \le K$ on $(0,1]$.  
  Sufficient conditions for admissibility will be studied in Section \ref{sec:admissibility} below.
  Similarly to \eqref{eq:existp},
  an equivalent condition for the p-merging function $F$ in \eqref{eq:calibrator} to be precise is 
  \begin{equation}\label{eq:existp-JM}
    \p\left(\sum_{k=1}^K \lambda_k f_k(P_k) = 1\right)=1
    \quad
    \text{for some $P_1,\dots,P_K\sim \mathrm U[0,1]$}.
  \end{equation}
  Using the terminology of \citet{WW16},
  \eqref{eq:existp-JM} means that the distributions of $\lambda_k f_k(P_k)$, $k=1,\dots,K$, are jointly mixable.
  Assuming convexity of the calibrators, \eqref{eq:existp-JM} has a similar equivalent condition \citep[Theorem 3.2]{WW16},
  and this result is essential to the proof of Theorem \ref{th:admissible} below.
\end{remark}

For a decreasing function $f:[0,\infty)\to [0,\infty]$ and a p-merging function $F$ taking values in $[0,1]$,
we say that $f$ \emph{induces} $F$ if \eqref{eq:calibrator2} holds;
similarly, we say that $\lambda_1,\dots,\lambda_K$ and $f_1,\dots,f_K$ \emph{induce} $F$ if \eqref{eq:calibrator} holds.
Theorems \ref{th:e} and \ref{pr:e} imply that admissible p-merging functions are induced by some  admissible calibrators.
Generally, the calibrator inducing a given p-merging function may not be unique.
In the following examples, p-merging functions are induced by calibrators, although these p-merging functions are not necessarily admissible.

\begin{example}\label{ex:ex-th:e}
  The p-merging function $F:= G_{k,K}$, $k\in \{1,\dots,K\}$,
  is induced by  the calibrator $(K/k) \id_{[0,k/K]}$.
\end{example}

\begin{example}\label{ex:ex-th:e3}
  In the case $K=2$,
  the p-merging function
  \[
    F:\mathbf p\mapsto
    2M_{1,K}(\mathbf p\wedge\mathbf{1})\wedge\id_{\{\min\mathbf p > 0\}}
    =
    2M_{1,K}(\mathbf p)\wedge\id_{\{\min\mathbf p > 0\}}
  \]
  is induced by the admissible calibrator $f: x\mapsto(2-2x)_+$ on $(0,\infty)$ and $f(0)=\infty$.
  The function $F$ is the zero-one adjusted version (see Proposition~\ref{prop:lsc}) of the arithmetic merging function,
  and it is dominated by the Bonferroni merging function.
  Hence, $F$ is not admissible.
\end{example}

\begin{example}\label{ex:ex-th:e2}
  One may also generate p-merging functions from \eqref{eq:calibrator2}
  where $f$ is not a calibrator. 
  For the arithmetic merging function $F:= 2M_{1,K}$,
  equality \eqref{eq:calibrator2} holds by choosing the  function $f: x\mapsto 2-2x$.
  Note that $f$ is not a calibrator and it takes negative values for $x>1$.
  For another example, we take $F:=F_{r,K}$ for $r<0$ in \eqref{eq:merge1}.
  Rewriting the equation $F(\epsilon \mathbf p)\le \epsilon$ as $b_{r,K}(\frac 1K \sum_{k=1}^K p_k^r)^{1/r}\le 1$,
  we see that
  \[
    R_\epsilon ( F) = \epsilon \left\{\mathbf p \in [0,\infty)^K:  \frac 1 K \sum_{k=1}^K b_{r,K}^{r}  p_k^{r} \ge 1 \right\},
  \]
  thus satisfying \eqref{eq:calibrator2} with $f:x\mapsto b_{r,K}^r x^r$.
  Such $f$ is generally not a calibrator (not even integrable for $r\le -1$),
  although it induces a precise p-merging function for a properly specified value of $b_{r,K}$
  in Section \ref{sec:M}.
\end{example}

The requirement $f(0)=\infty$ for an admissible calibrator $f$
implies that the combined test \eqref{eq:calibrator2} gives a rejection as soon as one of the input p-values is $0$,
which is obviously necessary for admissibility (Proposition \ref{prop:lsc}).
Although many examples in the M- and O-families, in particular $F_{r,K}$ for $r>0$ and $G_{k,K}$ for $k>1$, do not satisfy this,
we can make the zero-one adjustment \eqref{eq:zero},
which does not affect the validity of the p-merging function by Proposition \ref{prop:lsc}.
In the sequel, a calibrator will be specified by its values on $(0,1]$,
as $f=0$ on $(1,\infty)$ for any calibrator $f$,
and $f(0)$ should be clear in each specific example (in particular $f(0)=\infty$ if $f$ is admissible).
The value $f(0)$ does not affect the p-merging function determined by \eqref{eq:calibrator2}
as long as $f(0)\ge K$.

\section{Conditions for admissibility}
\label{sec:admissibility}

We have seen that p-merging functions induced by admissible calibrators via Theorems \ref{th:e} and \ref{pr:e} are not necessarily admissible
(Example~\ref{ex:ex-th:e3}).
In this section, we study sufficient conditions for admissibility based on calibrators.
First, Theorems \ref{th:e} and \ref{pr:e} lead to an immediate criterion for checking the admissibility of an induced p-merging function
(proved in Section~\ref{app:a5} of Supplemental Article).

\begin{proposition}\label{pr:g}
  Suppose that $F$ is a p-merging function taking values in $[0,1]$ and satisfying \eqref{eq:calibrator2}
  for a decreasing function $f$.
  The following statements hold:
  \begin{enumerate}[(i)]
  \item
    $F$ is admissible among symmetric p-merging functions if and only if there is no calibrator $g$ such that
    \begin{equation}\label{eq:rejrelation}
      \left\{\mathbf p \in [0,\infty)^K: \frac 1K \sum_{k=1}^K f(p_k) \ge 1 \right\}
      \subsetneq
      \left\{ \mathbf p \in [0,\infty)^K:  \frac 1K \sum_{k=1}^K g(p_k) \ge 1 \right\}.
    \end{equation}
  \item $F$ is admissible if and only if
    there are no $(\lambda_1,\dots,\lambda_K)\in\Delta_K$ and calibrators $g_1,\dots,g_K$ such that
    \begin{equation}\label{eq:rejrelation2}
      \left\{
        \mathbf{p}\in[0,\infty)^K: \frac{1}{K}\sum_{k=1}^{K} f(p_k) \ge 1
      \right\}
      \subsetneq
      \left\{
        \mathbf{p}\in[0,\infty)^K: \sum_{k=1}^{K} \lambda_k g_k(p_k) \ge 1
      \right\}.
    \end{equation}    
  \end{enumerate}
\end{proposition}

Note that \eqref{eq:rejrelation} does not imply $g\ge f$, making the existence of $g$ often complicated to analyze.
Proposition \ref{pr:g} implies, in particular, that for any calibrator $f$, $f\le K$ on $(0,1]$ is a necessary condition for the induced p-merging function to be admissible,
because otherwise the function $g:x\mapsto f(cx)\wedge K$ where $c:=\int_0^1 f(x)\wedge K\d x<1$ would induce a p-merging function strictly dominating $F$.
On the other hand, if $f(1)>0$, then the calibrator $g:=(f-f(1))/(1-f(1))\id_{[0,1]}$
induces the same p-merging function $F$.
Hence, it suffices to consider $f$ with $f\le K$ on $(0,1]$ and $f(1)=0$.

The main result of this section gives a sufficient condition for the admissibility of the corresponding p-merging function.
For a calibrator $f$, we define another calibrator $g:[0,\infty)\to [0,\infty]$, for some $\eta\in [0,1/K]$, via
\begin{equation}\label{eq:def-g}
  g: x\mapsto f\left(\frac{x-\eta}{1-K\eta}\right)\id_{\{x\in (\eta,1-(K-1)\eta]\}} + K\id_{\{x\in [0,\eta]\}}.
\end{equation}
It is straightforward to verify $\int_0^1 g(x)\d x \le 1$, and $g$ defined via \eqref{eq:def-g} is a calibrator. 

\begin{theorem}\label{th:admissible}
  Suppose that an admissible calibrator $f$ is strictly convex or strictly concave on $(0,1]$, $f(0+)\in (K/(K-1),K]$, and $f(1)=0$.
  The p-merging function induced by $f$, or $g$ in \eqref{eq:def-g} for any $\eta\in [0,1/K]$, is admissible.
\end{theorem}

\begin{proof}
  We will prove the statement on $f$, and the statement on $g$ would then follow from Lemma \ref{lem:f-trans}
  in Supplemental Article, Section~\ref{app:a5},
  which says that if $f$ induces an admissible p-merging function, then so does $g$ in \eqref{eq:def-g}.
  We only show the case where $f$ is strictly convex, as the case of a strictly concave $f$ follows from a symmetric argument;
  we remark that $f(0+)\le K$  for a convex $f$ and $f(0+)> K/(K-1)$ for a concave $f$ play the same role in the proof.

  Suppose for the purpose of contradiction that there exists a p-merging function $G$ which strictly dominates $F$,
  that is, there exist $\mathbf{p} = (p_1, \dots,p_K) \in [0,1]^K$ and $\alpha \in (0,1)$ such that $G(\mathbf{p}) < \alpha < F(\mathbf{p}) < 1$. 
  Set $a := \lim_{t\downarrow 0} f(t) \le K$.
  Clearly, $a>2$ since no strictly convex function on $[0,1]$ bounded by $2$ integrates to $1$.
  Hence, it suffices to assume $K\ge 3$.

  Note that $f$ is continuous  and strictly decreasing on $(0,1)$.
  Let $f^{-1}:(0,a)\mapsto(0,1)$ be the inverse function of $f$, which is strictly decreasing and strictly convex.
  Let $U$ be a uniform random variable on $[0,1]$,
  and let $h$ be the density function $f(U)$. Note that $h$ is a strictly decreasing density function.
  Since  $ \mathbf{p} \notin R_\alpha (F)$, we have $ \sum_{k=1}^{K} f(p_k / \alpha) < K$.
  Denote by $y_k := f(p_k / \alpha)$, $k =1,\dots,K$.
  Note that $y_1+\cdots +y_K < K $ and  $y_k < a$ for each $k$.
  Take a small constant
  \[\epsilon: = \frac{1}{4} \min\left\{\bigwedge_{k=1}^K ( a - y_k ),~ a-2 ,~ 1-\frac1{K}{\sum_{k=1}^K y_k}\right\}>0.\]
  For each $k=1,\dots,K$, $h$ is strictly decreasing in $[y_k + \epsilon, y_k + 2\epsilon]$ since $y_k + 2\epsilon \le a-2\epsilon$.
  Define another density function
  $v_k := (h - h(y_k + 2\epsilon)) \id_{[y_k + \epsilon, y_k + 2\epsilon]}$
  with its mass
  $m_k := \int_{y_k + \epsilon}^{y_k + 2\epsilon} v_k(t)\d t > 0$ and its mean $\mu(v_k)$ smaller than $y_k + 2\epsilon$. 

  Write $\beta: = 1 - \frac{1}{K} (\mu(v_1) + \cdots + \mu(v_K))$.
  Since $ \mu(v_1) + \cdots + \mu(v_K) < y_1 + \cdots + y_K + 2K \epsilon < K$, we have $\beta>0$. 
  Take another small constant
  \[
    \theta
    :=
    \min
    \left\{
      \bigwedge_{k=1}^K 
      \frac{m_k \beta}{a-1},~ f^{-1}(a - \epsilon),~\frac{(1-\alpha)(K-1)}{\alpha}
    \right\}
    >
    0,
  \]
  and let
  \[
    m^* := \frac{ \int_{ 0}^{ \theta} f(t) \dd t  - \theta  }{\beta } \le \frac{ (a-1)\theta}{\beta} \le \bigwedge_{k=1}^K m_k.
  \]
  We have
  $\int_{\theta}^{1} f(t)\d t = 1 - \int_{0}^{\theta} f(t) \d t = 1 - \theta - m^* \beta$.
  Note that $a > f(\theta) \ge a - \epsilon > \bigvee_{k=1}^K y_k + 2\epsilon$.
  For $k=1,\dots,K$, define a probability density function
  \begin{equation}\label{eq:h-cond}
    h_k = \frac{1}{1- \theta - m^*}\left(h \id_{(0,f(\theta)]} - m^* \frac{v_k}{m_k}\right),
  \end{equation}
  which is supported in interval $(0,f(\theta)]$, and its mean $\mu(h_k)$ satisfies
  \[
    \mu(h_k)
    =
    \frac{\int_{\theta}^{1} f(t)\d t - m^* \mu(v_k)}{1 - \theta - m^*}
    =
    \frac{1 - \theta - m^*\beta- m^*\mu(v_k)}{1 - \theta - m^*}.
  \]
  We have
  \[
    \sum_{k=1}^K \mu(h_k)
    =
    \frac{K (1 - \theta - m^* \beta) - m^* \sum_{k=1}^K\mu(v_k)}{1 - \theta - m^*}
    =
    K > f(\theta).
  \]
  Note that each of $h_1,\dots,h_K $ has a decreasing density in $(0,f(\theta)]$,
  and the sum of their means is larger than $f(\theta)$,
  thus satisfying the condition of joint mixability in \citet[Theorem 3.2]{WW16}.
  Using that theorem,
  there exists a random vector $\mathbf{X} = (X_1,\dots,X_K)$ satisfying $X_k\sim h_k$, $k=1,\dots,K$, and $X_1 + \cdots + X_K = K$. 

  Take disjoint events $A,B,C, B_1,\dots,B_K$ independent of $\mathbf{X}$ such that
  $\p(A) = (1 - \theta - m^*) \alpha$, $\p(B) = m^*\alpha$,
  $\p(C) = 1 - \alpha - {\theta \alpha}/(K-1)$
  and $\p(B_1)= \cdots = \p(B_K) = \theta \alpha /(K-1)$.
  Design a random vector $\mathbf P=(P_1,\dots,P_K)$ by letting, for $k= 1,\dots,K$,
  \begin{equation}\label{eq:convex-constr}
    P_k = \alpha f^{-1} (X_k) \id_{A} + p_k  \id_{B} +  \sum_{j = 1, j\neq k }^K \theta \alpha\id_{B_j} + \id_{B_k} + \id_{C}.
  \end{equation}
  The decomposition \eqref{eq:h-cond} gives, for each $k=1,\dots,K$, that
  \[
    \frac{\p( f^{-1} (X_k)    \id_{A}  + f^{-1} (y_k)\id_B > x)}{(1-\theta)\alpha} \ge \frac{1-x}{1-\theta}
    \qquad\text{for all $x\in (\theta,1)$,}
  \]
  and thus the conditional distribution of $f^{-1} (X_k) \id_{A} + f^{-1} (y_k)\id_B$ on $A\cup B$
  is stochastically larger than the $\U[\theta,1]$.
  As a consequence,
  the distribution of $P_k$ is stochastically larger than
  $\theta\alpha\delta_{\theta\alpha} + (1 - \theta )\alpha \U[ \theta\alpha,\alpha] + (1 - \alpha) \delta_{1}$,
  and hence $P_k$ is a p-variable.

  If $A$ happens, then $ f(P_k/\alpha)= X_k$ for each $k$,  and $\sum_{k=1}^K f(P_k/\alpha) = \sum_{k=1}^K X_k = K$. 
  If any of $B_k$ happens, 
  then $\sum_{k=1}^{K} f(P_k/\alpha)= (K-1)f(\theta) > (K-1)(a-\epsilon)> K$.
  In both cases, using \eqref{eq:calibrator2}, $\mathbf{P}\in R_{\alpha}(F) \subseteq R_{\alpha}(G)$. 
  If $B$ happens, then $\mathbf{P} = \mathbf{p} \in  R_{\alpha}(G) $. 
  Therefore,
  \begin{equation}\label{eq:admis-con}
    \p(\mathbf P \in R_{\alpha}(G)) \ge \p(A) + \p(B)  + \sum_{k=1}^K \p(B_k)  = \alpha + \frac{\theta \alpha}{K-1} > \alpha,
  \end{equation}
  a contradiction to $G $ being a p-merging function.
  This shows that $F $ is admissible. 
\end{proof}

Rephrasing the condition on $g$ in Theorem \ref{th:admissible},
we get a sufficient condition on an admissible calibrator $f$ to ensure that the induced p-merging function is admissible:   
\begin{equation}\label{eq:f-cond}
  \begin{aligned}
    &\mbox{For some $\eta\in[0,\frac1K)$ and $\tau:=1-(K-1)\eta$:
      $f=K$ on $(0,\eta]$, $f(\eta+)\in(\frac{K}{K-1},K]$,}\\[-2mm]
    &\mbox{$f$ is strictly convex or strictly concave on $(\eta,\tau]$, and $f(1)=0$.}
  \end{aligned}
\end{equation}
Notice that the condition $f(\eta+)\in(\frac{K}{K-1},K]$ in \eqref{eq:f-cond} and Theorem~\ref{th:admissible}
excludes the simple case $K=2$ (treated in Supplemental Article, Section~\ref{app:K2}).
One may try to relax the requirement that convexity or concavity be strict;
we explain technical difficulties in Remark \ref{rem:strict-con} in Supplemental Article, Section~\ref{app:a8}, for the interested reader.

\begin{algorithm}[bt]
  \caption{The p-merging function induced by a calibrator $f$ to accuracy $2^{-M}$}
  \label{alg:generic}
  \begin{algorithmic}
    \Require
      A calibrator $f$, $M\in\N$, and a sequence of p-values $p_1,\dots,p_K$.
    \State $L:=0$ and $R:=1$
    \For{$m=1,\dots,M$}
      \State $\epsilon:=(L+R)/2$
      \If{$\frac{1}{K}\sum_{k=1}^K f(p_k/\epsilon)\ge1$}
        $R:=\epsilon$ \textbf{else} $L:=\epsilon$
      \EndIf
    \EndFor
    \State\Return $R$
  \end{algorithmic}
\end{algorithm}

A natural way to compute the p-merging function induced by a calibrator $f$ to accuracy $2^{-M}$,
where $M$ is a natural number, is to use binary search,
which is given as Algorithm~\ref{alg:generic}.
The value of this merging function is given by $\phi_{\mathbf p}^{-1}(1)\wedge 1$,
where $\phi_{\mathbf p}^{-1}$ is the left-inverse of
\[
  \phi_{\mathbf p}:
  \epsilon\mapsto\frac1K\sum_{k=1}^K f(p_k/\epsilon),
\]
and the algorithm essentially solves the equation $\phi_{\mathbf p}(\epsilon)=1$.
Assuming that the calibrator $f$ is computable in time $O(1)$,
merging $K$ p-values by Algorithm~\ref{alg:generic} takes time $O(M K)$.
Notice that Algorithm~\ref{alg:generic} always produces a valid p-value
(which exceeds the p-value produced by the p-merging function induced by $f$ by at most $2^{-M}$).

In the following few sections, we analyze admissibility of the Hommel function, members of the O-family, and members of the M-family.
In cases of non-admissibility, we construct a dominating admissible p-merging function.
It turns out that, except for the Bonferroni p-merging function,
none of these p-merging functions has a calibrator satisfying the condition \eqref{eq:f-cond},
and many of them can indeed be improved, either trivially or significantly.
Theorem \ref{th:admissible} becomes very useful in the construction of admissible p-merging functions dominating the ones in the M-family.

\section{Hommel's function and the O-family}
\label{sec:o}

This section is dedicated to the admissibility of the Hommel function $H_K$ and the O-family of p-merging functions $(G_{k,K})_{k=1,\dots,K}$ for a given $K$.
The calibrators we see below are generally not continuous, and hence they do not satisfy the condition in Theorem \ref{th:admissible}.
Nevertheless, some alternative arguments will justify the (in-)admissibility of the induced functions.
The key result of this section is Theorem~\ref{th:o1} about the grid harmonic p-merging function.

\subsection{Grid harmonic merging function}

\begin{figure}[tb]
\begin{center}
  \includegraphics[width=0.75\textwidth, trim={0 0cm 0 0cm}, clip]{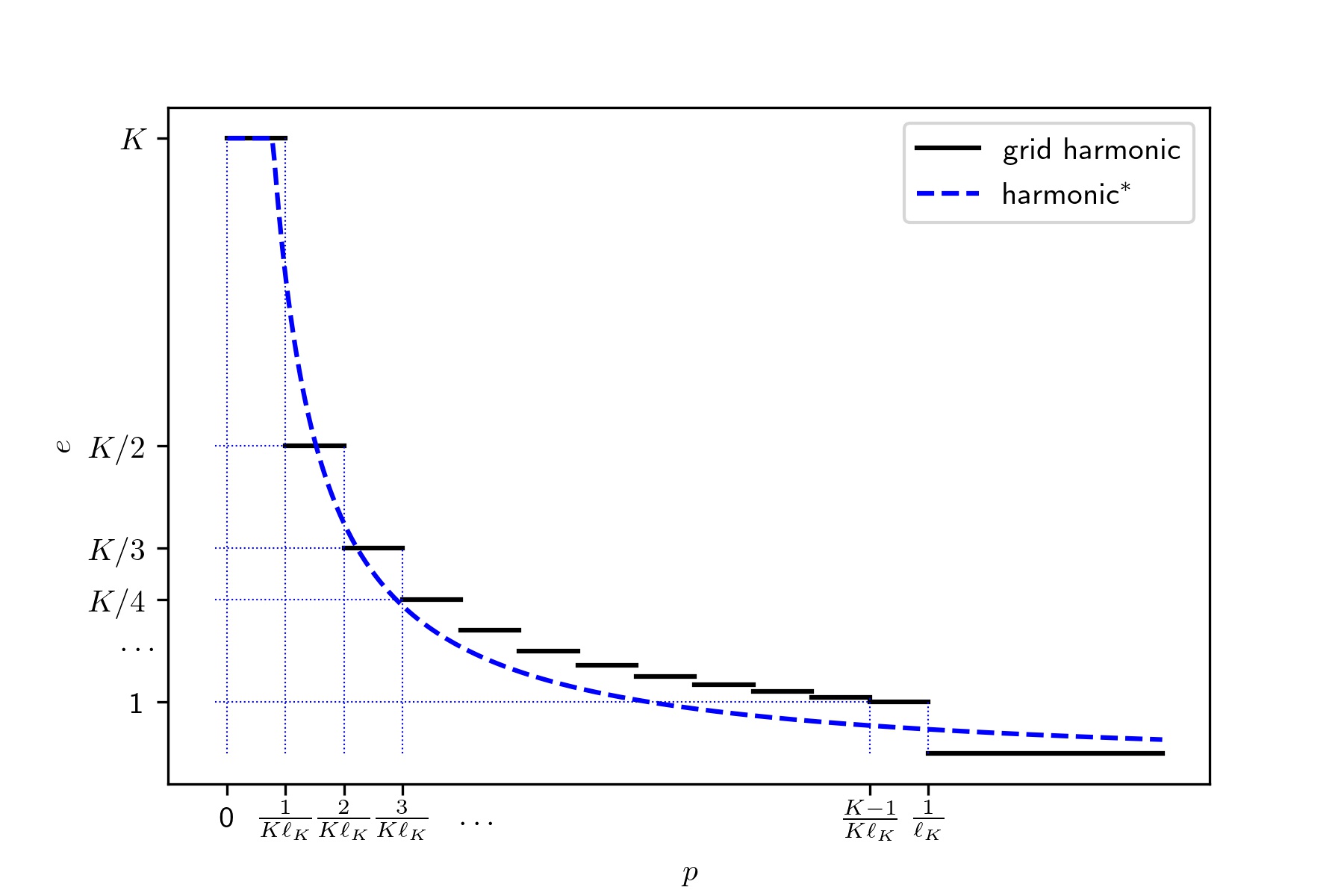}
\end{center} 
\caption{The grid harmonic calibrator (solid and black)
  and the harmonic${}^*$ calibrator (dashed and blue), for $K:=12$}\label{fig:Hommel}
\end{figure}

We first show that the Hommel function $H_K\wedge 1$ is not admissible,
and it can be strictly improved to an admissible p-merging function $H_K^*$.
Recall that $H_K$ is given by $H_K := \ell_K \bigwedge_{k=1}^K G_{k,K}$,
where $\ell_K := \sum_{k=1}^K \frac1k$.
Our modification $H^*_K$ of the Hommel function will be induced by the function $f:[0,\infty)\to[0,\infty)$ defined by
\begin{equation}\label{eq:Hkf}
  f: x \mapsto \frac{K\id_{\{\ell_K x\le1\}}}{\lceil K\ell_K x\rceil},
\end{equation}
which we call the \emph{grid harmonic calibrator} and whose graph is shown in Figure~\ref{fig:Hommel}
as the black piece-wise horizontal line.
It is straightforward to check that $f$ is decreasing, $f(1)=0$, and $\int_0^1 f(x) \d x = 1$,
and hence $f$ is indeed a calibrator.
We will also refer to $H^*_K$ as the \emph{grid harmonic p-merging function}.

\begin{theorem}\label{th:o1}
  The p-merging function $H_K\wedge 1$ is dominated (strictly if $K\ge 4$)
  by the grid harmonic p-merging function $H_K^*$.
  Moreover, $H_K^*$ is always admissible among symmetric p-merging functions, and it is admissible if $K$ is not a prime number.
\end{theorem}
\begin{proof} 
  Since $f$ induces $H^*_K$, by Theorem \ref{pr:e}, $H^*_K$ is a p-merging function.

Let us verify that $H_K\ge H^*_K$.
The rejection region  of $H^*_K$ satisfies
\begin{equation}\label{eq:hprime2}
  R_\epsilon(H^*_K)
  =
  \left\{
    \mathbf p \in [0,\infty)^K:
    \sum_{k=1}^K \frac{\id_{\{\ell_K p_k \le \epsilon\}}}{\lceil K\ell_K p_k/\epsilon \rceil} \ge 1\right\}.
\end{equation}   
For any $\mathbf p\in [0,\infty)^K$ and $\epsilon>0$, 
if $H_K(\mathbf p)\le \epsilon $, then there exists $m=1,\dots,K$
such that $\#\{k : K\ell_K  p_{k} / m \le \epsilon\} \ge m$.
It follows that
\[
  \sum_{k=1}^K
  \frac{\id_{\{\ell_K p_k \le \epsilon\}}}{\lceil K\ell_K p_k/\epsilon \rceil}
  \ge
  \sum_{k=1}^K
  \frac 1m \id_{\{ K\ell_Kp_k/\epsilon\le m\}}
  =
  \frac{1}{m} \#\{k : K\ell_K  p_{k} / m \le \epsilon\}
  \ge
  1.
\]
By \eqref{eq:hprime2}, $\mathbf p\in R_\epsilon(H^*_K)$, and thus $H^*_K(\mathbf p)\le \epsilon$.
This shows  $H_K\ge H^*_K$.
It is easy to check that the reverse direction holds (i.e., $H_K=H_K^*$) if and only if $K\le 3$.

Next, we prove the admissibility of $H_K^*$.
Set $\tau := 1/({K \ell_K})$.
Using Proposition \ref{pr:g}, suppose, for the purpose of contradiction,
that there exists a calibrator $g$ satisfying \eqref{eq:rejrelation}.
For $x\in (0,K\tau]$, set  $p_1 =\dots = p_m = x $ and $ p_{m+1} =\dots = p_K >1 $, where $m:= \lceil \tau x\rceil$.
Since  $f(x) = K/m$, we have $ \sum_{k=1}^{K} f(p_k) = K$.
Using \eqref{eq:rejrelation},
$ K \le \sum_{k=1}^{K} g(p_k) = m g(x)$, and thus $g(x) \ge K/m = f(x)$.

Since $x\in (0,K\tau]$ is arbitrary, we have $\int_0^{K\tau} g(x) \dd x \ge \int_0^{K\tau} f(x) \dd x = 1$.
As $g$ is a calibrator, this means  $g = f$ almost everywhere on $[0,1]$.
Moreover, $f$ is left-continuous, which further implies $g\le f$.
Hence, both sides of \eqref{eq:rejrelation} coincide, leading to a contradiction.
Thus, $H_K^*$ is admissible among symmetric p-merging functions.

Finally, we show that $H_K^*$ is admissible if $K$ is not a prime number.
Suppose that there exist $(\lambda_1,\dots,\lambda_K) \in \Delta_K $ and calibrators $g_1,\dots,g_K$ satisfying \eqref{eq:rejrelation2}.
For each $m,k =1,\dots,K$, set $y_{m,k} := \lambda_k g_k(m\tau)$ and $T_m := \sum_{k=1}^K y_{m,k}$.

Fix any $m = 1,\dots,K$.
Let $\Pi_m$ be the set of all subsets of $\{1,2,\dots,K\}$ of exactly $m$ elements.
There are $\binom{K}{m}$ elements (sets) in $\Pi_m$.
For any $J \in \Pi_m$, take any $\beta>1$ and let $\mathbf{p} = (p_1,\dots,p_K)$ be given by $p_k = m\tau \id_{\{k \in J\}} + \beta \id_{\{k \notin J\}}$, $k=1,\dots,K$.
Since $\sum_{k=1}^{K} f(p_k) = K$, \eqref{eq:rejrelation2} implies $ 1 \le \sum_{k=1}^K \lambda_k g_k(m\tau) =  \sum_{k\in J} y_{m,k}$.
Therefore,
\[
  \binom{K}{m}
  \le
  \sum_{J\in \Pi_m} \sum_{ k\in J } y_{m,k}
  =
  \binom {K-1}{m-1} \sum_{k=1}^{K} y_{m,k}
  =
  \binom {K-1}{m-1} T_m.
\]
This gives $T_m\ge K/m$.

For $x\in ((m-1)\tau,m\tau]$ and each $k$,
we have $\lambda_k g_k(x)\ge \lambda_k g_k(m\tau) = y_{m,k}$,
and hence $\lambda_k \ge \int_0^{K\tau} \lambda_k g_k(x) \dd x \ge \tau\sum_{m=1}^K y_{m,k}$.
Therefore,
\begin{equation}\label{eq:hk-new1}
  \sum_{m=1}^K T_m
  =
  \sum_{m=1}^K \sum_{k=1}^K y_{m,k}
  =
  \sum_{k=1}^K \sum_{m=1}^K y_{m,k}
  \le
  \frac{1}{\tau} \sum_{k=1}^K \lambda_K = \frac{1}{\tau} = \sum_{m=1}^K\frac K m.
\end{equation} 

Putting $\sum_{k\in J} y_{m,k} \ge 1$, $T_m\ge K/m$ and \eqref{eq:hk-new1} together,
we get $T_m = K/m$ for each $m=1,\dots,K$, and $\sum_{k\in J} y_{m,k} = 1$ for each $J\in \Pi_m$.
This further implies    $y_{m,k} = 1/m$ for all $m\le K-1$ and all $k$.
Note that the case of $m=K$ is not concluded here since $\Pi_K$ only has one element,
and the analysis of this case requires $K$ to not be a prime number.
Write $K=k_1k_2$ for some integers $k_1,k_2 \ge 2$.

Take any $I \in \Pi_{k_1} $ and $J \in \Pi_{k_2 -1}$ such that $I\cap J = \varnothing$, by noting that $k_1 + k_2 - 1 < K$.
Let $\mathbf{p} = (p_1,\dots,p_K)$ be given by
\[
  p_k =  K\tau \id_{\{k \in I\}} +  k_2\tau \id_{\{k \in J\}} + \beta \id_{\{k \notin I \cup J\}},
  \qquad
  k=1, \dots,K.
\]
We have $\sum_{k=1}^{K}f(p_k) = k_1 + (k_2-1) K/k_2 = K$.
By \eqref{eq:rejrelation2} and $y_{k_2,k}=1/k_2$, we have
\[
  1
  \le
  \sum_{k=1}^{K} \lambda_k g_k(p_k)
  =
  \sum_{k\in I} y_{K,k} + \sum_{k\in J} y_{k_2,k}
  =
  \sum_{k\in I} y_{K,k} + (k_2-1) \frac{1}{k_2}.
\]
Hence, $\sum_{k\in I} y_{K,k} \ge k_1/K $ for any $I \in \Pi_{k_1}$.
On the other hand, $\sum_{k=1}^K y_{K,k} = T_K = 1$, which leads to $y_{K,k} = 1/K$ for all $k=1,\dots,K$.
Therefore, we obtain $y_{m,k} = \frac{1}{m}$ for all $m,k = 1,\dots,K$.
This implies
\[
  \lambda_k \ge \int_0^{K\tau} \lambda_k g_k(x) \dd x \ge \tau\sum_{m=1}^K y_{m,k} = \frac 1K.
\]   
Since $\sum_{k=1}^K \lambda_k=1$,
we now know $g_k=f$ almost everywhere, which further implies $g_k\le f$, and $\lambda_k = 1/K$, $k=1,\dots,K$.
Therefore, both sides of \eqref{eq:rejrelation2} coincide, which is a contradiction.
Thus, $H^*_K$ is admissible if $K$ is not a prime number.
\end{proof}

For computing $H^*_K$, we can use our generic algorithm, Algorithm~\ref{alg:generic},
which takes time $O(-K\log\delta)$, where $\delta$ is the desired accuracy.
A precise expression is, e.g.,
\[
  H^*_K(p_1,\dots,p_K)
  :=
  \min
  \left\{
    \epsilon:=\frac{K\ell_K p_j}{i}
    :
    i,j\in\{1,\dots,K\},
    \frac{1}{K}
    \sum_{k=1}^K
    f\left(\frac{p_k}{\epsilon}\right)
    \ge
    1
  \right\},
\]
where the range of $\epsilon$ follows from $f$ changing its value only at the points of the form $i/(K\ell_K)$.
However, this expression takes time $O(K^3)$ to compute.

Since $f(x)\le1/(\ell_K x)$, we have $H^*_K\ge\ell_K M_{-1,K}$.
It is instructive to compare this with the row of \citet[Table~1]{Vovk/Wang:2020Biometrika} for the harmonic mean.

Using Theorem \ref{thm:Simes}, we have $S_K\le F\le H_K$ for any symmetric p-merging function $F$ dominating $H_K$, including $F=H^*_K$.
Hence, the improvement of any $F$ over $H_K$, measured by the ratio $H_K/F$, should always be in $[1,\ell_K]$.
The improvement ratio $H_K/H_K^*$ will be analyzed in Section \ref{sec:improve}.

In Theorem \ref{th:o1}, we obtain that $H_K^*$ is admissible if $K$ is not a prime number.
Quite surprisingly,
if $K$ is a prime number, then $H_K^*$ may be strictly dominated by some non-symmetric p-merging functions.
In the following simple example, we give the dominating functions for $K=2$ and $K=3$.
More complicated examples can be constructed for larger prime numbers,
although we do not know whether $K$ being prime always implies non-admissibility of $H_K^*$.
\begin{example}\label{ex:prime}
In the case $K=2$,
$H^*_2: (p_1,p_2)\mapsto  3 p_{(1)} \wedge \frac{3}{2}p_{(2)}$
is strictly dominated by $F: (p_1,p_2) \mapsto 3p_1 \wedge \frac{3}{2} p_2$,
which is a (non-symmetric) p-merging function because for any p-variables $P_1,P_2$ and $\alpha \in (0,1)$, 
\[
  \textstyle
  \p(F(P_1,P_2) \le \alpha)
  \le
  \p \left(P_1 \le \frac 13 \alpha \right) + \p\left(P_2 \le \frac 23 \alpha \right)
  \le
  \frac{1}{3} \alpha + \frac{2}{3} \alpha
  =
  \alpha.
\]
In the case $K=3$, $H_3^*$ is induced by the calibrator $3g$ on $(0,1]$, where
\[
  \textstyle
  g := \id_{[0,2/11]} + \frac12 \id_{(2/11,4/11]} + \frac13 \id_{(4/11,6/11]}.
\]
Let the function $F$ be  given by the rejection set, for $\epsilon \in (0,1)$,
\[
  R_{\epsilon}(F) = \epsilon \{\mathbf{p} \in [0,\infty)^3: g_1(p_1) + g_2(p_2) + g_3(p_3) \ge 1  \},
\]
where $g_1 := g + \frac{1}{6} \id_{(4/11, 6/11]}$,
$g_2:= g - \frac{1}{12}\id_{ (4/11, 6/11]}$, and $g_3:=g_2$.
By Theorem \ref{th:e}, $F$ is a (non-symmetric) p-merging function.
Direct calculation shows that $F$  strictly dominates $H_3^*$.
\end{example}

Example \ref{ex:prime} also shows that $H_K\wedge 1$ is not admissible for any $K\ge 2$,
since it is either strictly dominated by $H_K^*$ ($K\ge 4$) or by the functions in Example~\ref{ex:prime} ($K=2,3$).

\subsection{Admissibility for the O-family}

Next, we show that, except for the maximum merging function $G_{K,K}$, each member of the O-family is admissible
if we trivially modify it by a zero-one adjustment, as in Proposition~\ref{prop:lsc}.
Although $G_{K,K}$ fails to be admissible, it is admissible among symmetric p-merging functions after this modification.

\begin{theorem}\label{pr:o1} 
  The p-merging function
  \[
    \mathbf p\mapsto
    G_{k,K}(\mathbf{p}\wedge\mathbf{1})\wedge\id_{\{\min(\mathbf p)>0\}}
    =
    G_{k,K}(\mathbf{p})\wedge\id_{\{\min(\mathbf p)>0\}}
  \]
  is admissible for $k=1,\dots,K-1$,
  and it is admissible among symmetric p-merging functions for $k=K$.
\end{theorem}

\begin{proof}
As we see from Example \ref{ex:ex-th:e}, for each $k=1,\dots,K$,
$\mathbf p\mapsto G_{k,K}(\mathbf{p})\wedge\id_{\{\min(\mathbf p)>0\}}$
is induced by $f: x\mapsto\infty\id_{\{x=0\}} + (K/k) \id_{\{x\in(0,k/K]\}}$.
 
First, fix $m=1,\dots,K-1$.
Using Proposition \ref{pr:g}, suppose, for the purpose of contradiction,
that there exist $(\lambda_1,\dots,\lambda_K) \in \Delta_K $ and calibrators $g_1,\dots,g_K$ satisfying \eqref{eq:rejrelation2}.
For each $k=1,\dots,K$,
denote $y_k := \lambda_k g_k(m/K)$.
Since $1 = \int_0^1 g_k(x)\dd x \ge \frac{m}{K} g_k(m/K)$,
we have $ y_k \le \lambda_k K/m$, which implies $\sum_{k=1}^K y_k \le K/m$.

Let $\Pi_m$ be the set of all subsets of $\{1,2,\dots,K\}$ of exactly $m$ elements.
There are $\binom{K}{m}$ elements (sets) in $\Pi_m$.
For any $J\in \Pi_m$ , take any $\beta > 1$ and let $\mathbf{p} = (p_1,\dots,p_K)$ be given by $p_k = \frac{m}{K} \id_{\{k\in J\}} + \beta \id_{\{k\notin J\}}$,
$k=1,\dots,K$.
Since $\sum_{k=1}^K f(p_k) = K$, \eqref{eq:rejrelation2} implies $1\le \sum_{k=1}^K \lambda_k g_k (p_k) = \sum_{k\in J} y_k$.
Therefore, 
\[
  \binom{K}{m} \le \sum_{J\in \Pi_m} \sum_{ k\in J } y_k = \binom {K-1}{m-1} \sum_{k=1}^{K} y_k \le \binom {K-1}{m-1} \frac{K}{m} = \binom{K}{m}.
\]
This implies $ \sum_{k\in J} y_k =1$ for each $J \in \Pi_m$, and further $y_k = 1/m$ for each $k=1,\dots,K$.
Therefore, $ \lambda_k \ge \int_{0}^{m/K} \lambda_k g_k(x)\dd x \ge \frac{m}{K} y_k = 1/K$.
Since $\sum_{k=1}^K \lambda_k =1$, we have $g_k = f$ almost everywhere,
which further implies $g_k \le f $, and $\lambda_k = 1/K $, $k=1,\dots,K$.
Therefore, both sides of \eqref{eq:rejrelation2} coincide, which is a contradiction.
Thus, $G_{m,K}(\mathbf{p}) \wedge \id_{\{ \min(\mathbf{p})>0\}}$ is admissible for each $m=1,\dots,K-1$.

To prove the statement for $m=K$, suppose that there exists a calibrator $g$ satisfying  \eqref{eq:rejrelation}.
Since $f(x) = 1$ for $x\in (0,1]$, we have $\sum_{k=1}^K f(x) =K$, which gives $K\le  K g(x) $, and thus $ g(x) \ge K/m = f(x)$.  
We have $\int_0^{m/K} g(x)\dd x \ge \int_0^{m/K} f(x)\dd x = 1 $.
As $g$ is a calibrator, this means $g=f$ almost everywhere and further implies $g\le f$.
Therefore, both sides of \eqref{eq:rejrelation} coincide, which is a contradiction.
Thus, $G_{K,K}(\mathbf{p}) \wedge \id_{\{ \min(\mathbf{p})>0\}}$ is admissible among symmetric p-merging functions. 
\end{proof}

\section{The M-family}\label{sec:M}

In this section, we study admissibility and the domination structure among the M-family of p-merging functions,
which turn out to be drastically different from those of the O-family, as members in the M-family are generally not admissible,
except for the cases of $ F_{-\infty,K}$ and $F_{\infty,K}$ covered in Theorem \ref{pr:o1}.
The key result of this section is Theorem~\ref{th:m1}, which gives another admissible p-merging function.

\subsection{Coefficients in the M-family}
\label{sec:brk}

To study functions $F_{r,K}=b_{r,K} M_{r,K} \wedge 1$ in the M-family,
we first need to identify the constants $b_{r,K}$, which unfortunately do not always admit an analytical form.
The values of $b_{r,K}$ are obtained in \citet{Vovk/Wang:2020Biometrika} for the cases $r\ge 1/(K-1)$ (Proposition 3),
$r=0$ (Proposition 4), and $r=-1$ (Proposition 6),
where the proposition numbers refer to those in \citet{Vovk/Wang:2020Biometrika}.
In addition, the values $b_{-\infty,K}=K$ and $b_{\infty,K}=1$ are trivial to check.
Below, we complement these results by providing formulas of $b_{r,K}$ for all $r\in \R$ via an  analytical equation.
We fix some notation which will be useful throughout this section.
For a fixed $K$ and $r\in (-\infty,1/(K-1))$, let $c_r$ be the unique number $c\in(0,1/K)$ solving the equation
\begin{align*}
  (K-1)(1-(K-1)c)^r+c^r
  =
  K\frac{(1-(K-1)c)^{r+1}-c^{r+1}}{(r+1)(1-Kc)}, & \quad\text{if $r\notin\{-1,0\}$}; \\
  \frac{1-Kc}{Kc(1-(K-1)c)} = \log(1/c-(K-1)), & \quad\text{if $r=-1$}; \\
  K(1-Kc) = \log(1/c-(K-1)), & \quad\text{if $r=0$}.
\end{align*}
The existence and uniqueness of the solution $c$ to the above equation can be checked directly,
and it is implied by Lemma~3.1 of \citet{Jakobsons/etal:2016} in a more general setting.
Moreover, set $c_r:=0$ if $r\ge 1/(K-1)$, and write
\begin{align}\label{eq:def-dr}
  d_r:=1-(K-1)c_r, \quad r\in \R.
\end{align}
Notice that we always have $0\le c_r<1/K<d_r\le1$.

The proofs of propositions in this section are put in Supplemental Article, Section~\ref{app:a7}.

\begin{proposition}\label{prop:grand}
  For $K\ge 2$ and $r\ge\frac{1}{K-1}$, we have $b_{r,K}=((r+1)\wedge K)^{1/r}$.
  For $K\ge 3$ and $r\in(-\infty,\frac{1}{K-1})$, we have $b_{r,K}=1/M_{r,K}(c_r,d_r,\dots,d_r)$. 
  For $r\in(-\infty,1)$, we have $b_{r,2}=2$. 
\end{proposition}

Via well-known inequalities on generalized mean functions \citep{Hardy/etal:1952},
it is straightforward to check, without using Proposition \ref{prop:grand},
that if $r<s$ and $rs>0$, then
\begin{equation} \label{eq:compare-p1}
  K^{1/s - 1/r}b_{r,K} \le b_{s,K} \le b_{r,K}.
\end{equation}
The relationship \eqref{eq:compare-p1} conveniently gives, among other implications,
the monotonicity of the mapping $r\mapsto b_{r,K}$ and its continuity except at $0$.
The continuity at $0$ can be verified via Proposition~\ref{prop:grand}.

\subsection{Admissibility of the M-family and improvements}

As illustrated by the numerical examples in \citet{Vovk/Wang:2020Biometrika} and \citet{Wilson:2020},
the most useful cases of the M-family are those with $r\le 0$.
In particular, the \emph{harmonic p-merging function} $F_{-1,K}$,
which is a constant times the harmonic mean p-value of \citet{Wilson:2019} (truncated to 1),
has a special role among the M-family, and it performs similarly to the Hommel function; see \citet{Chen:2020}.
On the other hand, the members $F_{r,K}$ for $r>1$ are rarely useful in practice due to their heavy dependence on large realized p-values.

As we already mentioned, members of the M-family are generally not admissible,
and we will construct dominating admissible functions.
We briefly explain the main idea for the case $r<0$, as the other cases are similar.
Using the equality $b_{r,K}^{r} = K(c^r_r+(K-1)d_r^r)^{-1}$ in Proposition \ref{prop:grand},
the rejection region of $F_{r,K}$ for $\epsilon \in (0,1)$ is given by
\[
  R_\epsilon(F_{r,K})
  =
  \epsilon
  \left\{
    \mathbf p \in [0,\infty)^K: \frac{\sum_{k=1}^K p_k^{r}}{c_r^r + (K-1)d_r^r} \ge 1
  \right\}
  =
  \epsilon
  \left\{
    \mathbf p \in [0,\infty)^K:
    \sum_{k=1}^K\frac{p_k^{r} - d_r^r}{c_r^r - d_r^r} \ge 1
  \right\}
\]
(see Example \ref{ex:ex-th:e2}).
The strictly convex function $x\mapsto K(x^r-d_r^r)/(c_r^r - d_r^r)$ is generally not a calibrator. 
Nevertheless, there is a simple modification which induces a p-merging function dominating $F_{r,K}$. 
Define the function 
\[
  f_r: x\mapsto K
  \left(
    \frac{x^r - d_r^r}{c_r^r - d_r^r} \wedge 1
  \right)_+.
\]
We can check that each $f_r$ is a calibrator.
Let $F_{r}^*$ be the p-merging function induced by $f_r$, that is,
\begin{equation}\label{eq:Frej}
  R_{\epsilon}(F_{r}^*)
  =
  \epsilon
  \left\{
    \mathbf p \in [0,\infty)^K:
    \sum_{k=1}^K
    \left(
      \frac{p_k^r - d_r^r}{c_r^r - d_r^r}
    \right)_+
    \ge
    1
  \right\},
  \qquad
  \epsilon\in(0,1).
\end{equation} 
It is clear that $F_{r}^*$ dominates $F_{r,K}$.
Moreover, the calibrator $f_r$ satisfies \eqref{eq:f-cond} with $\eta=c_r$,
which means that $F_{r}^*$ is admissible by Theorem \ref{th:admissible}.
In this way, an admissible p-merging function dominating $F_{r,K}$ is constructed.

In the next result, we give a rigorous statement of the above idea for all $r<K-1$,
and show that the rejection regions of $F_{r}^*$ have a very simple relationship to those of $F_{r,K}$.
Remember that the minimum $\wedge$ of two vectors is understood component-wise.

\begin{theorem}\label{th:m1}
  For $K\ge 3$ and $r\in(-\infty,K-1)$,
  $F_{r,K}$ is strictly dominated by the p-merging function $F_{r,K}^*$ defined,
  for $\mathbf p\in (0,\infty)^K$ and $\epsilon \in (0,1)$,
  via
  \begin{equation}\label{eq:M-equiv}
    F^*_{r,K}(\mathbf p) \le \epsilon
    ~\Longleftrightarrow~
    F_{r,K}(\mathbf p\wedge (\epsilon d_r  \mathbf 1 ))\le \epsilon
    \mbox{~~or~~$\min(\mathbf p)=0$},
  \end{equation}
  where $d_r$ is given in \eqref{eq:def-dr}.
  Moreover, $F_{r,K}^*$ is admissible unless $r=1$.
\end{theorem}

The proof of the theorem will show that $F^*_{r,K}=F^*_{r}$.

\begin{proof}
We first address the case $r<1/(K-1)$. 
Note that, for $r\in (0,1/(K-1))$,
\[
  R_\epsilon ( F_{r,K})   = \epsilon \left\{\mathbf p \in [0,\infty)^K:  \frac { \sum_{k=1}^K p_k^{r} }{c_r^r + (K-1)d_r^r} \le 1 \right\}
  =
  \epsilon
  \left\{
    \mathbf p \in [0,\infty)^K:
    \sum_{k=1}^K\frac{p_k^{r}-d_r^r}{c_r^r - d_r^r} \ge 1
  \right\}
\]
and
\[
  R_\epsilon(F_{0,K})
  =
  \epsilon
  \left\{
    \mathbf p \in [0,\infty)^K:
    \sum_{k=1}^K  \frac { \log p_k - \log d_0}{\log c_0 - \log d_0} \ge 1
  \right\},
\]
which share a form very similar to the case $r<0$.
Define the functions
\[ f_r : x\mapsto  K  \left(\frac{   x^r -  d_r^r }{ c_r^r - d_r^r    } \wedge 1   \right)_+
  \mbox{~for $r\ne 0$~~~and~~~} 
  f_0: x\mapsto K   \left(  \frac{  \log x -  \log d_0 }{ \log c_0 - \log d_0  } \wedge 1\right)_+.
\]
We can check with Proposition \ref{prop:grand} that
\[
  \int_{c_r} ^{d_r} \frac{   x^r -  d_r^r }{ c_r^r - d_r^r    } \dd x
  =
  \frac{1-Kc_r}{ c_r^r - d_r^r  }\left( \frac{c_r^r + (K-1)d_r^r}{K}  -d_r^r\right) = \frac{1-Kc_r}{K},
\]
which implies $\int_0^1 f_r (x) \d x=1$, 
and similarly for $r=0$.
Hence, $f_r$ is a calibrator, which further satisfies \eqref{eq:f-cond}.
As we explained above for the case $r<0$,  the p-merging function $F^*_r$ induced by $f_r$ strictly dominates 
$F_{r,K}$, and  the admissibility of $F^*_r$  follows from Theorem \ref{th:admissible}.   
Finally,
comparing the conditions for $\mathbf p \in R_\epsilon(F_{r,K})$ and $\mathbf p \in R_\epsilon(F_{r}^*)$, i.e., if $r\ne 0$,
\[
  \sum_{k=1}^K\frac {  (p_k/\epsilon)^{r} -d_r^r }{c_r^r - d_r^r} \ge 1
  \mbox{~~and~~}
  \sum_{k=1}^K\left(\frac {  (p_k/\epsilon)^{r} -d_r^r }{c_r^r - d_r^r} \right)_+\ge 1,
\]
the only difference is that any value $p_k $ larger than $d_r\epsilon$ is treated as $d_r\epsilon$ by $ F^*_{r}$.
This implies $F^*_r=F_{r,K}^*$ for $F_{r,K}^*$ in \eqref{eq:M-equiv}.
The case $r=0$ is similar.

Next, we prove the statement for $r\in [1/(K-1),K-1)$.
Using Proposition \ref{prop:grand}, $  b_{r,K}^{r} = r+1$.
Hence,   the rejection region  of $F_{r,K}$ for $\epsilon \in (0,1)$ is given by 
\[
  R_\epsilon (F_{r,K})
  =
  \epsilon \left\{\mathbf p \in [0,\infty)^K: \frac{r+1}{K} \sum_{k=1}^K p_k^{r} \le 1 \right\} 
  =
  \epsilon \left\{\mathbf p \in [0,\infty)^K: \frac{1}{K} \sum_{k=1}^K g_r(p_k) \ge 1 \right\},
\]
where $g_r:x\mapsto (r+1)(1-x^r)/r$.
Let $\tau=  r/(r+1)\in [1/K, 1-1/K)$.
Define a function $f_r: x\mapsto \tau^{-1} (1-x^r)_+$ for $x>0$ and $f_r(0)=K$.
It is clear that $f_r$ is a calibrator by checking $\int_0^1 f_r (x) \d x =1 $.
Since $f_r \ge g_r$,
we know that the p-merging function  $F^*_r$ induced by $f_r$ dominates $F_{r,K}$.
The domination $F^*_r\le F_{r,K}$ is strict
since it is easy to find some $p_1,\dots,p_K\in (0,\infty)$ such that $\sum_{k=1}^K f_r(p_k) \ge K > \sum_{k=1}^K g_r(p_k)$.
Moreover, for $r\ne 1$, $f_r$ is either strictly convex or strictly concave on $(0,1)$ satisfying \eqref{eq:f-cond},
and hence $F^*_r$ is admissible by Theorem \ref{th:admissible}.
The statement $F^*_r=F^*_{r,K}$ is analogous to the case $r<1/(K-1)$.
\end{proof}

As seen from the proof of Theorem \ref{th:m1}, the calibrator $f_r$ of $F_{r,K}^*$ is given by 
\begin{align*}
  x\mapsto K \left(\frac{x^r - d_r^r}{c_r^r - d_r^r} \wedge 1\right)_+
    &\qquad\text{if $r<1/(K-1)$ and $r\ne 0$}; \\
  x\mapsto K \left(\frac{\log x - \log d_0}{\log c_0 - \log d_0} \wedge 1\right)_+
    &\qquad\text{if $r= 0$}; \\
  x\mapsto K\id_{\{x=0\}} + \frac{r+1}{r} (1-x^r)_+
    &\qquad\text{if $r \in [1/(K-1),K-1)$}.
\end{align*}
\begin{remark}\label{rem:M-calibrators}
  Although in different disguises, the harmonic${}^*$ calibrator $f:=f_{-1}$ of $F_{-1,K}^*$
  (which we refer to as the \emph{harmonic${}^*$ p-merging function})
  and the grid harmonic calibrator \eqref{eq:Hkf} are remarkably similar:
  on the set $\{x>0: 0<f(x)<K\}$, one of them takes the form $f(x)= a/x-b$,
  and the other one takes the form $f(x) = a/\lceil bx\rceil$ for some suitably chosen values of $a,b>0$.
  In other words, the calibrator of $F_{-1,K}^*$ can be seen as a continuous version of that of $H_K^*$.
  Both calibrators are shown in Figure~\ref{fig:Hommel}.
  In Section \ref{sec:9}, we shall see that $F^*_{-1,K}$ and $H_K^*$ perform similarly in our simulation experiments. 
\end{remark}

To approximate $F_{r,K}^*$, we can apply Algorithm~\ref{alg:generic} to the calibrator $f_r$.
Remember that this algorithm computes an upper bound that approximates the true value with accuracy $\delta$
in time $O(-K\log\delta)$.

In the next proposition, we give an explicit formula for $F_{r,K}^*$ in Theorem \ref{th:m1}. 
In what follows, $p_{(1)},\dots,p_{(K)}$ are always the order statistics of components of $\mathbf p$, from the smallest to the largest,
and $\mathbf p_m: =(p_{(1)},\dots,p_{(m)})$ is the vector of the $m$ smallest components of $\mathbf p$.

\begin{proposition}\label{th:m2}
For $K\ge 3$ and $\mathbf p\in [0,\infty)^K$, we have, if $r\in(-\infty, 1/(K-1)) $,
\begin{equation}\label{eq:F-rep}
  F^*_{r,K} ( \mathbf p) = \left(\bigwedge_{m=1}^{K} \frac{ M_{r,m} (\mathbf p_m)}{ M_{r,m} (c_r,d_r,\dots,d_r )}\right)\wedge \id_ {\{ p_{(1)} > 0 \} },  
\end{equation}
and, if $r\in [1/(K-1), K-1)$, with the convention $\cdot/0=\infty$,
\begin{equation} \label{eq:arith} 
  F^*_{r,K}(\mathbf p)
  =
  \left(  \bigwedge_{ m = 1 }^K  \frac{ M_{r,m} (\mathbf p_m) }{(1 - \frac{r K}{(r+1)m})_+} \right) \wedge \id_{\{ p_{(1)} > 0 \}}.
\end{equation}  
\end{proposition}

Proposition~\ref{th:m2} allows us to compute $F_{r,K}^*(\mathbf{p})$ in time $O(K\log K)$.
This is the time needed for sorting the elements of $\mathbf{p}$;
the rest of the computations takes time $O(K)$ since $M_{r,m+1}(\mathbf p_{m+1})$
can be computed from $M_{r,m}(\mathbf p_m)$ in time $O(1)$,
for any $m\in\{1,\dots,K-1\}$.

The remaining functions $F_{r,K}$ for $r \ge K-1$ are all strictly dominated by the maximum merging function $F_{\infty,K}$,
which will be discussed in Proposition \ref{prop:grand-3} below.
To summarize,
except for the Bonferroni and the maximum p-merging functions,
any other member of the M-family is not admissible among homogeneous symmetric p-merging functions.
Nevertheless, for $r<K-1$,
a simple modification in \eqref{eq:M-equiv} leads to admissible p-merging functions based on the generalized mean,
which has a stronger power than the original members of the M-family.

The (in-)admissibility of $F_{r,K}^*$ for $r = 1$ cannot be studied via Theorem \ref{th:admissible}
since the calibrator is neither strictly convex or strictly concave.
A discussion of the technical challenges in this special case is provided in Remark \ref{rem:strict-con}
in Supplemental Article, Section~\ref{app:a8}.

\subsection{Domination structure within the M-family}\label{sec:mfamily}

Next, we study the domination structure within the M-family of p-merging functions $F_{r,K}$,
which are generally not admissible.
It turns out that most members of the family are not comparable;
however, for $K=2$ or large $r$, there are some domination relationships among the members in the family. 
We note that $M_{s,K}$ and $M_{r,K}$ for $r\ne s$ are not proportional to each other,
and hence the relations of domination among members of the M-family are all strict.

The following proposition gives a simple comparison for $a M_{r,K}$ and $b M_{s,K}$,
where $a,b$ are two positive constants, e.g., $a=b_{r,K}$ and $b=b_{s,K}$.
Using this result, we can compare two p-merging functions that are not precise (but perhaps have simpler forms),
such as the asymptotically precise p-merging functions in \citet{Vovk/Wang:2020Biometrika}.

\begin{proposition}\label{prop:grand-2}
  For $r<s$, $K\ge 2$, and $a,b\in (0,\infty)$, the following statements hold.
  \begin{enumerate}[(i)]
  \item
    $a M_{r,K}$  dominates $b M_{s,K}$ 
    if and only if $a \le b$.
  \item
    $b M_{s,K}$ dominates $a M_{r,K}$ 
    if and only if $rs>0$ and $a K^{-1/r} \ge b K^{-1/s}$.
  \end{enumerate}
\end{proposition}

Proposition \ref{prop:grand-2} immediately implies that the asymptotically precise p-merging functions ($K\to\infty$)
in Table~1 of \citet{Vovk/Wang:2020Biometrika}
do not dominate each other.

\begin{proposition}\label{prop:grand-3}
 Suppose $r\ne s$.
 If $K = 2$, 
 $F_{r,K}$ is dominated by $F_{s,K}$ if and only if  $1\le r<s$
 or $s< r\le 1$.  
 If $K\ge 3$, 
 $F_{r,K}$ is dominated by $F_{s,K}$
 if and only if $K-1\le r<s  $. 
\end{proposition}

As a consequence of Proposition \ref{prop:grand-3},
in addition to $F_{\infty,K}$, the members $F_{r,K}$ for $r<K-1$ are admissible within the M-family if $K\ge 3$, and the members for $r\in [K-1,\infty)$ are not.
In the simple case $K=2$,
the only two admissible members in the M-family are $F_{-\infty,2}$ and $F_{\infty,2}$,
and the arithmetic average $F_{1,2}$ is the worst, as it is strictly dominated by every other member of the M-family.

\section{Magnitude of improvement}
\label{sec:improve}

By focusing on some of the most important cases,
in the following proposition
(proved in Supplemental Article, Section~\ref{app:ratio})
we calculate four ratios measuring the improvement of the dominating p-merging functions
over the standard ones in Theorems~\ref{th:o1} and~\ref{th:m1}.

\begin{proposition}\label{pr:ratio}
  For $K\ge 3$, we have
  \begin{multline*}
    \inf_{\mathbf p\in(0,1]^K} \frac{ F^*_{1,K}(\mathbf p)}{F_{1,K}(\mathbf p)}
    =
    \inf_{\mathbf{p}\in(0,1]^K} \frac{F^*_{0,K}(\mathbf{p})}{ F_{0,K}(\mathbf{p})}
    =
    0,
    \qquad
    \inf_{\mathbf p\in(0,1]^K} \frac{F^*_{-1,K}(\mathbf p)}{F_{-1,K}(\mathbf p)} = 1-(K-1)c_{-1},\\
    \min_{\mathbf{p}\in(0,1]^K}
    \frac{H^*_K(\mathbf{p})}{H_K(\mathbf{p})}
    =
    \min \left\{t>0: \sum_{k=1}^K \frac{\id_{\{t \ge k/K\}}} {\lceil k/t\rceil}\ge 1\right\}
    =:
    \gamma_K.
  \end{multline*}
  Moreover, $c_{-1}\sim 1/(K\log K)$ and $\gamma_K \sim 1/\log K$ as $K\to\infty$.
\end{proposition}

In Proposition \ref{pr:ratio},
there is a sharp contrast between the greatest improvement of $F^*_{-1,K}$ and that of $H^*_K$
over their standard counterparts:
asymptotically as $K\to\infty$,
$F^*_{-1,K}$ can improve $F_{-1,K}$ only by a factor of $1-1/\log K\to1$,
while $H^*_K$ can improve $H_K$ by a significant factor of $1/\log K\to0$.
This observation is interesting especially seeing that $H_K$ and $F_{-1,K}$ perform similarly in simulation scenarios
(see, e.g., the simulation studies in \citet{Wilson:2020} and \citet{Chen:2020}).
Moreover, since $H_K=\ell_K S_K$ and
$
  \gamma_K\sim 1/\log K \sim 1/\ell_K
$,
$H^*_K$ performs similarly to the Simes function $S_K$ for some input p-values $\mathbf p$,
e.g., those with order statistics close to $(1,\dots,K)$ times a constant
(as can be seen from \eqref{eq:Hk-improve} in Supplemental Article),
a situation that likely happens if the p-values are generated iid from a flat density around $0$.
This is remarkable as we see in Theorem \ref{thm:Simes} that all symmetric p-merging functions are dominated by $S_K$.
See also the numerical illustrations in Section \ref{sec:9}.

\section{Simulation results}
\label{sec:9}

\begin{figure}
  \begin{center}
  \includegraphics[width=0.60\textwidth]{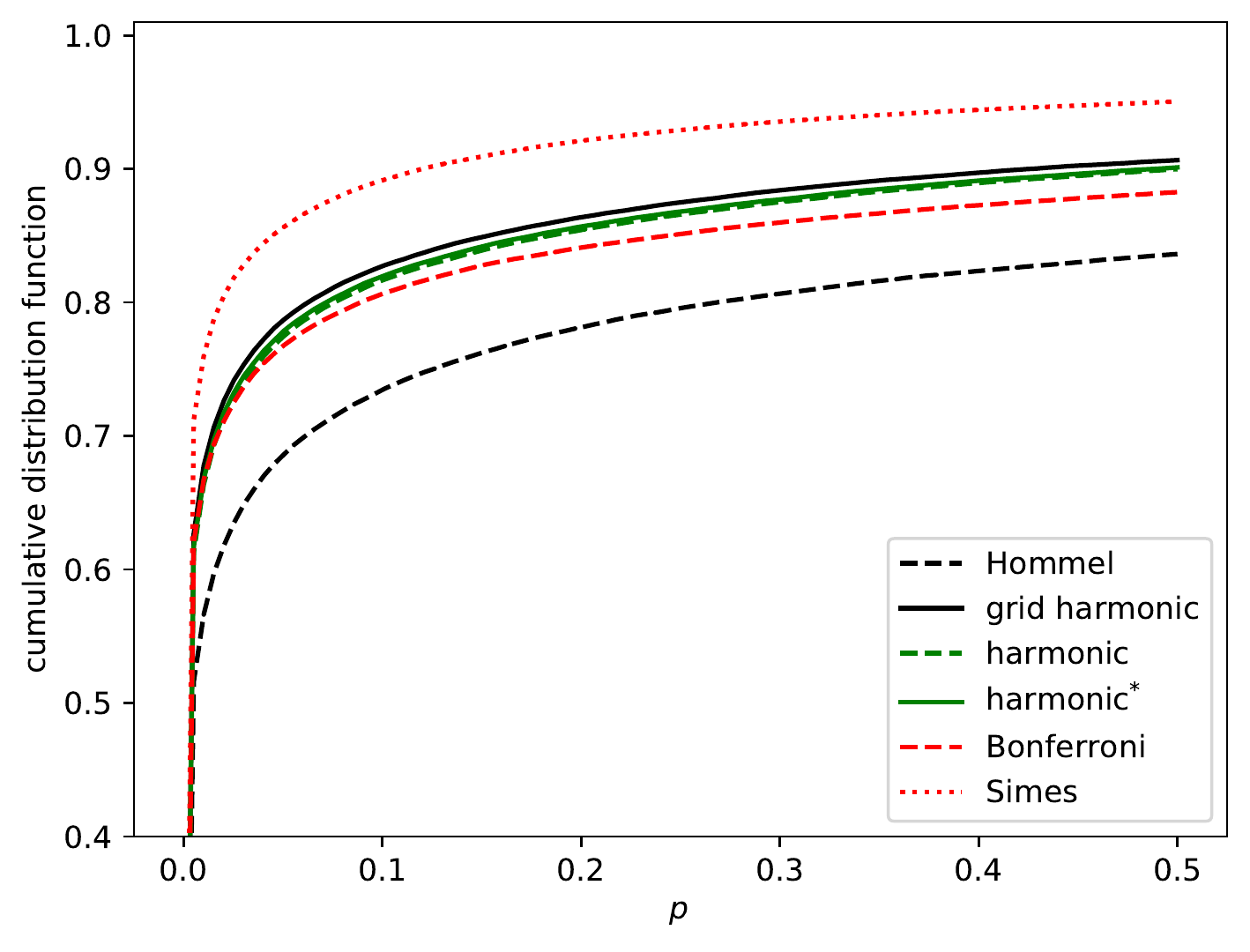}
  \end{center}
  \caption{Cumulative distribution functions of $F(P_1,\dots,P_K)$ for correlated z-tests
    (with correlation $0.9$ between the vast majority of observations).}
  \label{fig:cdf}
\end{figure}

In this section, we compare the performance of p-merging functions via simulation.
First, as a simple illustration,
in Figure \ref{fig:cdf} we plot the cumulative distribution functions of $F(P_1,\dots,P_K)$,
where $F$ is one of $H_K$, $H_K^*$, $F_{-1,K}$, $F_{-1,K}^*$, Bonferroni, or $S_K$.
The Simes function $S_K$ is used as a lower bound because it is the minimum of all symmetric p-merging functions (Theorem \ref{thm:Simes}).
The random variables $P_1,\dots,P_K$ are generated following \citet[Section 8]{Vovk/Wang:E},
essentially using correlated z-tests.
Overall we generate $K=10^6$ observations $x$ from the Gaussian models $N(\mu,1)$
in such a way that the correlation between any pair of observations is $0.9$
(the correlation $0.9$ is chosen for a better visibility of the comparison;
other choices of the correlation give qualitatively similar results,
except for the Bonferroni function, which performs better for small correlations when testing the global null;
see Section~\ref{app:simulation} in Supplemental Article).
An exception is the last observation, whose correlation with the other observations is $-0.9$.
This violates the standard $\text{MTP}_2$ assumption \citep{sarkar1998some},
and so the application of the Simes test is not justified.
(It is not justified anyway unless we \emph{know} that $\text{MTP}_2$ holds;
such knowledge is rare in practice.)

The null hypotheses are $N(0,1)$ and the alternatives are $N(-5,1)$.
First we generate $K_1=10^3$ observations from the alternative distribution $N(-5,1)$
and then $K_0:=K-K_1$ observations from the null distribution $N(0,1)$.
As the base p-values we take
$
  P(x)
  :=
  N(x)
$,
where $N$ is the standard Gaussian distribution function.
The empirical cumulative distribution function of $F(P_1,\dots,P_K)$ is computed via an average of $10^5$ independent simulations. 
A larger cumulative distribution function indicates greater power.

The Bonferroni and Hommel methods appear the worst and, of course, Simes is the best
(but remember that it is not a valid method in our context).
The other methods are roughly midway between these two.
We can hardly distinguish between $F_{-1,K}$ and $F^*_{-1,K}$,
but the grid harmonic method $H^*_K$ performs somewhat better.
In agreement with Proposition~\ref{pr:ratio},
the improvement of $H^*_K$ over $H_K$
is much more significant than the improvement of $F^*_{-1,K}$ over $F_{-1,K}$.
Additional simulation results for discrete p-values are included
in Section~\ref{app:simulation} of Supplemental Article.

Next let us see what our procedures give for multiple hypothesis testing.
We will use a general procedure of \citet{Genovese/Wasserman:2004} and \citet{Goeman/Solari:2011},
which we shall refer to as the GWGS procedure,
and see how the new p-merging functions improve the performance over the classic ones.

Let $F=(F_k)_{k=1}^K$ be a family of symmetric p-merging functions,
which will be chosen from the ones presented in Figure \ref{fig:cdf}.
Each $F_k$ is a function of $k$ p-variables,
defined in the same way as its counterpart in Figure \ref{fig:cdf}
but replacing $K$ p-values by $k$ p-values as its input.
For any input p-values $\mathbf{p}=(p_1,\dots,p_K)$ and any non-empty subset $I$ of $\{1,\dots,K\}$,
we will write $F_{\mathbf{p}}(I)$ for the value of $F_{|I|}$ on a sequence consisting of $\left|I\right|$ elements $p_i$, $i\in I$
(in any order).
With such an $F$ and input p-values $\mathbf{p}$ we associate the array
\begin{equation}\label{eq:D}
  \DM_{l,j}
  :=
  \max_{I:\left|R\setminus I\right|<j}
  F_{\mathbf{p}}(I),
  \quad
  l\in\{1,\dots,K\},
  \quad
  j\in\{1,\dots,l\},
\end{equation}
where $R\subseteq\{1,\dots,K\}$ is a set of indices of $l$ smallest p-values among $p_1,\dots,p_K$
(such a set $R$ may not be unique if there are ties among $p_1,\dots,p_K$,
but $\DM_{l,j}$ does not depend on the choice of $R$).
We regard $\DM$ as a $K\times K$ matrix whose elements above the main diagonal are undefined
and call it the \emph{(GWGS) discovery matrix};
this is our representation of the GWGS procedure.
A small value of $\DM_{l,j}$ is evidence for the statement
``there are at least $j$ true discoveries among the $l$ hypotheses (with the smallest p-values) that we choose to reject'';
namely, $\DM_{l,j}$ is a valid p-value for testing the negation of this statement.
These p-values are jointly valid in the sense that,
for each confidence level $1-\alpha$,
with probability at least $1-\alpha$,
the maximum number $j$ satisfying $\DM_{l,j} \le \alpha$ is a lower bound on the number of true discoveries among $l$ smallest p-values
for all $l$ simultaneously.

\begin{algorithm}[bt]
  \caption{Discovery matrix}
  \label{alg:DM}
  \begin{algorithmic}[1]
    \Require
      A  family  of merging functions $F$.
    \Require
      An increasing sequence $\mathbf{p}$ of p-values $p_1\le\dots\le p_K$.
    \For{$l=1,\dots,K$}
      \For{$j=1,\dots,l$}
        \State $S_{j,l}:=\{j,\dots,l\}$
        \State $\DM'_{l,j}:=F_{\mathbf{p}}(S_{j,l})$
        \For{$i=K,\dots,l+1$}
          \State $p := F_{\mathbf{p}}(S_{j,l}\cup\{i,\dots,K\})$
          \If{$p > \DM'_{l,j}$}
            \State $\DM'_{l,j} := p$
          \EndIf
        \EndFor
      \EndFor
    \EndFor
  \end{algorithmic}
\end{algorithm}

See the recent paper \citet{Goeman/etal:arXiv1901} for an interesting justification
of the GWGS procedure (it is the only admissible, in some sense, procedure
with the true discovery guarantee).
The goal of the GWGS procedure is somewhat similar to that of the partial conjunction test
(see, e.g., \citet{Wang/Owen:2019})
looking for evidence that at least $j$ out of $l$ null hypotheses are false.
The difference is that a GWGS matrix is jointly valid for all $j$ and $l$ (as described earlier),
and the $l$ null hypotheses are those with the smallest p-values.

Algorithm~\ref{alg:DM} computes the modification
\begin{equation*}
  \DM'_{l,j}
  :=
  \max_{I:\left|R\setminus I\right|=j-1}
  F_{\mathbf{p}}(I),
  \quad
  l\in\{1,\dots,K\},
  \quad
  j\in\{1,\dots,l\},
\end{equation*}
of the discovery matrix~\eqref{eq:D}.
It assumes, without loss of generality, that the input p-values are given in the increasing order.
We will usually have $\DM=\DM'$, but unlike $\DM_{l,j}$,
the function $\DM'_{l,j}$ does not need to be monotonically increasing in $j$.
(The monotonicity may be violated when, e.g., $F$ represents the Bonferroni p-merging functions.)
But even in such unusual cases it is always true that
$
  \DM_{l,j}
  =
  \max_{j'\le j}
  \DM'_{l,j'}
$.

\begin{figure}
    \includegraphics[width=0.32\textwidth]{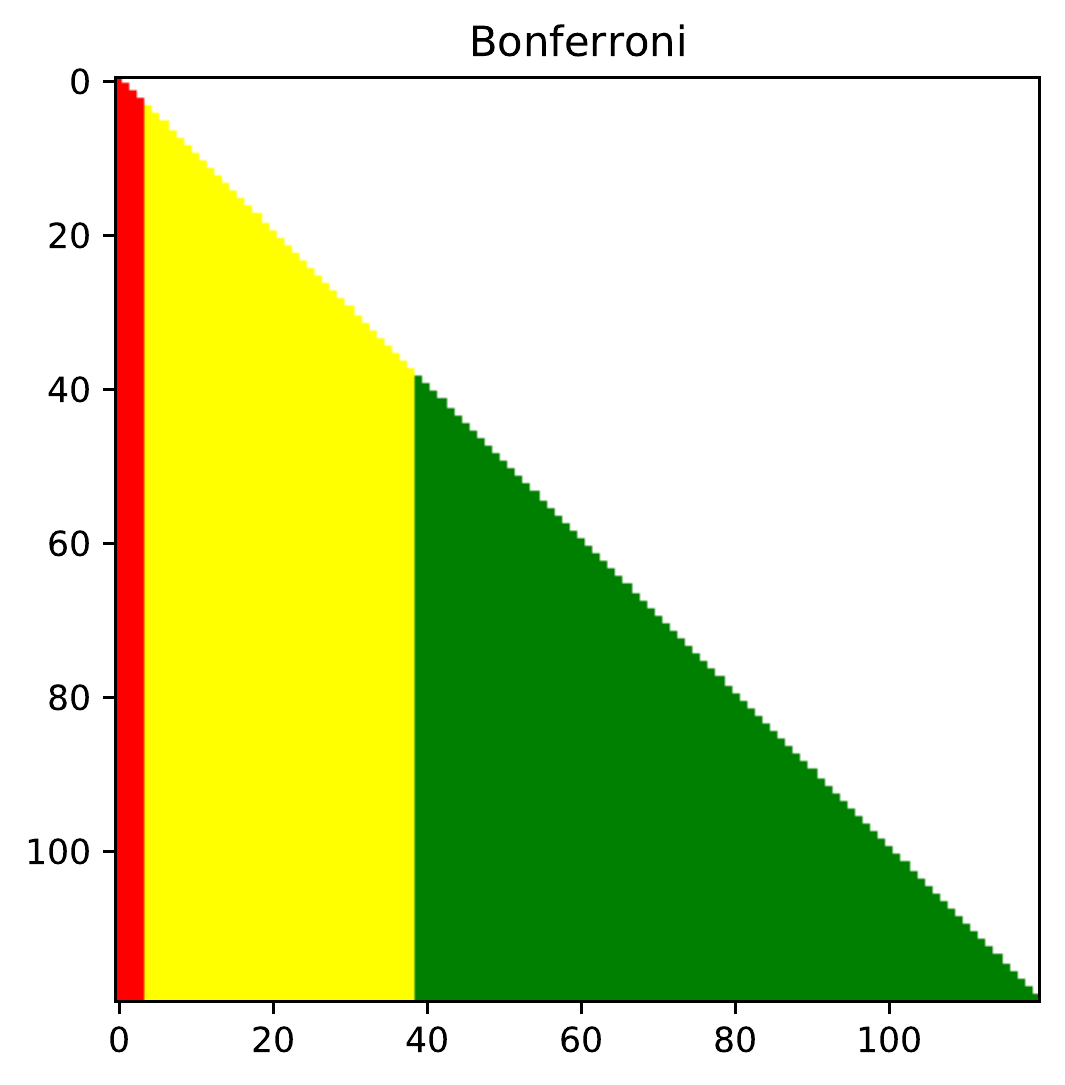}
    \includegraphics[width=0.32\textwidth]{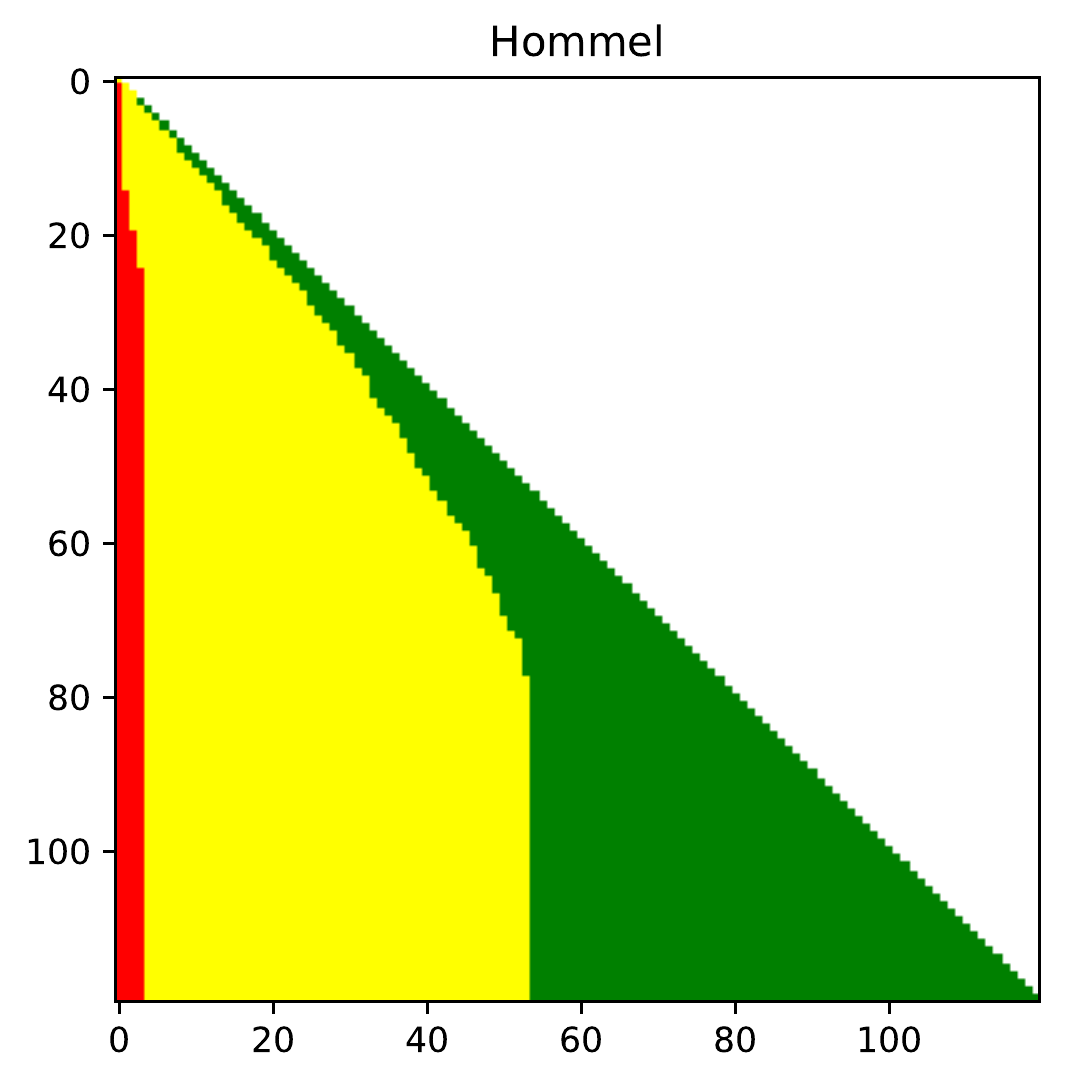}
    \includegraphics[width=0.32\textwidth]{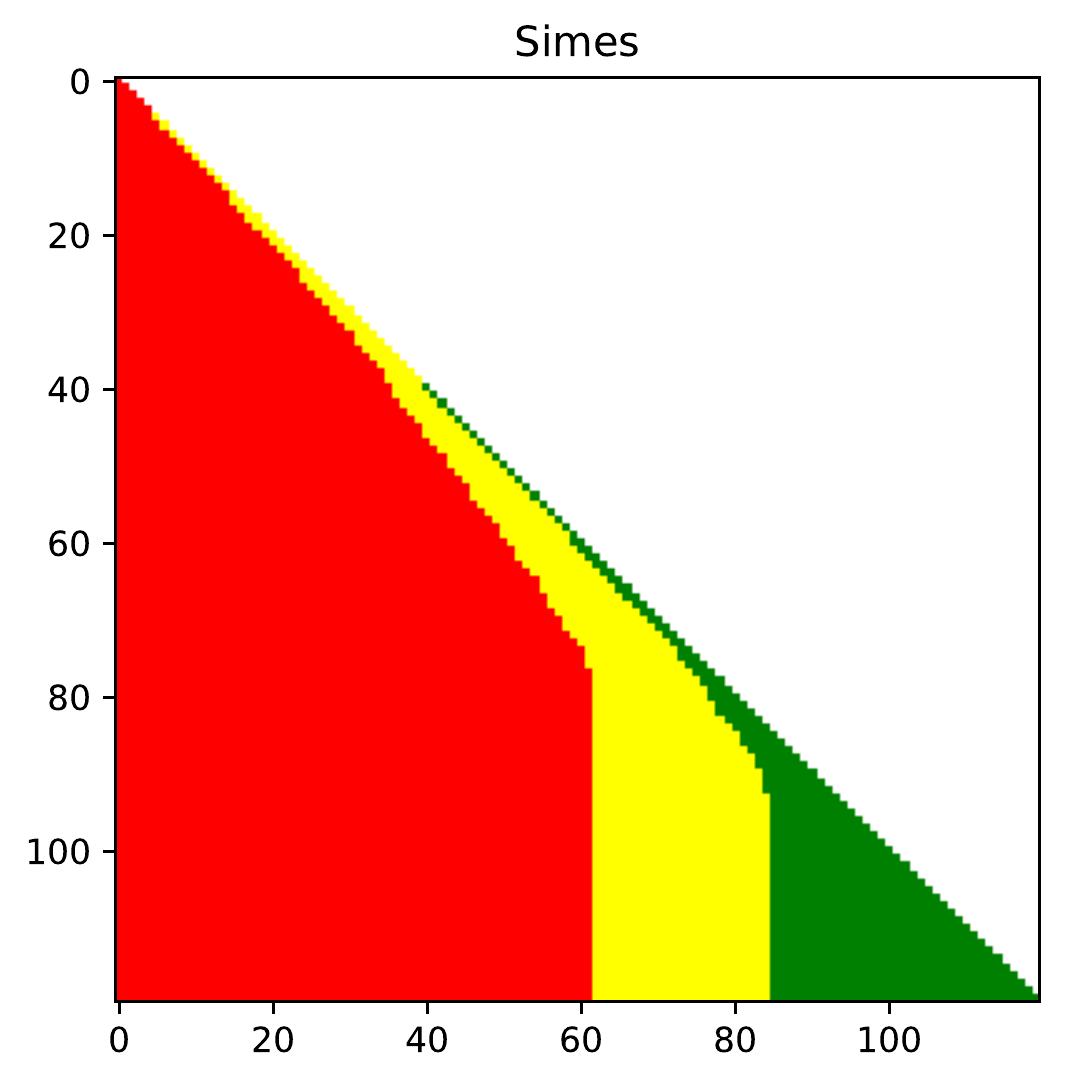} \\
    \includegraphics[width=0.32\textwidth]{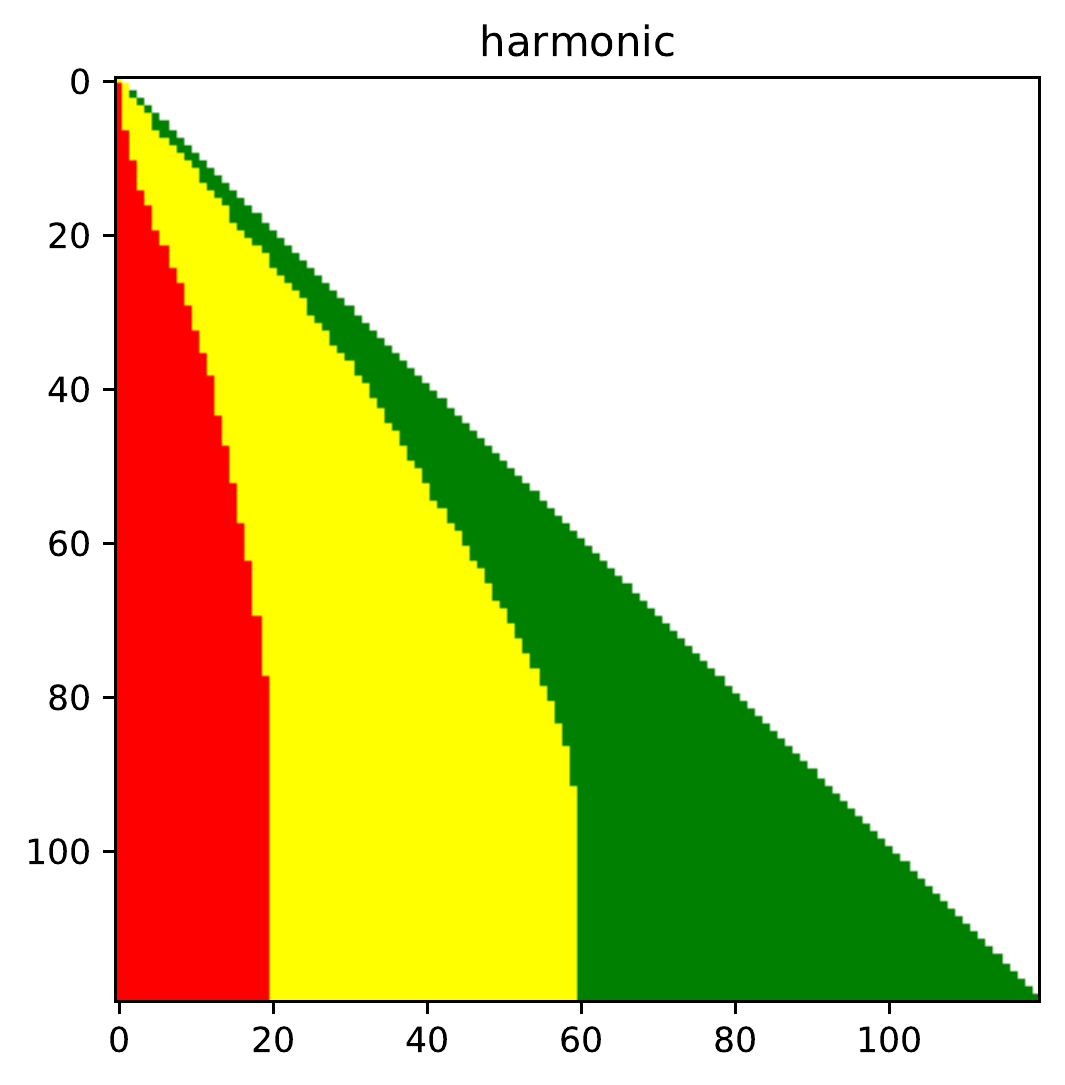}
    \includegraphics[width=0.32\textwidth]{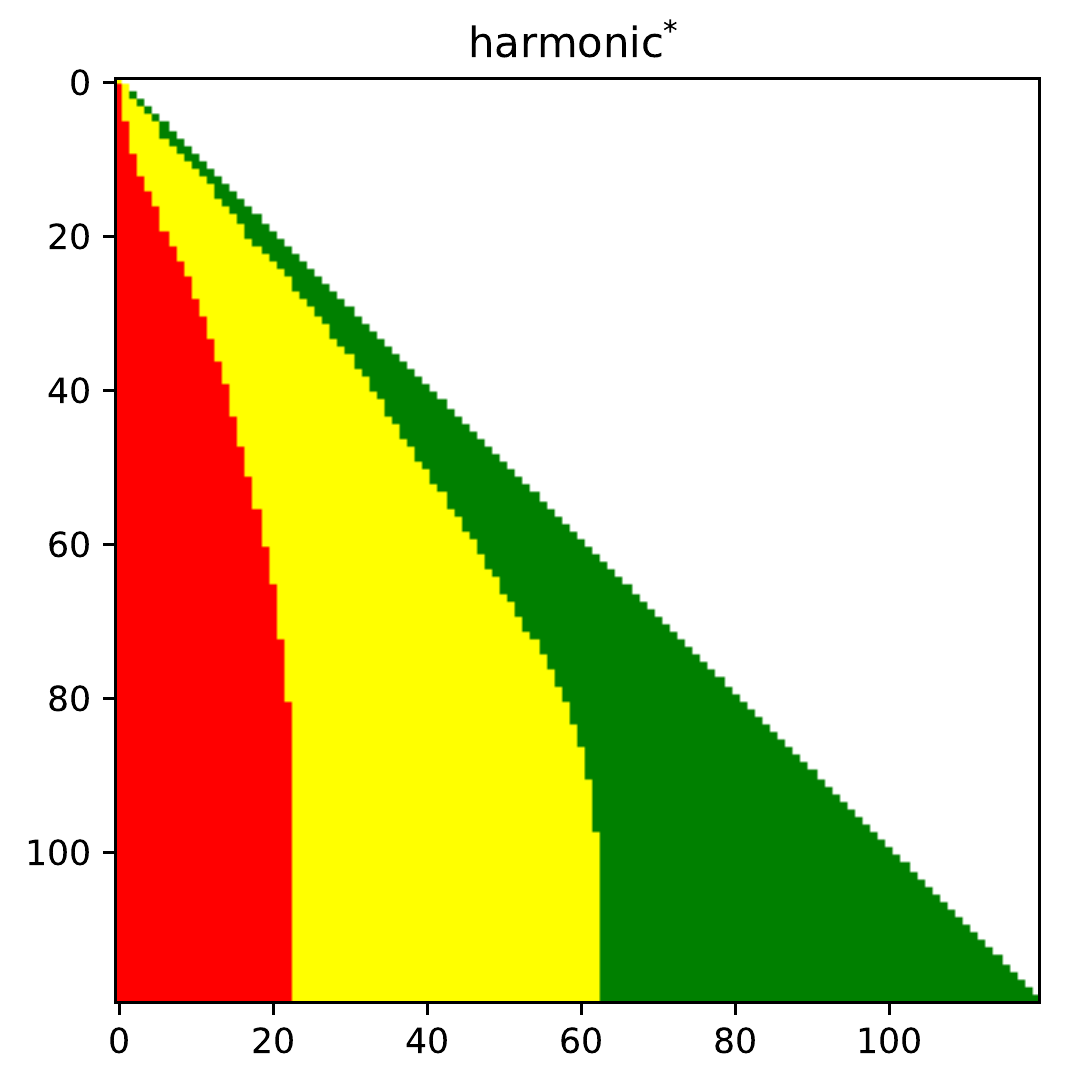}
    \includegraphics[width=0.32\textwidth]{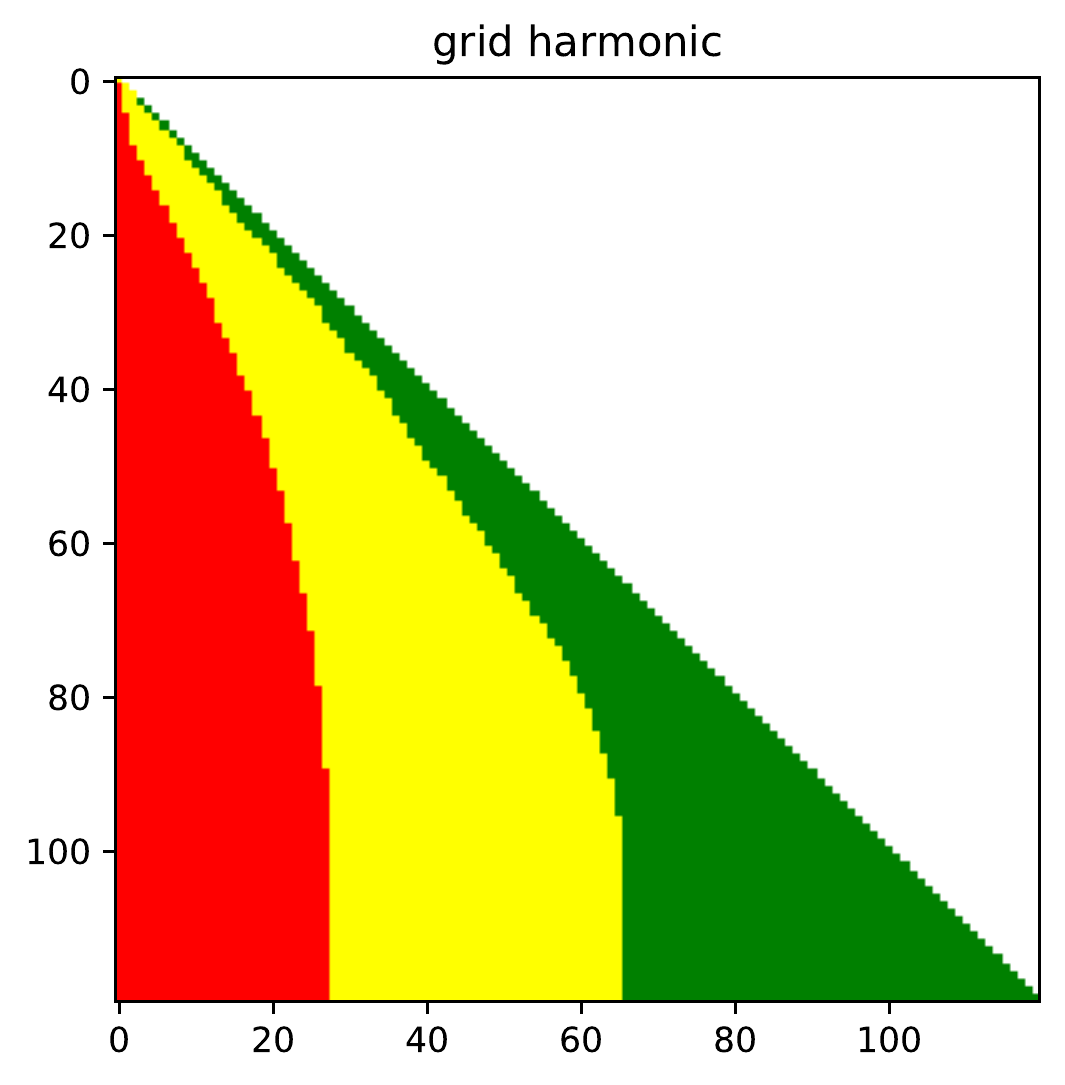}
  \caption{The GWGS discovery matrices for the simulation data using significance levels $1\%$ and $5\%$.
    We give results for the p-merging functions
    $F_{-\infty,K}$ (``Bonferroni''), $H_K$ (``Hommel''), $S_K$ (``Simes''),
    $F_{-1,K}$ (``harmonic''), $F^*_{-1,K}$ (``harmonic$^*$''), and $H^*_K$ (``grid harmonic'').}
  \label{fig:DM}
\end{figure}

Figure~\ref{fig:DM} shows the upper left corners of size $120\times120$ of the discovery matrices
produced by six of the p-merging functions
considered in this paper for the p-variables $P_1,\dots,P_{1000}$ defined as before
with the first $100$ observations coming from the alternative distribution $N(-5,1)$
and the remaining $900$ from the null distribution $N(0,1)$.
It uses the standard significance levels $1\%$ and $5\%$ as thresholds;
the values in the discovery matrices below $1\%$ are shown in red,
between $1\%$ and $5\%$ in yellow, and above $5\%$ in green.
As explained above, the number of red entries in the $l$th row of the discovery matrix
is a lower bound on the number of true discoveries among $l$ smallest p-values at the confidence level $99\%$,
and the total number of red and yellow entries in the $l$th row
is the analogous lower bound at the confidence level $95\%$.

The upper row of plots in Figure~\ref{fig:DM} shows the results for three standard methods,
and the lower row for three new methods.
The two of the standard methods that are universally valid, Bonferroni and Hommel, perform worst.
Harmonic averaging leads to better results.
The results for $F^*_{-1,K}$ are better, but the difference is not substantial.
The best results for a universally valid method are achieved by the grid harmonic merging function $H^*_K$.
The results for the Simes merging function $S_K$ are, of course, even better
(in view of Theorem \ref{thm:Simes}),
but $S_K$ is not valid in our setting.

Discovery matrices depend very much on the seed used for the pseudo-random number generator,
especially for high correlations (such as $0.9$ used in Figure~\ref{fig:DM}).
To make our results more reproducible, the discovery matrices in Figure~\ref{fig:DM}
are in fact element-wise medians over 10 simulations.
For other correlation coefficients, we obtain qualitatively similar results;
see Section \ref{app:simulation} in Supplemental Article.

\section{Concluding remarks}
\label{sec:10}

In this paper, we establish a representation and some conditions for admissible p-merging functions via calibrators.
Several new p-merging functions, most notably $H_K^*$ and $F_{-1,K}^*$, are proposed and shown to be admissible.
As  seen from our main results and their proofs, admissibility of p-merging functions is a sophisticated object.

We mention a few open questions.
First, our study is mainly confined to homogeneous p-merging functions.
The homogeneity requirement in Theorem \ref{th:e} is essential to our proof, and it is unclear whether or how one could relax it.
On the other hand,  most p-merging functions used in practice are homogeneous
(an exception is the Cauchy combination test of \cite{Liu/Xie:2020},
which is not valid for arbitrary dependence and hence does not fit into our setting).
Second, it is unclear how the strict convexity in  Theorem \ref{th:admissible} can be relaxed;
see discussions in Remark \ref{rem:strict-con}.
As a consequence, we suspect, but could not prove, the admissibility of $F_{1,K}^*$ for $K\ge 3$.
This function is not admissible for $K=2$; see Example \ref{ex:ex-th:e3}.
Third, we do not know whether $H_K^*$ is always inadmissible for all prime numbers $K$ (see Example \ref{ex:prime} for the cases of $K=2$ and $K=3$).
Fourth, an admissible p-merging function dominating a given p-merging function is typically not unique.
We wonder whether there are other admissible p-merging functions which dominate $H_K$ and $F_{-1,K}$,
the two most important inadmissible p-merging functions, that have analytical formulas as well as superior statistical performance.
Finally, it is important to develop more efficient ways of computing $H^*_K$;
in our simulation studies we used a brute-force method based on Algorithm~\ref{alg:generic}.

  \subsection*{Author contributions}

  The author names are listed in the alphabetical order.
  The main mathematical results are due to Bin Wang and Ruodu Wang.
  Vladimir Vovk has contributed to the presentation and computational experiments.

\subsection*{Acknowledgments}

  The authors thank the Editor, an Associate Editor, and three anonymous referees of the journal version of this paper for very helpful comments.
  V.~Vovk's research has been partially supported by Amazon, Astra Zeneca, and Stena Line.
  R.~Wang is supported by the Natural Sciences and Engineering Research Council of Canada (RGPIN-2018-03823, RGPAS-2018-522590).

\appendix
  \begin{center}
   {\huge Supplemental Article}
  \end{center}
  \addcontentsline{toc}{section}{Supplemental Article}

\section{Technical details}

\subsection{Proofs of Propositions \ref{prop:precise-p}, \ref{prop:lsc}, \ref{prop:limit} and \ref{prop:dominated}}
\label{app:a2}
 
\begin{proof}[Proof of Proposition \ref{prop:precise-p}]
Suppose that  $F$ is an admissible p-merging function 
and there exists $b\in (0,1)$ such that 
\[
  a:=\sup_{\mathbf P\in \mathcal P_Q^K} Q(F(\mathbf P) \le b) < b.
\]
Define the increasing function $h: [0,\infty) \to [0,\infty)$ by $h(x) := a \id_{\{x\in[a,b]\}}+x\id_{\{x\notin[a,b]\}}$.
We can check, for $t\in [a,b]$,
\begin{align*}
  \sup_{\mathbf P\in \mathcal P_Q^K} Q(h \circ F(\mathbf P) \le t)
  =
  \sup_{\mathbf P\in \mathcal P_Q^K} Q(F(\mathbf P)\le b) = a\le t,
\end{align*}
and for $t\notin[a,b]$, 
\begin{align*}
  \sup_{\mathbf P\in \mathcal P_Q^K}
  Q(h \circ F(\mathbf P) \le t ) 
  =
  \sup_{\mathbf P\in \mathcal P_Q^K}
  Q(F(\mathbf P) \le t) \le t.
\end{align*} 
Hence, $h\circ F$ is a p-merging function.
The fact that $(h\circ F)\wedge 1$ strictly dominates $F\wedge 1$ contradicts the admissibility of $F$.
Therefore, we obtain
$\sup_{\mathbf P\in \mathcal P_Q^K} Q(F(\mathbf P) \le t) \ge t$, $t\in (0,1)$.
Together with the fact that $F$ is a p-merging function,
we have
\[
  \sup_{\mathbf P\in \mathcal P_Q^K} Q(F(\mathbf P) \le t) = t,
  \quad
  t\in (0,1),
\]
and thus $F$ is precise.  
\end{proof}

\begin{proof}[Proof of Proposition \ref{prop:lsc}]
  Fix $\mathbf{P}=(P_1,\dots,P_K)\in\mathcal P^K_Q$ and $\alpha \in(0,1)$,
  and we will first show $Q(F'(\mathbf P) \le \alpha) \le \alpha$.
  For every $\lambda \in (0,1)$, let $A_\lambda$ be an event independent of $\mathbf{P}$ with $Q(A_\lambda)=\lambda$
  and define the random vector $\mathbf{P}^{\lambda} =(P^\lambda _1,\dots,P_K^\lambda)$
  via $\mathbf{P}^{\lambda} = \lambda \mathbf P $ if $A_\lambda$ occurs,
  and $ \mathbf{P}^{\lambda} = (1,\dots,1)$ if  $A_\lambda$ does not occur.
  For all $\lambda\in (0,1)$ and $k=1,\dots,K$, noting that 
  $Q(P_k\le \alpha/\lambda) \le \alpha/\lambda$, 
   we have
   \[Q(P^{\lambda}_k \le \alpha) = \lambda Q(\lambda P_k\le \alpha) = \lambda Q(P_k\le\alpha/\lambda) \le \alpha.\]
  Thus, $\mathbf{P}^\lambda \in \PPP^K_Q$  and by the fact that $F$ is a p-merging function,
  we have $Q(F(\mathbf P^\lambda) \le \alpha) \le \alpha$.
  Note that
  \[
    Q(F(\mathbf P^\lambda) \le \alpha)
    \ge
    Q(A_\lambda) Q(F(\mathbf P^\lambda) \le \alpha |A_\lambda) = \lambda Q(F(\lambda \mathbf P ) \le \alpha),
  \]
  from which we obtain
  \[Q(F(\lambda \mathbf P ) \le \alpha) \le \frac{\alpha}\lambda.\]
Since $F$ is increasing, by \eqref{eq:lsc}, we have
$F'(\mathbf P) \ge F(\lambda \mathbf P) $ for all $\lambda \in (0,1)$.
Therefore,
\[
Q(F'(\mathbf P) \le \alpha) \le Q(F(\lambda \mathbf P ) \le \alpha)\le \frac{\alpha}\lambda.
\]
Since $\lambda \in (0,1) $ is arbitrary, we have
$Q(F'(\mathbf P) \le \alpha) \le \alpha$,
thus showing that $F'$ is a p-merging function.  

For the statement on $\tilde F$, it is clear that
\[Q\left(\mathbf P \in [0,1]^K\setminus (0,1]^K \right) = Q\left(\bigcup_{k=1}^K \{P_k = 0\}\right) \le \sum_{k=1}^K Q(P_k=0)= 0.\]
Therefore, the values of $F$ on $[0,1]^K\setminus (0,1]^K$ do not affect its validity as a p-merging function.

  To show the last statement, let $F$ be an admissible p-merging function.
  Using the above results, we obtain that $F'\le F$ is a p-merging function.
  Admissibility of $F$ forces $F=F'$,
  implying that $F$ is lower semicontinuous.
  Similarly, $F=\widetilde F$, implying that $F$ takes value $0$ on $[0,1]^K\setminus (0,1]^K$. 
\end{proof}

\begin{proof}[Proof of Proposition \ref{prop:limit}]
Let $(F_n)_{n\in \N}$ be a sequence of  p-merging functions 
which converges to its point-wise limit $F$. 
For any $\mathbf P =(P_1,\dots,P_K)\in \PPP_Q^K$,
we know that $F_n(\mathbf P) \to F (\mathbf P)$ in distribution. 
Using the Portmanteau theorem,
we have  for all $\alpha \in (0,1)$, 
\[Q (F (\mathbf P) <\alpha) \le \liminf_{n\to \infty} Q(F_n(\mathbf P) <\alpha) \le \alpha.\]
It follows that for any $\epsilon>0$ and $\alpha\in (0,1)$, 
\[Q (F (\mathbf P) \le \alpha) \le \alpha+\epsilon.\]
Since $\alpha$ and $\epsilon$ are arbitrary, we know that $F (\mathbf P)$ is a p-variable, and 
 $F$ is a p-merging function.
\end{proof}

\begin{proof}[Proof of Proposition \ref{prop:dominated}]
  Let $R$ be the uniform probability measure   on $[0,1]^K$. 
  Fix a p-merging function $F$. 
  Set $F_0:=F$ and let
  \begin{equation}\label{eq:c}
    c_i
    :=
    \sup_{G:G\le F_{i-1}}
    \int_0^1   R (G\le \epsilon) \d \epsilon,
  \end{equation}
  where $i:=1$ and $G$ ranges over all p-merging functions dominating $F_{i-1}$.
  Let $F_i$ be a p-merging function satisfying
  \begin{equation}\label{eq:f}
    F_i\le F_{i-1}
    \quad\text{ and }\quad
    \int_0^1   R  (F_i\le \epsilon) \d \epsilon
    \ge c_{i}-2^{-i},
  \end{equation}
  where $i:=1$.
  Continue setting \eqref{eq:c} and choosing $F_i$ to satisfy \eqref{eq:f}
  for $i=2,3,\dots$.
  Set $G:=\lim_{i\to\infty}F_i$.
  By Proposition \ref{prop:limit}, $G$ is a p-merging   function.
  Clearly, $G$ dominates $F$ and
  \[
    \int_0^1 R (G\le \epsilon) \d \epsilon =  \int_0^1 R(H\le\epsilon) \d\epsilon
  \]
  for any p-merging function $H$ dominating $G$.

  By Proposition~\ref{prop:lsc},
  the zero-one adjusted version $\widetilde G$ of $G$ is a p-merging function,
  and so is the lower semicontinuous version $\widetilde G'$ of $\widetilde G$.
  Clearly $\widetilde G'=0$ on $[0,1]^K\setminus(0,1]^K$. 
  Let us check that $\widetilde G'$ is admissible.
  Suppose that there exists a p-merging function $H$ such that $H\le \widetilde G'$ and $H\ne \widetilde G'$ on $(0,1]^K$.
  Fix such an $H$ and a $\mathbf{p}\in(0,1]^K$ satisfying $H(\mathbf{p})<\widetilde G'(\mathbf{p})$. 
  Since $\widetilde G'$ is lower semicontinuous and $H$ is increasing,
  there exists $\lambda \in (0,1)$ such that $H<\widetilde G'$ on the hypercube
  $[\lambda \mathbf{p},\mathbf{p} ]\subseteq[0,1]^K$,
  which has a positive $R$-measure.
  This gives
  \[
    \int_0^1
    R  (G\le \epsilon) \d \epsilon
    \le
    \int_0^1 R(\widetilde G' \le \epsilon) \d\epsilon
    <
    \int_0^1 R(H\le\epsilon) \d\epsilon,
  \]
  a contradiction.
\end{proof}

\subsection{Proof of Proposition~\ref{pr:g} and a lemma used in the proof of Theorem \ref{th:admissible}}
\label{app:a5}

\begin{proof}[Proof of Proposition~\ref{pr:g}]
  We will only show the first statement, as the second one follows from essentially the same proof.
  It suffices to show that $F$ is not admissible among symmetric p-merging functions
  if and only if \eqref{eq:rejrelation} holds for some calibrator $g$. 
  First, if there exists such $g$, then the p-merging function based on the calibrator $g$ strictly dominates $F$.
  Second, if $F$ is not admissible,
  using Proposition \ref{prop:dominated} and Remark \ref{rem:symmetry},
  we know that there exists $G\le F$ that is admissible among symmetric p-merging functions. 
  Note that $G$ can be safely chosen as homogeneous.
  Using Theorem \ref{pr:e}, $G$ is induced by a calibrator $g$.
  Since $G$ strictly dominates $F$, we know that \eqref{eq:rejrelation} holds.
\end{proof}

\begin{lemma}\label{lem:f-trans}
  If the p-merging function induced by a calibrator $f$ is admissible,
  then so is the p-merging function induced by $g$ in \eqref{eq:def-g} for any $\eta\in [0,1/K]$.
\end{lemma}
\begin{proof}[Proof of the lemma]
  The case $\eta=0$ is trivial since $g=f$.
  If $\eta=1/K$, then $g$ induces the Bonferroni p-merging function, which is admissible as shown in Proposition 6.1 of \citet{Vovk/Wang:E}.
  Below we assume $\eta\in (0,1/K)$.
  Let $F$ and $G$ be the p-merging functions  induced by $f$ and $g$, respectively,
  and let $G'$ be a p-merging function dominating $G$.
  Suppose for the purpose of contradiction that there exists $\mathbf p\in [0,\infty)^K$ and  $\alpha \in(0,1)$
  such that $ \alpha  \mathbf p\in R_\alpha(G')$ and $  \alpha \mathbf p\notin R_\alpha(G)$.
  Clearly, no component of $\mathbf p $ can be in $[0,\eta]$,
  and hence   $\mathbf p \in  (\eta,\infty)^K$.
  Let $\mathbf p'= (\mathbf p -\eta \mathbf 1)/(1-K\eta)$.
  By the relationship between $f$ and $g$, we know $\alpha \mathbf p' \notin R_\alpha(F)$.
  Let $A = R_\alpha(F)\cup\{\alpha \mathbf p'\}$.
  Take any vector $\mathbf P$ of p-variables, and let $\nu$ be the distribution of $\alpha ((1-K\eta) \mathbf P + \eta \mathbf 1)$. 
  Further, let $\Pi$ be the set of all permutations of the vector $(\alpha \eta,1,\dots,1)$ and $\mu$ be the discrete uniform distribution over $\Pi$. 
  Clearly, $\Pi \subseteq R_\alpha(G)\subseteq R_\alpha(G')$. 
  Let $\mathbf P'$ follow the distribution
  $( K \eta \alpha )\mu + K\eta(1-\alpha) \delta_{\mathbf 1} + (1-K\eta)\nu$.
  It is easy to verify that the components of $\mathbf P'$ are p-variables. 
  Note that if $\alpha \mathbf P\in  A$, then
  $\alpha((1-K\eta) \mathbf P + \eta \mathbf 1)\in (R_\alpha (G)\cup\{\alpha \mathbf p\})\subseteq R_\alpha (G')$.
  We have
  \begin{align*}
    \alpha \ge \p(\mathbf P'\in R_\alpha(G'))
    &= K\eta \alpha + (1-K\eta)\p(\alpha((1-K\eta) \mathbf P + \eta \mathbf 1 )\in R_\alpha(G))\\
    &\ge
    K \eta \alpha + (1-K\eta)\p(\mathbf P \in A).
  \end{align*} 
  Hence, $\p(\mathbf P \in A) \le \alpha$.
  Since $\mathbf P$ is arbitrary,
  this implies that the rejection region of $F$ at level $\alpha$ can be enlarged to $A$, a contradiction of the admissibility of $F$.
  Therefore, the above $\mathbf p$ does not exist, and $G$ is admissible. 
\end{proof}

\subsection{Proofs of Propositions \ref{prop:grand}, \ref{th:m2}, \ref{prop:grand-2} and \ref{prop:grand-3}}
\label{app:a7}

\begin{proof}[Proof of Proposition \ref{prop:grand}]
The simple case $K=2$ is discussed in Section~\ref{app:K2}, and we assume $K>2$.
The cases $r\ge 1/(K-1)$, $r=-1$ and $r=0$ are obtained in Propositions 3, 4, and 6 of \citet{Vovk/Wang:2020Biometrika},
and are easily obtained from the case $r\notin\{-1,0\}$ by letting $r\to-1$ or $r\to0$, respectively.
It remains to show the remaining cases.
Let $q_0$ and $q_1$ be the essential infimum and the essential supremum of a random variable, respectively,
and $\mathcal U\subset \mathcal P_Q$ be the set of $\mathrm {U}[0,1]$ random variables. 
Note that
\[
  R_\epsilon (F_{r,K}) =\left \{\mathbf p\in [0,\infty)^K:
  M_{r,K}(\mathbf p) \le \frac{\epsilon }{b_{r,K}}\right\} = \epsilon\left \{\mathbf p\in [0,\infty)^K:
  M_{r,K}(\mathbf p) \le \frac{1}{b_{r,K}} \right\}.
\]
From $R_{\epsilon}(F_{r,K})$,
in order for $\sup_{\mathbf P\in \mathcal P_Q^K } \p(\mathbf P \in R_\epsilon (F_{r,K})) = \epsilon$, 
it is necessary and sufficient to choose
\[
  \frac{1 }{b_{r,K}} =\inf_{\mathbf P\in \mathcal P_Q^K }  {q_1( M_{r,K}(\mathbf P))},
\] 
Simple algebra gives, for $r<0$,   
\[
  b_{r,K}^{-1} 
  =  
  \left(
    \frac 1K \sup\{\q_0 (U_1^r+\dots+U_K^r): U_1,\dots,U_K\in\mathcal{U}\}
  \right)^{1/r},
\] 
and for $r>0$,
\[
  b_{r,K}^{-1} 
  =   
  \left(
  \frac 1K
  \inf
  \{
    \q_1(U^r_1+\dots+U_K^r)
    :
    U_1,\dots,U_K\in\mathcal{U}
  \}
  \right)^{1/r}.
\] 
The rest of the proof is a direct consequence of Lemma \ref{lem:newlem} below, which gives, for $r<0$,
\[
  \sup\{\q_0 (U_1^r+\dots+U_K^r): U_1,\dots,U_K\in\mathcal{U}\}   =  (K-1) (1-(K-1)c)^r +c^{r}, 
\]
and for $r\in (0,1/(K-1))$, 
\[ 
  \inf
  \{
    \q_1(U^r_1+\dots+U_K^r)
    :
    U_1,\dots,U_K\in\mathcal{U} 
  \}
  =
  (K-1) (1-(K-1)c)^r +c^{r},
\]
where $c=c_{r}$.
Therefore, $b_{r,K}^{-1} = M_{r,K}(c_r,d_r,\dots,d_r)$.
\end{proof}
 
\begin{lemma}\label{lem:newlem}
  For any increasing convex function $f: [0,1)\to\R$ satisfying
  either $f(1)=\infty $ or $f(1)-f(0) > K \int_0^1 (f(u)-f(0)) \d u$ where $f(1)$ is the limit of $f(x)$ as $x\uparrow 1$,
  we have
  \[
    \sup\{\q_0 (f(U_1) +\dots+f(U_K)): U_1,\dots,U_K\in\mathcal{U}\}
    =
    (K-1) f((K-1)c_F)+f(1-c_F),
  \]
  where $c_F$ is the unique solution $c\in(0,1/K)$ to the following equation
  \begin{equation}\label{eq:general-formula2}
    (K-1) F^{-1}((K-1)c) + F^{-1}(1-c)
    =
    K \frac{\int_{(K-1)c} ^{1-c} F^{-1}(y) \d y}{1-Kc}.
  \end{equation}  
\end{lemma}
\begin{proof}[Proof of the lemma]
  The lemma is essentially Theorem 3.4 of \citet{Wang/etal:2013} applied to the probability level $\alpha=0$,
  noting that any convex quantile function $f$ can be approximated by distributions with a decreasing density.
\end{proof}

\begin{proof}[Proof of Proposition \ref{th:m2}]
We use the calibrators $f_r$ mentioned after Theorem \ref{th:m1}.
We first consider $r<0$.
For $m=1,\dots,K$ and $p_1,\dots,p_K>0$,
let $\mathbf{v}_m:=(c_r,d_r,\dots,d_r)\in\R^m$, and we have 
\[ 
  \sum_{k=1}^m \frac{p_{(k)} ^r - d_r^r}{c_r^r - d_r^r} \ge 1
  ~ \Longleftrightarrow ~
  M_{r,m} (\mathbf{p}_m) \le M_{r,m} (c_r,d_r,\dots,d_r)= M_{r,m} (\mathbf{v}_m  ).
\]
Hence,
\[
  \sum_{k=1}^K \left(\frac{ p_{(k)} ^r - d_r^r}{c_r^r - d_r^r}\right)_+ \ge 1
  ~\Longleftrightarrow~
  \bigvee_{m=1}^K\left(\sum_{k=1}^m \frac{p_{(k)} ^r - d_r^r}{c_r^r - d_r^r} \right)\ge 1 
  ~\Longleftrightarrow~
  \bigwedge_{m=1}^K \frac{M_{r,m}(\mathbf p_m)}{M_{r,m}(\mathbf v_m )}\le 1.
\] 
Using its calibrator $f_r$, for each $\epsilon \in (0,1)$, $F_{r}^* (\mathbf p)\le \epsilon$ 
if and only if $  \bigwedge_{m=1}^K \frac{M_{r,m}(\mathbf p_m)}{M_{r,m}(\mathbf v_m )}\le \epsilon $,
and hence \eqref{eq:F-rep} holds.
The case $r\in [0,1/(K-1))$ is similar. 
 
Next, consider the case $r\ge 1/(K-1)$.
For $m=1,\dots,K$ and $p_1,\dots,p_K>0$, we have 
\[
  \frac 1 K\sum_{k=1}^m \tau^{-1} ( 1-p_{(k)}^r ) \ge 1
  ~\Longleftrightarrow~
  M_{r,m} (\mathbf{p}_m) \le 1 -  \frac{\tau K}{m}.
\]
Hence,  
\[
  \sum_{k=1}^m  f_r(p_{(k)}^r ) \ge  K
  ~\Longleftrightarrow~
  \bigvee_{m=1}^K\left(\sum_{k=1}^m \tau^{-1} ( 1-p_{(k)}^r )_+\right)\ge K 
  ~\Longleftrightarrow~
  \bigwedge_{m=1}^K \frac{M_{r,m}(\mathbf p_m)}{(  1 -  \tau K /m )_+  }\le 1.
\] 
Since $F^*_r$ is induced by $f_r$, we have, for $\epsilon \in (0,1)$,
$F^*_r (\mathbf p)\le \epsilon$ if and only if either
$\bigwedge_{m=1}^K \frac{M_{r,m}(\mathbf p_m)}{(1 - \tau K /m )_+}\le \epsilon$ or $p_{(1)}=0$.
Hence, \eqref{eq:arith}  holds.
\end{proof}

\begin{proof}[Proof of Proposition \ref{prop:grand-2}] \quad
\begin{enumerate}[(i)]
\item 
  To show the ``if'' statement, it suffices to note again that
  $
  M_{r,K} (\mathbf u)\le  M_{s,K}(\mathbf u)
  $ for all $\mathbf u\in (0,\infty)^K$
  and the above inequality is strict unless $\mathbf u$ has only one positive component
  \citep[Theorem 16]{Hardy/etal:1952}.
  Therefore, $a M_{r,K}$ (strictly) dominates $b M_{s,K}$.
  To show the ``only if'' statement,
  we note that $aM_{r,K}$ cannot dominate $bM_{s,K}$ if $a>b$ since $M_{r,K}$ and $M_{s,K}$ agree on vectors with equal components.
\item
  We first assume $0<r<s$.
  To show the ``if'' statement, it suffices to note again that
  $
  K^{1/r} M_{r,K} (\mathbf u)\ge K^{1/s} M_{s,K}(\mathbf u)$
  for all   $\mathbf u\in [0,\infty)^K$
  and the above inequality is strict if $\mathbf u$ does not have equal components \citep[Theorem 19]{Hardy/etal:1952}.
  Therefore, $b M_{s,K}$ (strictly) dominates $a M_{r,K}$.
  To show the ``only if'' statement, we note that, if $ a K^{-1/r} < bK^{-1/s}$,
  \[F_{r,K}(1,0,\dots,0) = a  K^{-1/r} <  b K^{-1/s} = F_{s,K}(1,0,\dots,0),\]
  and thus  $bM_{s,K}$ cannot dominate  $aM_{r,K}$ if $a K^{-1/r} < b K^{-1/s}$. 

We next assume $ r<s<0$.
To show the ``if'' statement, 
we first note that, using \citet[Theorem 19]{Hardy/etal:1952}, for all $\mathbf u \in (0,\infty]^K$, 
\[
  K^{1/r} M_{r,K}(1/\mathbf u) = \frac{1}{K^{-1/r} M_{-r,K}(\mathbf u)}
 \ge 
 \frac{1}{K^{-1/s} M_{-s,K}(\mathbf u)}
 =  K^{1/s} M_{s,K}(1/\mathbf u),
\]
and the above inequality is strict if at least one of the components of $\mathbf u$ is $0$.
Therefore, $bM_{s,K}$ strictly dominates $a M_{r,K}$ if
$a K^{-1/r} \le b K^{-1/s}$.
To show the ``only if'' statement, we note that, if $ a K^{-1/r} < bK^{-1/s}$, we have
\begin{align*}
  \lim_{\epsilon\downarrow 0} a M_{r,K}(1,1/\epsilon,\dots,1/\epsilon)
  &=
  a K^{-1/r} 
  < b K^{-1/s} 
  =
  \lim_{\epsilon\downarrow 0}  b M_{s,K}(1,1/\epsilon,\dots,1/\epsilon) ,
\end{align*}
and thus $b M_{s,K}$ cannot dominate $aM_{r,K}$ if $a K^{-1/r} < b K^{-1/s}$.

Finally, we consider the case $rs\le 0$. If $r\le 0<s$,
then using simple properties of the averages, we have 
\[ M_{r,K} (0,1,\dots,1) = 0 <  \left(\frac{K-1}{K}\right)^{1/s} =  M_{s,K} (0,1,\dots,1). \] 
If $r<s=0$, we have 
\begin{equation*}
  \lim_{\epsilon \downarrow 0} \frac1 \epsilon M_{r,K} (\epsilon^K,1,\dots,1)
  =
  \lim_{\epsilon \downarrow 0} \left( \frac{  M_{r,K} (\epsilon^K,1,\dots,1)   }\epsilon\right) 
  =
  \lim_{\epsilon \downarrow 0} \left(  \frac{\epsilon^{Kr} + K-1}{K\epsilon^r }\right)^{1/r}
  =
  0,
\end{equation*} 
whereas
\begin{equation*}
  \lim_{\epsilon \downarrow 0} \frac1 \epsilon M_{0,K} (\epsilon^K,1,\dots,1)
  =
  \lim_{\epsilon \downarrow 0} \frac1 \epsilon M_{0,K} (\epsilon^K,1,\dots,1)
  =
  1 > 0.
\end{equation*}  
In either case, $b M_{0,K}$ cannot dominate  $a M_{r,K}$.
 
Summarizing the above cases, $bM_{0,K}$ dominates $a M_{r,K}$ if and only if $a K^{-1/r} \ge b K^{-1/s}$ and $rs>0$.
\qedhere
\end{enumerate}
\end{proof}

\begin{proof}[Proof of Proposition \ref{prop:grand-3}] 
  In this proof, we do not truncate our merging functions at $1$.
  That is, we directly treat $F_{r,K}=b_{r,K}M_{r,K}$ without loss of generality,
  since the functions in the M-family are homogeneous.
  We say that two p-merging functions are {not comparable} if neither of them dominates the other one. 
 
  Using Table~1 of \citet{Vovk/Wang:2020Biometrika} (or Section~\ref{app:K2}), 
  the case $K=2$ follows directly from Proposition \ref{prop:grand-2} since $b_{r,2}=2^{1/r}$ for all $r\in [-\infty,1]$ and $b_{r,2}=2$ for $r<1$.
  We next study the case $K\ge 3$. 
  Using Table 1 of \citet{Vovk/Wang:2020Biometrika}, $b_{r,K}= K^{1/r}$ for $r\ge K-1$.
  Hence, by Proposition \ref{prop:grand-2},
  $F_{r,K}$ is dominated by $F_{s,K}$
  if $K-1\le r<s$.
  We next show that this is the only possible domination between $F_{r,K}$ and $F_{s,K}$.
  
  First,  for $r,s\in [(K-1)^{-1}, K-1]$, we have $b_{r,K}= (1+r)^{1/r}$.
  Clearly, $b_{r,K}$ is strictly decreasing in $r$, and hence 
  Proposition \ref{prop:grand-2} (i) implies that $F_{r,K}$ does not dominate $F_{s,K}$ for $r<s$.
  Moreover, we can calculate
  \[
    \frac{b_{r,K} K^{-1/r}}{b_{s,K} K^{-1/s}}
    =
    \frac{\left(\frac{1+r} K \right)^{1/r}}{\left(\frac{1+s} K \right)^{1/s}}
    =
    \left(\frac{1+r}{1+s} \right)^{1/s} \left(\frac{1+r} K \right)^{1/r-1/s} < 1.
  \]
  Therefore, $F_{s,K}$ does not dominate $F_{r,K}$ either.
  We thus know that $F_{s,K}$ and $F_{r,K}$ are not comparable in this case.  

Next, we consider $s<r\le(K-1)^{-1}$.
To show that $F_{s,K}$ and $F_{r,K}$ are not comparable,
by \eqref{eq:compare-p1} and Proposition \ref{prop:grand-2},
it suffices to show $b_{r,K}\ne  b_{s,K}$ and $b_{r,K} K^{-1/r} \ne  b_{s,K}K^{-1/s}$.
These can be shown by straightforward (although cumbersome) calculation from the explicit formulas in Proposition \ref{prop:grand}.
An intuitive explanation is that the dependence structure of the vector $\mathbf P_r\in \mathcal P_Q^K$
which gives the precise probability $\p(F_{r,K}(\mathbf P_r)\le \epsilon)=\epsilon$ is different across $r\in (-\infty, K-1]$
(see, e.g., \citet{Wang/etal:2013}).
This leads to $\p(F_{s,K}(\mathbf P_r)\le \epsilon)<\epsilon$ and $\p(F_{r,K}(\mathbf P_s)\le \epsilon)<\epsilon$
for $s\ne r$, and hence the two p-merging functions cannot be compared. 

The above arguments show that each $F_{r,K}$, $r<K-1$, is not comparable with $F_{s,K}$ for $s$ in a neighbourhood of $r$.
Finally, using Lemma \ref{lem:compare} below, we obtain that  $F_{r,K}$ for $r \le K-1$ is admissible within the M-family  
\end{proof}
 
\begin{lemma}\label{lem:compare}
  If $F_{r,K}$ is not dominated by $F_{s,K}$ for any $s$ in a neighbourhood of $r$,
  then $F_{r,K}$ is admissible within the M-family.
\end{lemma}

\begin{proof}[Proof of Lemma \ref{lem:compare}] 
Since $F_{r,K}$ is not dominated by any $F_{s,K}$ for $s$ in a neighbourhood of $r$,
we obtain from Proposition \ref{prop:grand-2} (i) that $b_{r,K} > b_{s,K}$ for all $s>r$
using monotonicity of $b_{r,K}$ in \eqref{eq:compare-p1}.
Similarly, $b_{r,K} K^{-1/r} < b_{s,K} K^{-1/s}$ for all $s<r$ with $rs>0$.
Using Proposition \ref{prop:grand-2} (i) and (ii),
we know that $F_{r,K}$ is not dominated by $F_{s,K}$ if $rs>0$.
Also, by Proposition \ref{prop:grand-2} (ii), $F_{r,K}$ is not dominated by $F_{s,K}$ if $s<r$ and $rs\le 0$. 
Therefore, $F_{r,K}$ is admissible within the M-family.   
\end{proof}

\subsection{Proof of Proposition \ref{pr:ratio}}
\label{app:ratio}

\begin{proof}[Proof of Proposition \ref{pr:ratio}]
  Let $\boldsymbol\epsilon=(\epsilon,\dots,\epsilon,1)\in \R^K$ and $\boldsymbol\epsilon'=(\epsilon,1,\dots,1)\in \R^K$
  for some $\epsilon>0$.
  \begin{enumerate}[(i)]
  \item
    By definition,
    $F_{1,K}(\boldsymbol \epsilon) \ge 2 /K$ and $F^*_{1,K}(\boldsymbol\epsilon) \le \frac{2K}{K-2}\epsilon \le 6\epsilon$.
    Hence, ${F^*_{1,K}(\boldsymbol\epsilon)}/{F_{1,K}(\boldsymbol\epsilon)} \to 0$ as $\epsilon\downarrow 0$.      
  \item
    By definition, 
    $F_{0,K}(\boldsymbol \epsilon')  = \epsilon^{1/K} c$ for some constant $c>0$ and
    $F^*_{0,K}( \boldsymbol \epsilon' ) \le \epsilon c' $ for some constant $c'>0$.
    Hence,  
    ${F^*_{0,K}(\boldsymbol \epsilon' )}/{ F_{0,K}(\boldsymbol \epsilon')} \to 0$ as $\epsilon \downarrow 0$.      
  \item
    Write $c:=c_{-1}$. For any $\mathbf p\in (0,\infty)^K$, we have
    \begin{align*}
      \frac{F^*_{-1,K}(\mathbf{p})}{F_{-1,K}(\mathbf{p})}
      & =
      \bigwedge_{m=1}^K \frac{M_{-1,m} (\mathbf{p}_m) / M_{-1,m} (\mathbf{v}_m(c))}{M_{-1,K} (\mathbf{p}) / M_{-1,K} (\mathbf{v}_K(c))} \\
      & =
      \bigwedge_{m=1}^K
      \left(
        \frac{c^{-1} + (m-1)(1-(K-1)c)^{-1}}{c^{-1} + (K-1)(1-(K-1)c)^{-1}} \times \frac{\sum_{k=1}^K p_{(k)}^{-1}}{\sum_{k=1}^mp_{(k)}^{-1}}
      \right)\\
      &\ge
      \bigwedge_{m=1}^K
      \frac{1-(K-1)c+(m-1)c}{1-(K-1)c+(K-1)c}
      =
      1-(K-1)c,
    \end{align*}
    where $\mathbf{v}_m(c):=(c,d,\dots,d)$ with $m-1$ entries of $d:=1-(K-1)c$.
    Taking $\mathbf p=\boldsymbol \epsilon'$ and letting $\epsilon\downarrow 0$ justifies the infimum value.
  \item
    Take any $\mathbf{p} $ and let $\alpha = \bigwedge_{k=1}^K p_{(k)}/k$.
    Without loss of generality, we assume $\alpha K \ell_K \le 1$ and hence $H^*_K(\mathbf p)\le H_K(\mathbf p)\le 1$.
    Since $H^*_K$ is homogeneous, symmetric and increasing, we have
    \begin{equation}\label{eq:Hk-improve}
      H^*_K(\mathbf{p})
      \ge
      H^*_K (\alpha,2\alpha,\dots,K\alpha)
      =
      \alpha K\ell_K \gamma_K
      =
      \gamma_K  H_K(\mathbf p).
    \end{equation}
    The minimum ratio $H^*_K(\mathbf{p}) /H_K(\mathbf p) = \gamma_K$
    is attained by $\mathbf{p} = (\alpha,2\alpha,\dots,K\alpha)$ for $\alpha\in (0,1/K\ell_K]$. 
  \item
    We continue to write $c=c_{-1}$.
    Proposition 6 of \citet{Vovk/Wang:2020Biometrika} gives that $b_{-1,K} \sim \log K$,
    and with Proposition \ref{prop:grand} we get $c(1-(K-1)c) \sim 1/ (K \log K)$.
    Since $c\in(0,1/K)$, the above implies $K c\to 0$ as $K\to\infty$,
    and this further implies $c\sim 1/(K\log K)$.
    Next, we look at the quantity
    \[
      y_K
      :=
      \frac 1 {\gamma_K }
      =
      \max \left\{ y\ge 1 : \sum_{k=1}^{ K}\frac{\id_{\{y\le K/k\}}}{ \lceil k  y \rceil  } \ge 1 \right\}.
    \]
    Note that $ y':= \lfloor y_K \rfloor + 1$ satisfies
    $\sum_{k=1}^{ K}\frac{\id_{\{y'\le K/k\}} }{ \lceil k y' \rceil  } <1$,
    and we get
    \[
      1 > \sum_{i=1}^{  K }\frac{\id_{\{y'\le K/k\}} }{ ky'}  = \frac{1}{y'} \ell_{\lfloor K/y'\rfloor} \ge \frac{\log K - \log y'}{y'},
    \]
    where the last inequality is due to $\ell_k \ge \log(k+1)$ for all $k\in\N$.
    Hence, $y' +  \log y'  > \log K$,
    which implies $y' > \log K - \log \log K$ and thus $y_K \ge \lfloor \log K - \log \log K \rfloor$.
    On the other hand,
    Theorem \ref{thm:Simes} implies that $y_K\le  H_K/S_K =  \ell_K \le \log K+1$.
    Therefore, $y_K\sim \log K$ as $K\to\infty$.
    \qedhere
  \end{enumerate}
\end{proof}

\subsection{Naive procedure for merging p-values}\label{app:naive}

As we saw in Section~\ref{sec:e},
p-to-e merging is easy.
We can restate it formally as follows.

\begin{corollary}
  The class of admissible p-to-e merging functions
  coincides with the class of functions \eqref{eq:general},
  $f_1,\dots,f_K$ ranging over the admissible calibrators and $(\lambda_1,\dots,\lambda_K)$ over $\Delta_K$.
\end{corollary}

\begin{proof}
  Combine Proposition \ref{pr:dual} with a slightly generalized version (with the same proof)
  of \citet[Proposition G.2]{Vovk/Wang:E}.
\end{proof}

A dual notion to p-to-e calibrators is that of e-to-p calibrators;
the latter are functions that transform e-variables into p-variables.
It turns out that the only admissible e-to-p calibrator
is the reciprocal function $p\mapsto1/p$ \citep[Proposition 2.2]{Vovk/Wang:E}.
The ease of merging e-values suggests merging p-values using a detour via e-values:
(i) calibrate p-values $p_1,\dots,p_K$ via calibrators $f_1,\dots,f_K$ getting e-values $f_k(p_k)$;
(ii) merge the e-values via weighted arithmetic average, getting $\sum_k\lambda_k f_k(p_k)$;
(iii) calibrate the resulting e-value back to the p-value
\begin{equation}\label{eq:detour}
  F(p_1,\dots,p_K)
  :=
  \frac{1}{\sum_k\lambda_k f_k(p_k)}.
\end{equation}
This detour via e-values is in fact a poor procedure;
e.g., the p-merging function \eqref{eq:detour} is not admissible.
Let us check this.

To check that \eqref{eq:detour} is not admissible,
suppose (temporarily allowing $K=1$), without loss of generality, that all $\lambda_k$ are positive
and that all $f_k$ are admissible and so upper semicontinuous.
Arguing indirectly, suppose \eqref{eq:detour} is admissible and $c>1$.
We then have
\begin{equation}\label{eq:sup}
  \sup_P
  P
  \left(
    \left\{
      (p_1,\dots,p_K)\in[0,1]^K:
      \sum_k \lambda_k f_k(p_k)
      \ge
      c
    \right\}
  \right)
  =
  \frac1c,
\end{equation}
$P$ ranging over the probability measures on $[0,1]^{\infty}$ with the uniform marginals.
Since this is true for any $c$, at least one of the $f_k$ is unbounded on $(0,1]$.
Now let us fix a $c>1$.
Since the set of probability measures $P$ is compact in the topology of weak convergence
and the set in \eqref{eq:sup} is closed, the supremum in \eqref{eq:sup} is attained,
and so $\sum_k \lambda_k f_k(p_k)=c$ $P$-a.s.;
this contradicts one of the $f_k$ being unbounded on $(0,1]$.

As we can see, the naive procedure does not produce useful p-merging functions,
but it turns out that it can be repaired.
In the following somewhat informal argument we will ignore issues of measurability.
To recover any p-merging function,
it suffices to perform the detour via e-values for each rejection region \eqref{eq:region} separately.
Namely, for any $\epsilon\in(0,1)$:
(i) calibrate p-values $p_1,\dots,p_K$ via calibrators $f_{1,\epsilon},\dots,f_{K,\epsilon}$
  getting e-values $e_{k,\epsilon}=f_{k,\epsilon}(p_k)$.
(ii) Merge the e-values via weighted arithmetic average, getting $e_{\epsilon}=\sum_k\lambda_{k,\epsilon} e_{k,\epsilon}$.
(iii) Include $(p_1,\dots,p_K)$ in $R_\epsilon$ if $1/e_{\epsilon}\le\epsilon$.
If $f_{k,\epsilon}$ are chosen in such a way that $R_\epsilon$ is increasing in $\epsilon$,
this will be a p-merging family
(in the sense of satisfying
$Q(\mathbf P \in R_\epsilon) \le \epsilon$ for all $\epsilon\in (0,1)$ and $\mathbf P\in \mathcal P_Q^K$).
And vice versa, by the duality theorem in the form of Proposition~\ref{pr:dual},
for any p-merging function $F$ and any $\epsilon\in(0,1)$,
the rejection region $R_{\epsilon}(F)$ will be rejected in the sense
\[
  (p_1,\dots,p_K)\in R_{\epsilon}(F)
  \Longrightarrow
  \frac{1}{\sum_k\lambda_{k,\epsilon} f_{k,\epsilon}(p_k)} \le \epsilon
\]
for suitably chosen $f_{k,\epsilon}$ and $\lambda_{k,\epsilon}$.

The conclusion of Proposition~\ref{pr:dual},
as applied to $F$ that is constant in a region $R$ and zero outside $R$,
can be strengthened if we assume that $R$ is a rejection region of an admissible p-merging function.
The proof of Theorem~\ref{th:e} also proves the following proposition.

\begin{proposition} 
  For any admissible p-merging function $F$ and $\epsilon\in(0,1)$,
  there exist $(\lambda_1,\dots,\lambda_K)\in\Delta_K$ and admissible calibrators $f_1,\dots,f_K$
  such that
  \[
    F(\mathbf p) \le \epsilon ~~\Longleftrightarrow~~ \sum_{k=1}^K \lambda_k f_k(p_k) \ge \frac{1}{\epsilon}.
  \]
  If $F$ is symmetric, then there exists an admissible calibrator $f$ such that
  \[
    F(\mathbf p) \le \epsilon ~~\Longleftrightarrow~~ \frac1K\sum_{k=1}^K f(p_k) \ge \frac{1}{\epsilon}.
  \]
\end{proposition}

Let us specialize the modified naive procedure to homogeneous p-merging functions.
According to Theorem~\ref{th:e}, in the homogeneous case
we can use calibrators $f_{k,\epsilon}(x) := f_k(x/\epsilon)/\epsilon$.
The procedure becomes almost as simple as the naive procedure;
both depend on a sequence $f_1,\dots,f_K$ of calibrators as parameter.
If we are interested in homogeneous and symmetric p-merging functions,
the detour via e-values can use calibrators $f_1=\dots=f_K$ and the arithmetic mean as e-merging function
(Theorem~\ref{pr:e}).

\subsection{An additional technical remark on Theorems \ref{th:admissible} and \ref{th:m1}}
\label{app:a8}

\begin{remark}\label{rem:strict-con}
We discuss technical challenges arising in trying to relax the strict convexity (or strict concavity) imposed in Theorem \ref{th:admissible}
and to prove the admissibility of $F_{1,K}^*$ in Theorem \ref{th:m1} for $K\ge 3$. 
Recall in  the proof of Theorem \ref{th:admissible} that the density $h$ is obtained from a distribution with quantile function $f$,
and $h$ is decreasing if $f$ is convex. 
A crucial step in this proof is to verify that the distributions with densities $h_1,\dots,h_K$ are jointly mixable,
which ensures that in \eqref{eq:convex-constr}, if $A$ happens, the vector $(P_1,\dots,P_K)/\alpha=( f^{-1} (X_1),\dots,f^{-1}(X_K))$
satisfies $\sum_{k=1}^K  f(P_k)\ge K$,
so that $(P_1,\dots,P_K)\in R_\alpha(F)$.
The densities $h_1,\dots,h_K$ are obtained from the density $h$ by removing a tiny piece 
$m^* v_k / m_k$ for each $k$; see \eqref{eq:h-cond}.
Since $ m^*   v_k   /m_k$ is tiny, the resulting density is still decreasing (or increasing) if $h$ is strictly decreasing (or strictly increasing),
and hence joint mixability can be obtained from Theorem 3.2 of \citet{WW16}. 
In case the convex function $f$ is linear on some interval (which is the case for $F_{1,K}^*$),
$h$ is constant on this interval.
After removing a tiny piece on this interval from $h$,
the resulting density is no longer monotone, and no result  for joint mixability is available in this case.
Proving joint mixability is known to be a very difficult task, although we suspect that it holds true for the above special case
(if a proof is available, it likely will require a new paper).
Unfortunately, it seems to us that  one could not avoid this task for a generalization of Theorem \ref{th:admissible},
since showing $\sum_{k=1}^K f(P_k)\ge K$ for $h$ with some pieces removed is essential for constructing any counter-example,
at least to the best of our imagination.
\end{remark}

\section{The case $K=2$}
\label{app:K2}

In the simple case $K=2$, where the task is to merge two p-values,
the class of admissible p-merging functions admits an explicit description.

For $E\subseteq[0,1]^K$,
let us set
\[
  \ucp(E)
  :=
  \sup_{\textbf{P}\in\PPP^K_Q}
  Q(\textbf{P}\in E)
\]
and call $\ucp(E)$ the \emph{upper p-probability} of $E$.
In the case $K=2$ upper p-probability admits a simple characterization.

\begin{lemma}
  The upper p-probability of any nonempty Borel lower set $E\subseteq[0,1]^2$ is
  \begin{equation}\label{eq:ucp}
    \ucp(E)
    =
    1
    \wedge
    \inf
    \left\{
      u_1+u_2: (u_1,u_2)\in[0,1]^2\setminus E
    \right\}.
  \end{equation}
\end{lemma}

\begin{proof}
  Let $E$ be a nonempty lower Borel set in $[0,1]^2$;
  suppose $\ucp(E)$ is strictly less than the right-hand side of \eqref{eq:ucp}.
  Let $t$ be any number strictly between $\ucp(E)$ and the right-hand side of \eqref{eq:ucp}.
  If $\textbf{P}$ is concentrated on
  \begin{equation}\label{eq:antitonic}
    [(t,0),(0,t)]\cup[(t,t),(1,1)],
  \end{equation}
  and each of its components is uniformly distributed on $[0,1]$,
  $\textbf{P}\in E$ with probability at least $t$ since $E$ contains $[(t,0),(0,t)]$.
  Therefore, $\ucp(E)\ge t$.
  This contradiction proves the inequality $\ge$ in \eqref{eq:ucp}.

  As for the opposite inequality, we will check
  \begin{equation*}
    \ucp(E)
    \le
    \inf
    \left\{
      u_1+\dots+u_K: (u_1,\dots,u_K)\in[0,1]^K\setminus E
    \right\}
  \end{equation*}
  for an arbitrary $K\ge2$.
  Let us assume that $E$ does not contain the set of all $(u_1,\dots,u_K)$ with $u_1+\dots+u_K=1$
  (the case when it does is trivial).
  Choose $\epsilon>0$ and $(p_1,\dots,p_K)\in[0,1]^K\setminus E$
  such that $t:=p_1+\dots+p_K\in[\epsilon,1]$ and $E$ contains all $(u_1,\dots,u_K)\in[0,1]^K$
  satisfying $u_1+\dots+u_K=t-\epsilon$.
  Since $E$ is a lower set, we have
  \[
    E
    \subseteq
    \bigcup_{k=1}^K
    \left\{
      (u_1,\dots,u_K)\in[0,1]^K : u_k\le p_k
    \right\},
  \]
  and the subadditivity of $\ucp$ further implies
  \begin{align*}
    \ucp(E)
    &\le
    \sum_{k=1}^K
    \ucp
    \left(\left\{
      (u_1,\dots,u_K)\in[0,1]^K: u_k\le p_k
    \right\}\right)\\
    &=
    \sum_{k=1}^K
    p_k
    =
    t
    \le
    \inf
    \left\{
      u_1+\dots+u_K: (u_1,\dots,u_K)\in[0,1]^K\setminus E
    \right\}
    +
    \epsilon.
  \end{align*}
  It remains to notice that $\epsilon$ can be chosen arbitrarily small.
\end{proof}

There is a natural bijection between the admissible p-merging functions for $K=2$
and increasing right-continuous functions $f:[0,1)\to[0,1]$.
The \emph{epigraph boundary} of such $f$ is the set of points $(u_1,u_2)\in[0,1]^2$
such that $f(u_1-)\le u_2\le f(u_1)$,
where $f(0-)$ is understood to be $0$ and $f(1)$ is understood to be $1$.
A \emph{diagonal curve} is the epigraph boundary of some increasing function.
The admissible p-merging function corresponding to a diagonal curve $A\subseteq[0,1]^2$
is defined by $F(p_1,p_2):=u_1+u_2$, where $(u_1,u_2)\in A$ is the largest point in $A$
that is less than or equal to $(p_1,p_2)$ in the component-wise order
($A$ is linearly ordered by this partial order).

In particular, the only symmetric admissible p-merging function for $K=2$ is Bonferroni.
It corresponds to the identity function $f:u\mapsto u$.

\section{Additional simulation results}
\label{app:simulation}

\begin{figure}
  \begin{center}
    \includegraphics[width=0.48\textwidth]{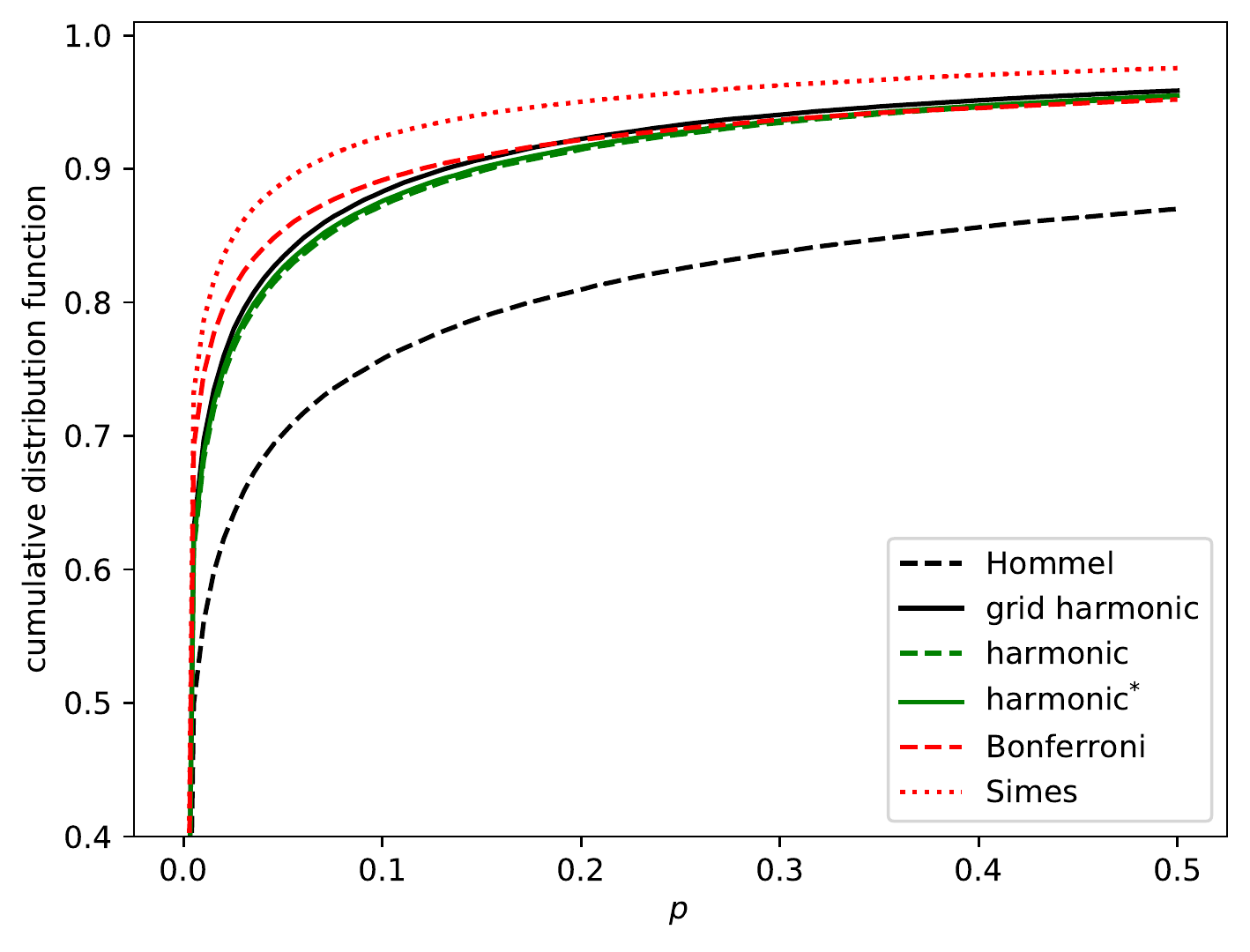}
    \hfill
    \includegraphics[width=0.48\textwidth]{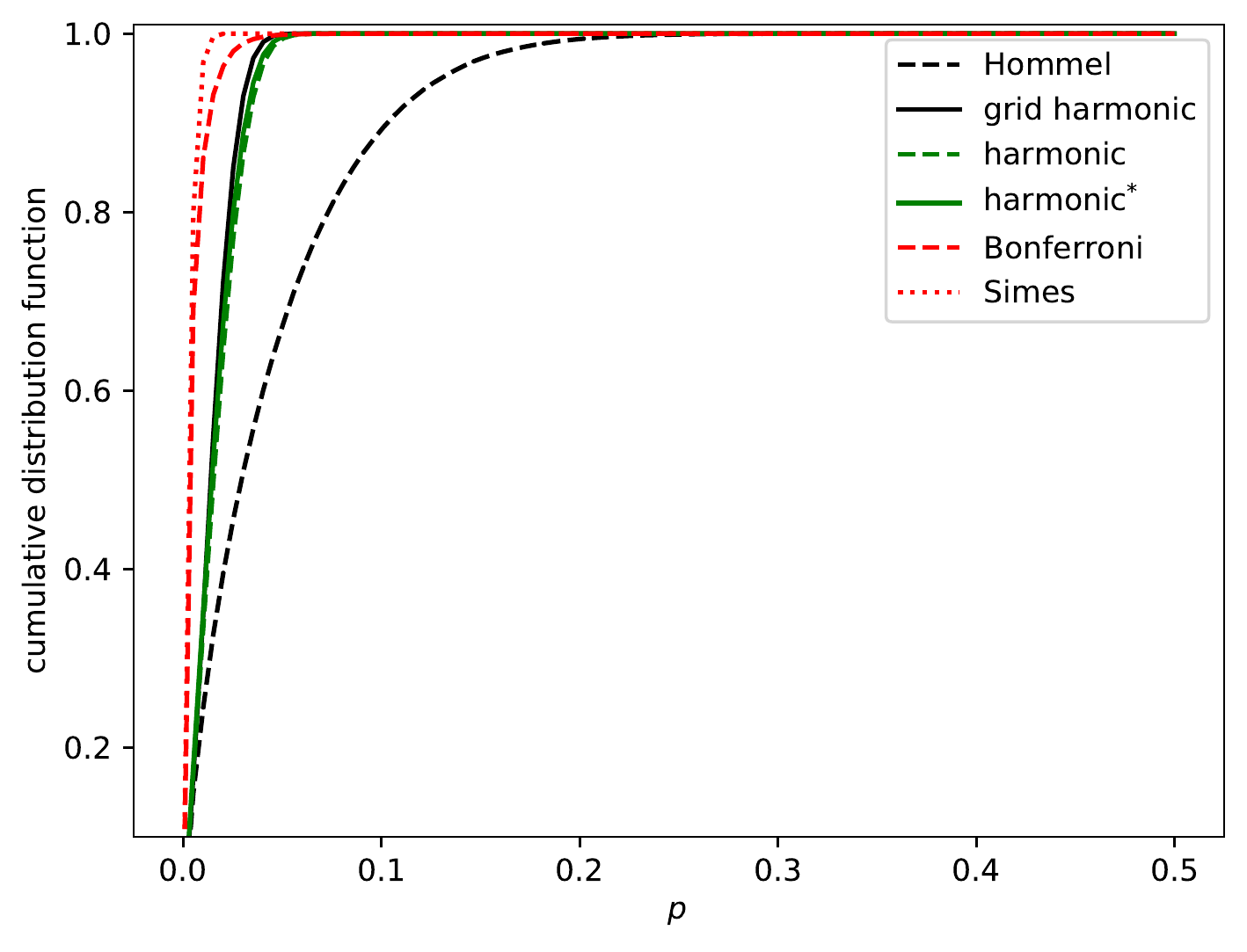}
  \end{center}
  \caption{An analogue of Figure \ref{fig:cdf}
    for $10\%$ of observations from the alternative distribution
    with correlation $0.5$ (left panel) and $0$ (right panel) in place of $0.9$.}
  \label{fig:cdf2}
\end{figure}

\begin{figure}
  \begin{center}
    \includegraphics[width=0.48\textwidth]{cdf_5_0_9_K_10_6_K1_10_3_N_10_5_inv.pdf}
    \hfill
    \includegraphics[width=0.48\textwidth]{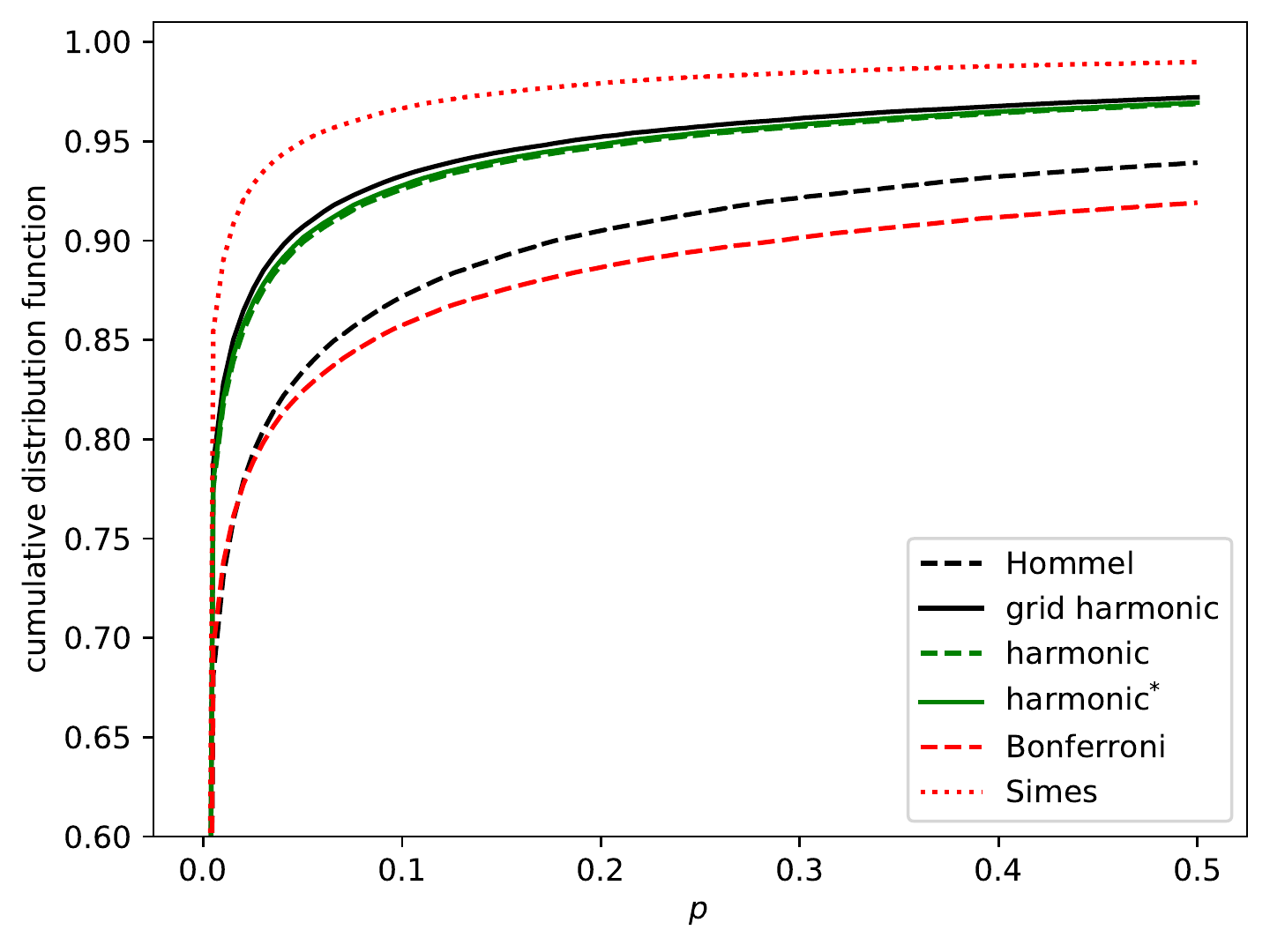}
  \end{center}
  \caption{Figure \ref{fig:cdf} (left panel),
    where $K_1=10^3$ and the correlation is $0.9$ for the bulk of the observations,
    and its counterpart with $K_1:=10^4$ (right panel).}
  \label{fig:cdf3}
\end{figure}

\begin{figure}
  \begin{center}
    \includegraphics[width=0.32\textwidth]{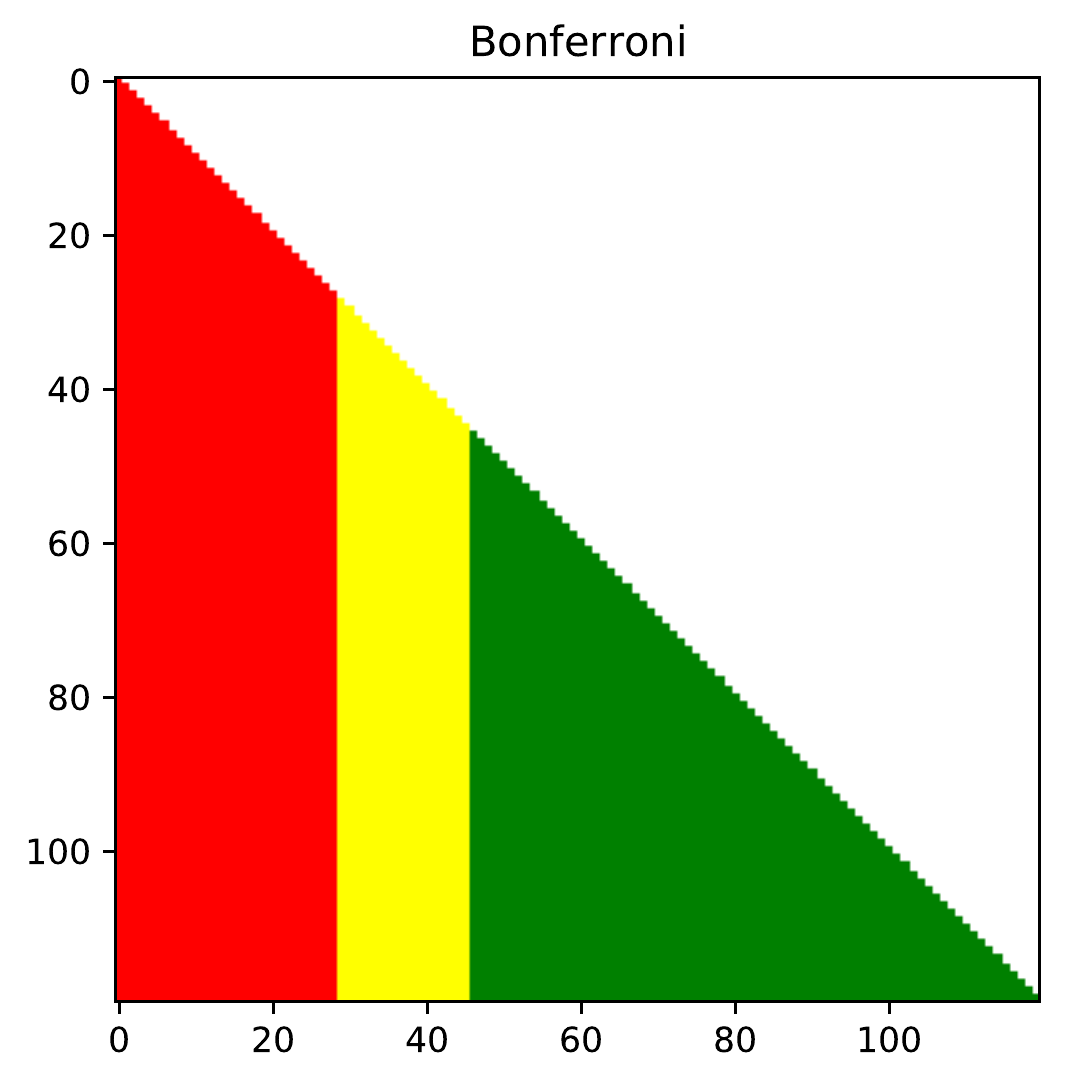}
    \includegraphics[width=0.32\textwidth]{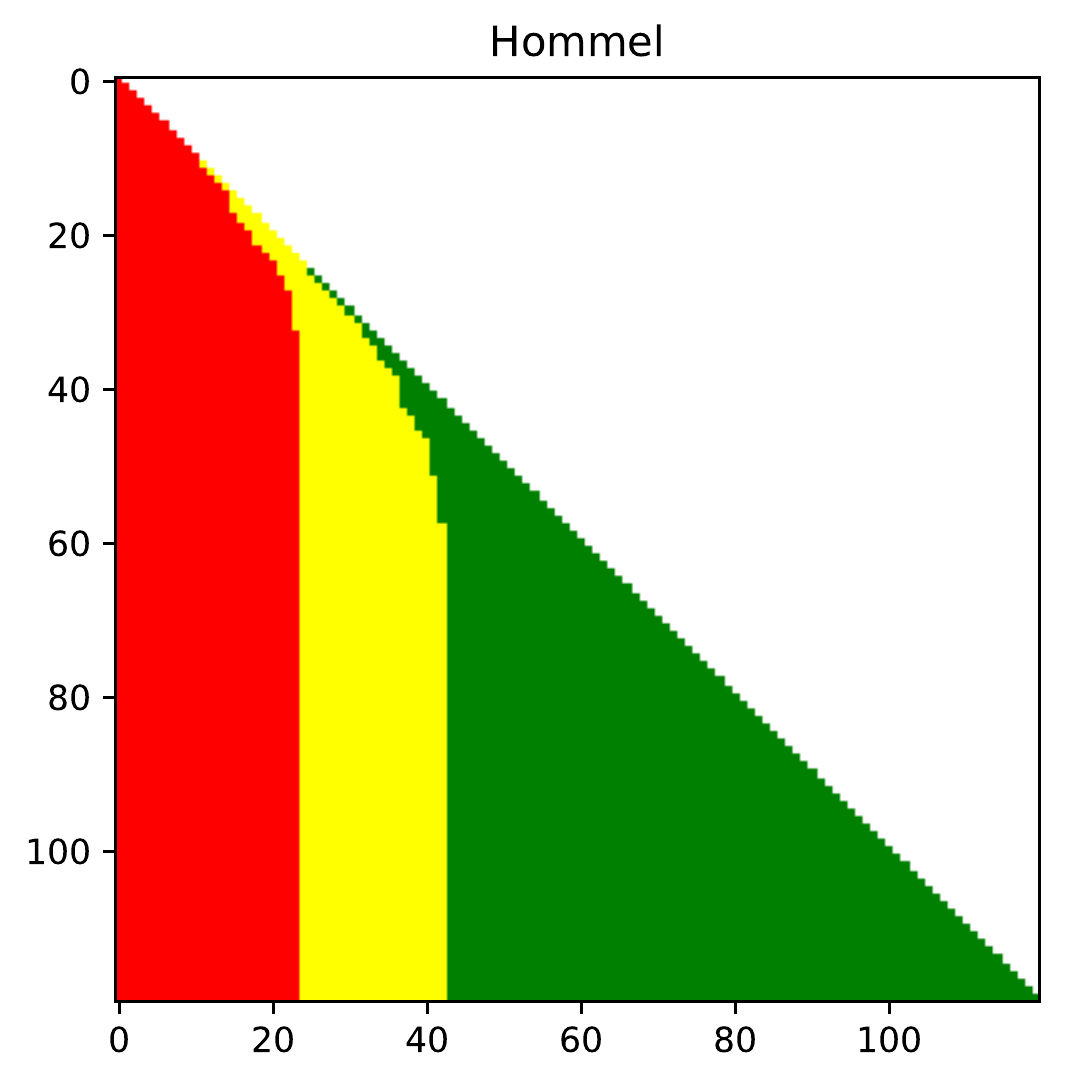}
    \includegraphics[width=0.32\textwidth]{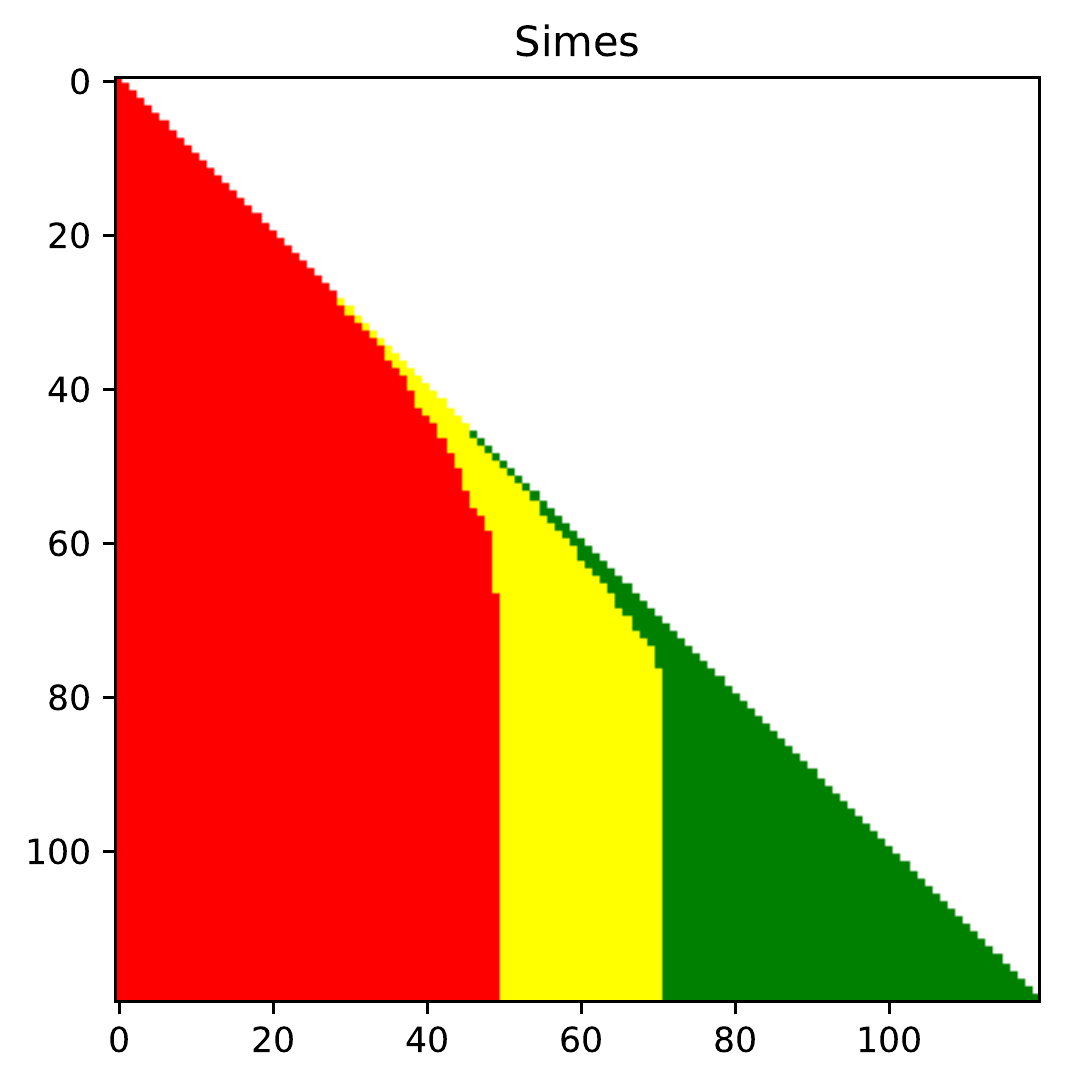} \\
    \includegraphics[width=0.32\textwidth]{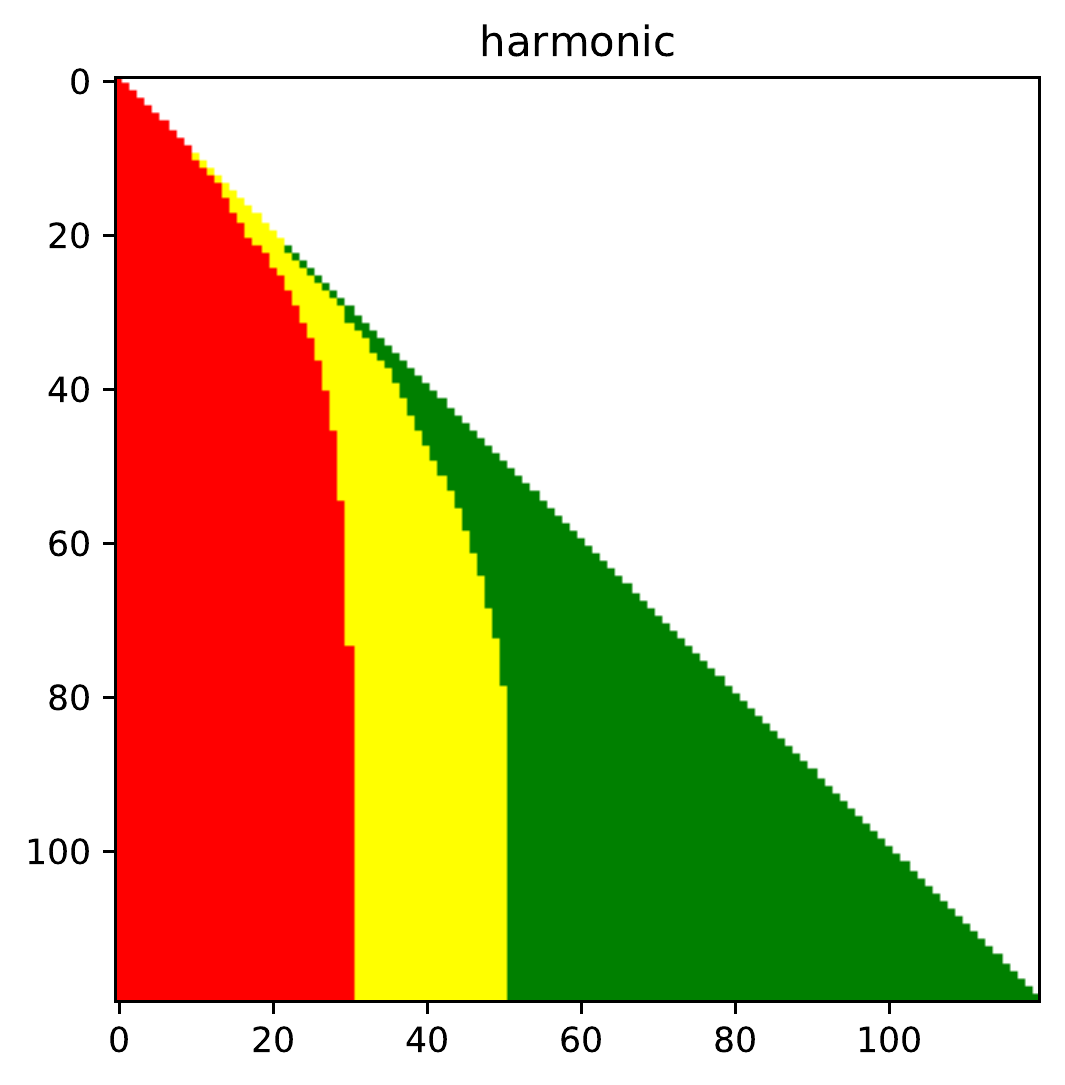}
    \includegraphics[width=0.32\textwidth]{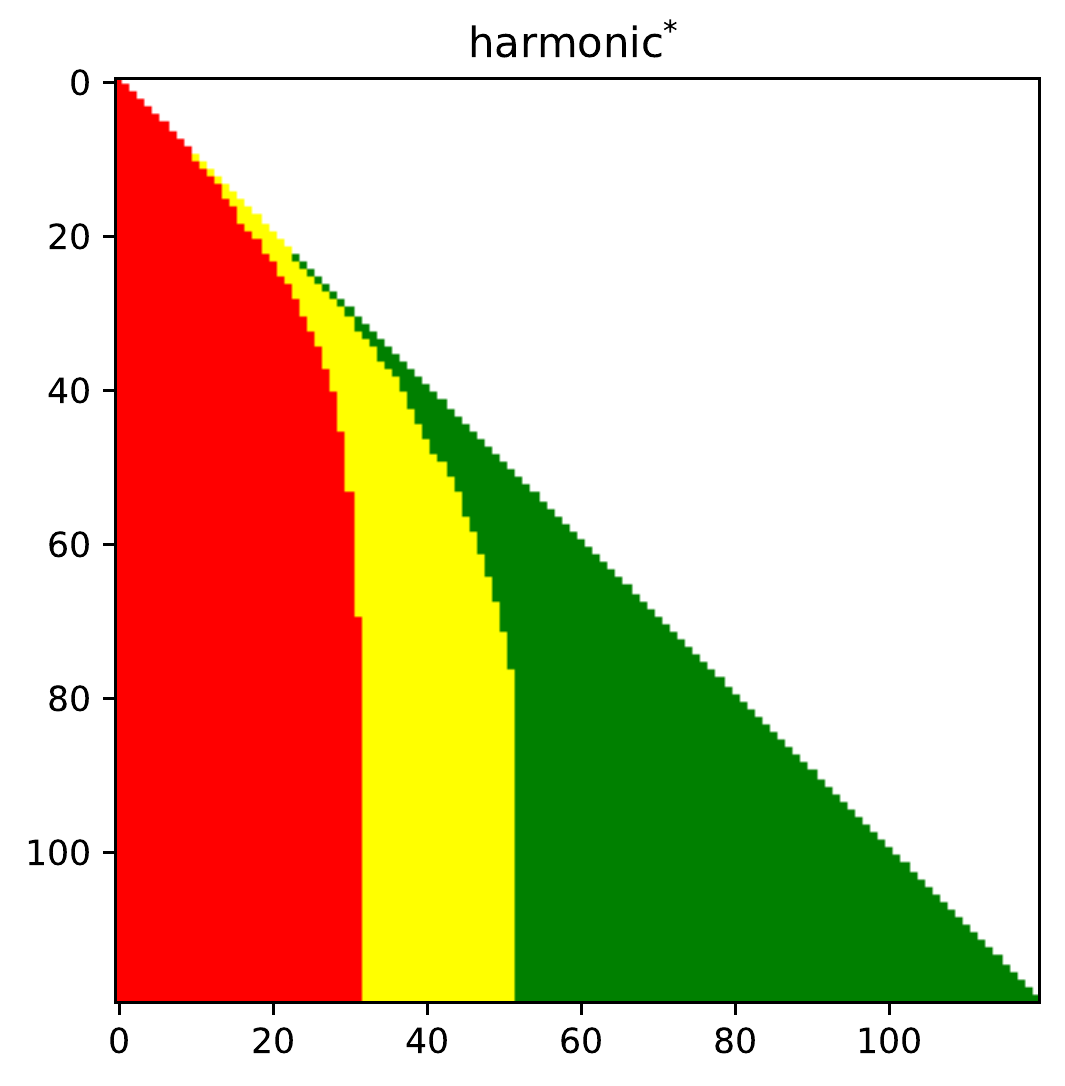}
    \includegraphics[width=0.32\textwidth]{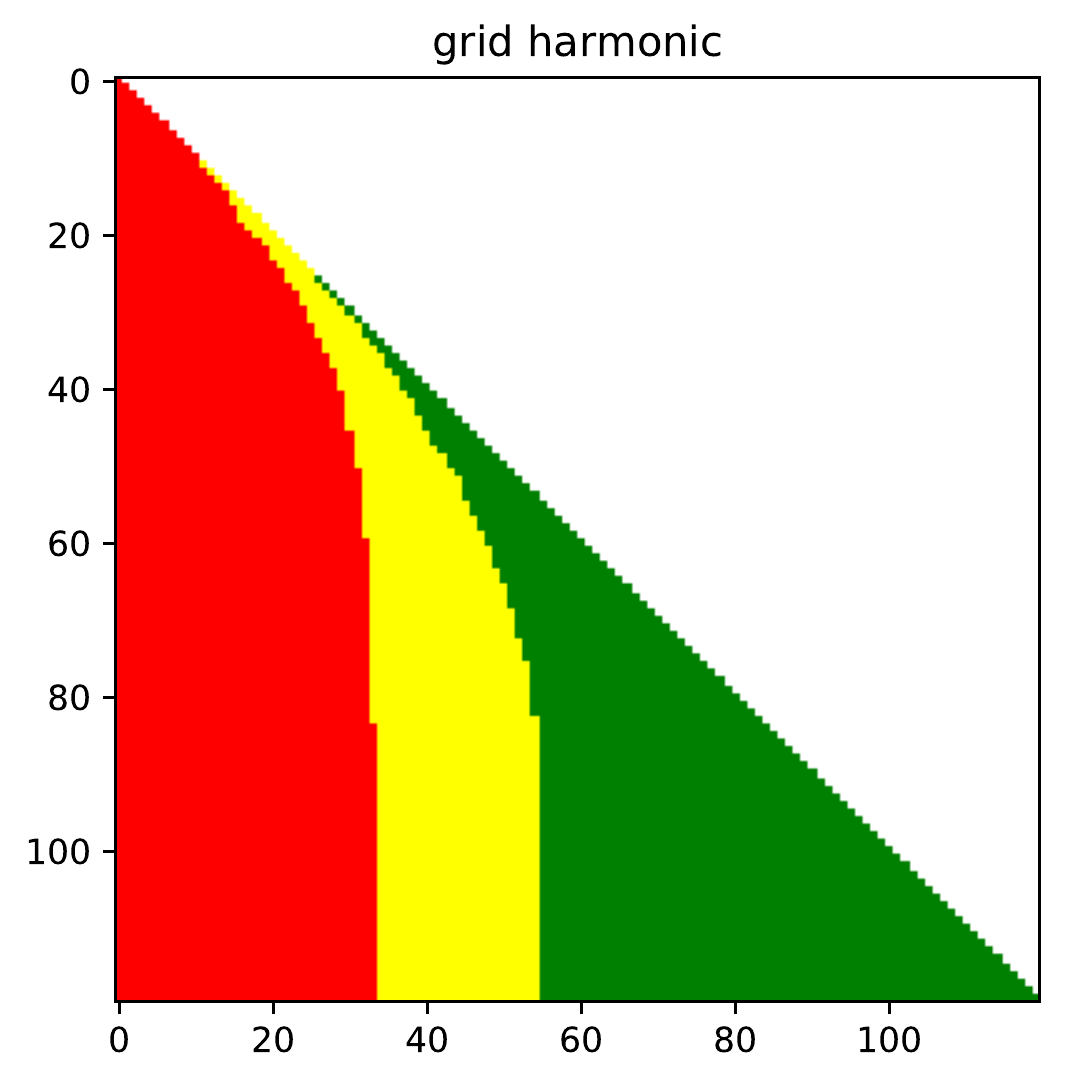}
  \end{center}
  \caption{An analogue of Figure \ref{fig:DM} with correlation $0.5$: GWGS discovery matrices for the simulation data.}
  \label{fig:DM2}
\end{figure}

\begin{figure}
  \begin{center}
    \includegraphics[width=0.32\textwidth]{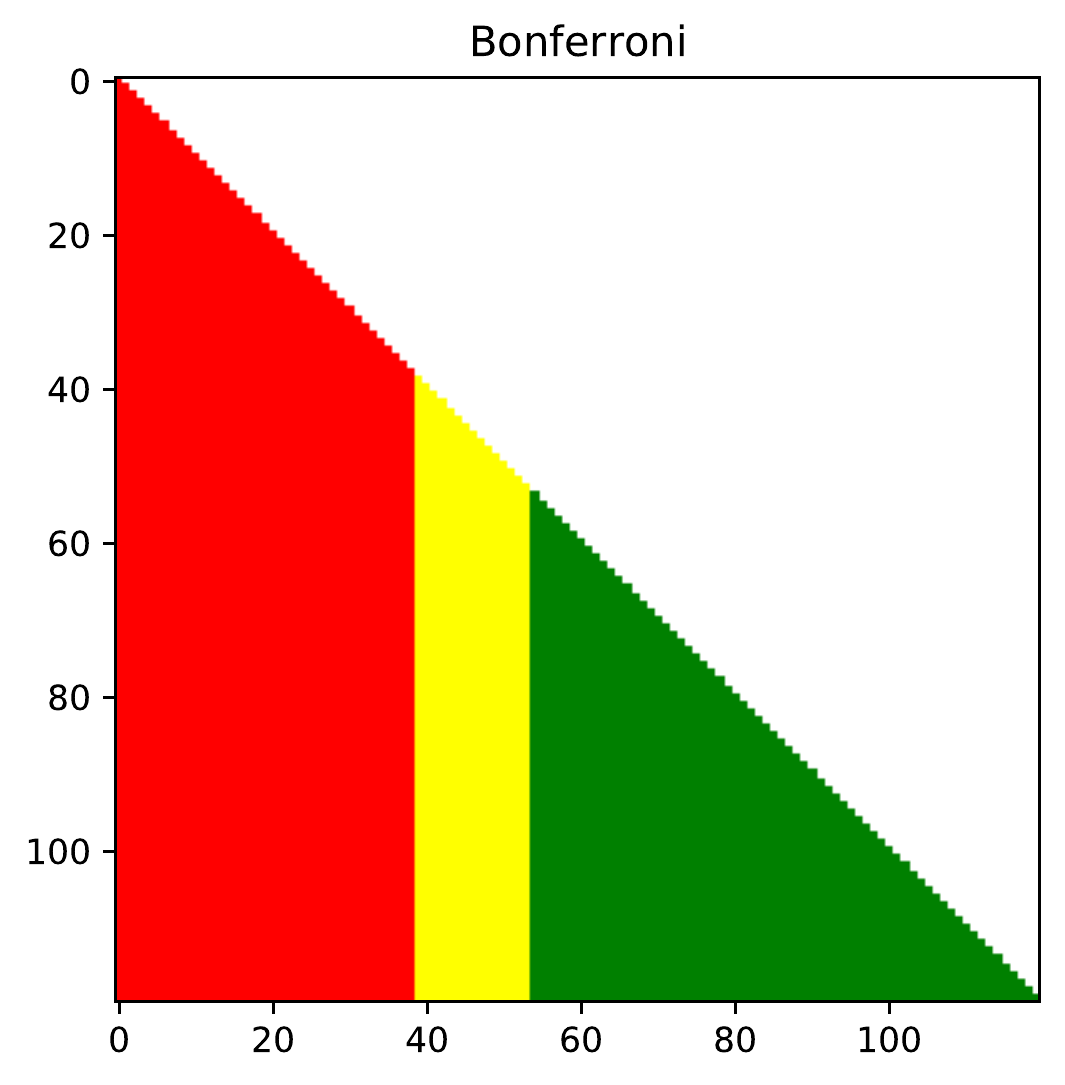}
    \includegraphics[width=0.32\textwidth]{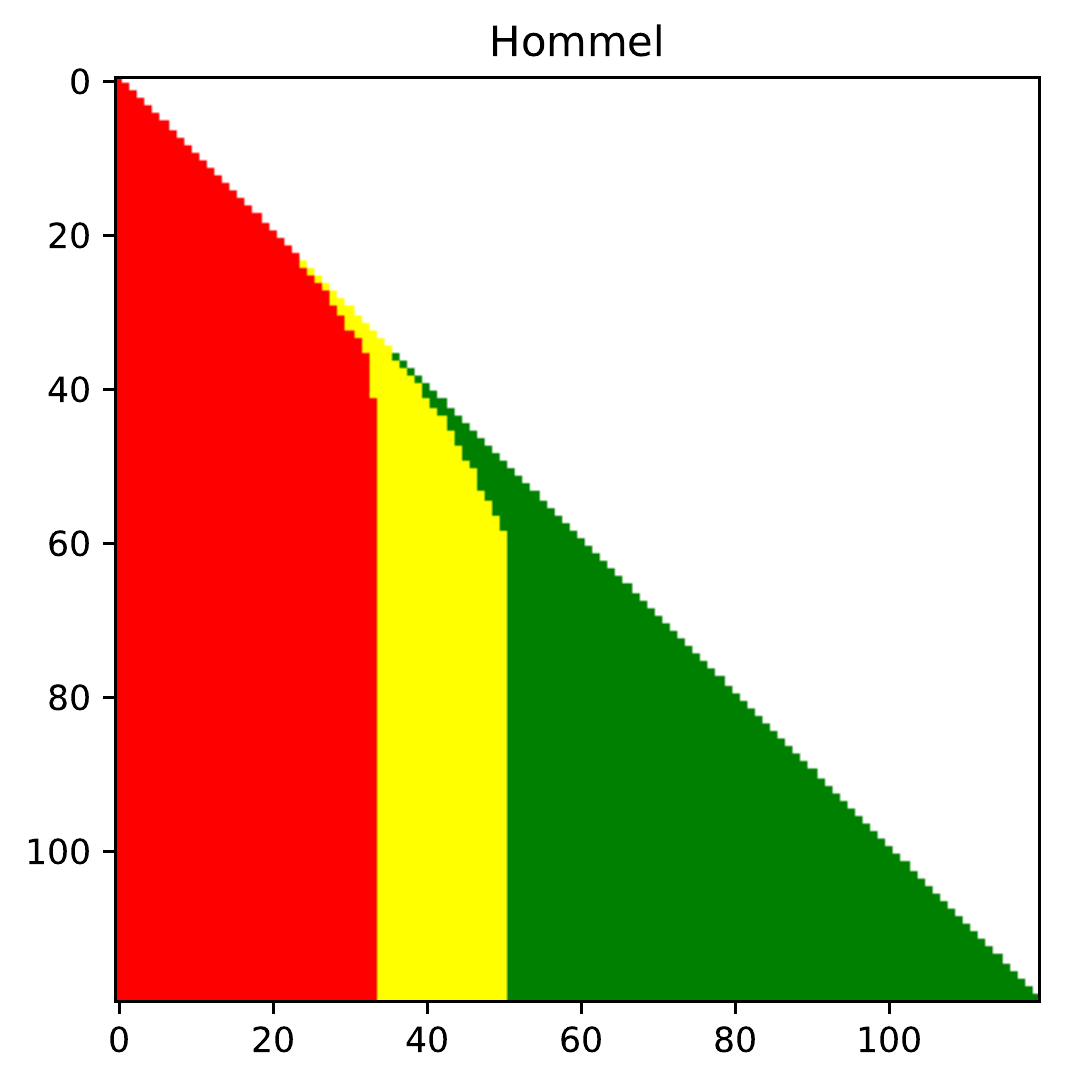}
    \includegraphics[width=0.32\textwidth]{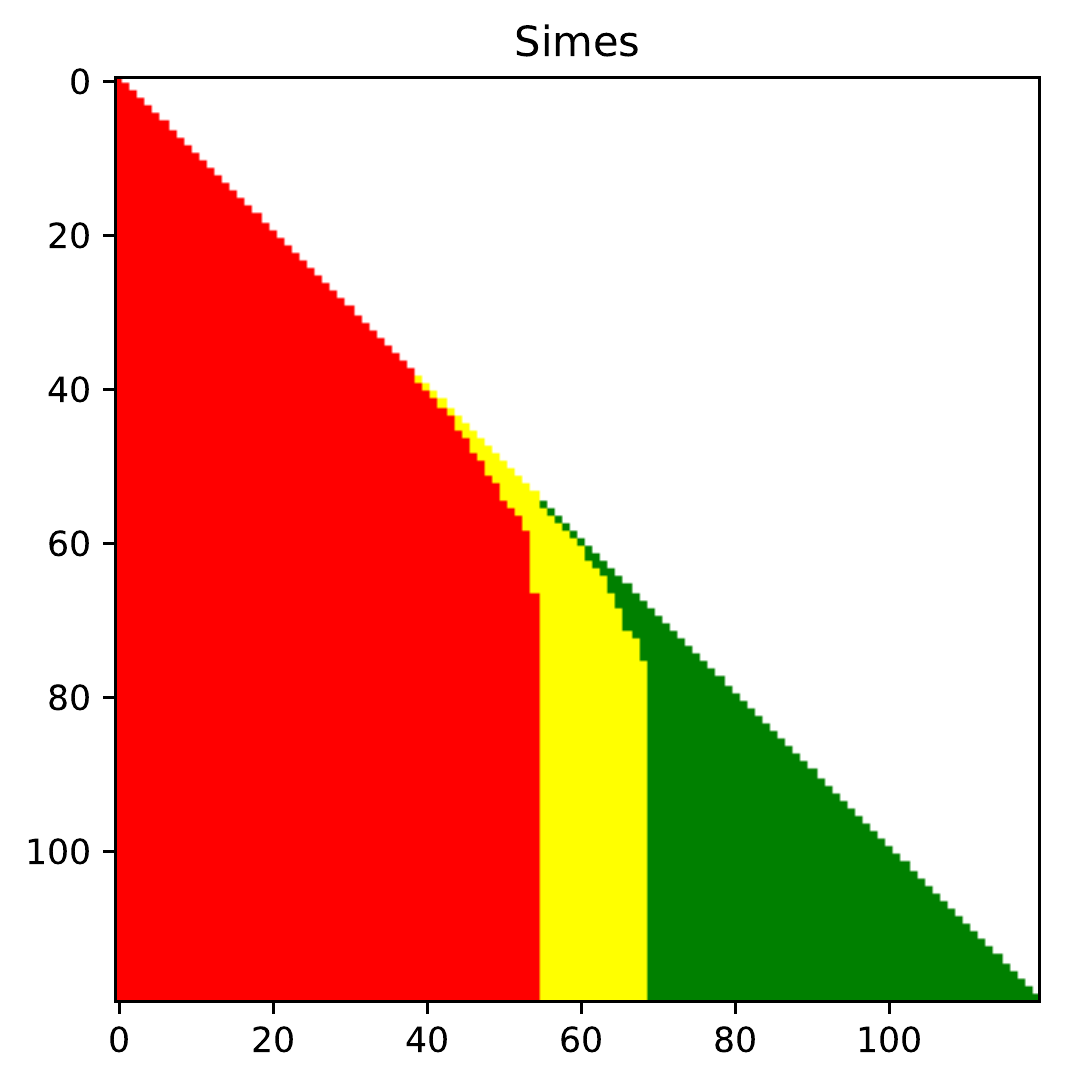} \\
    \includegraphics[width=0.32\textwidth]{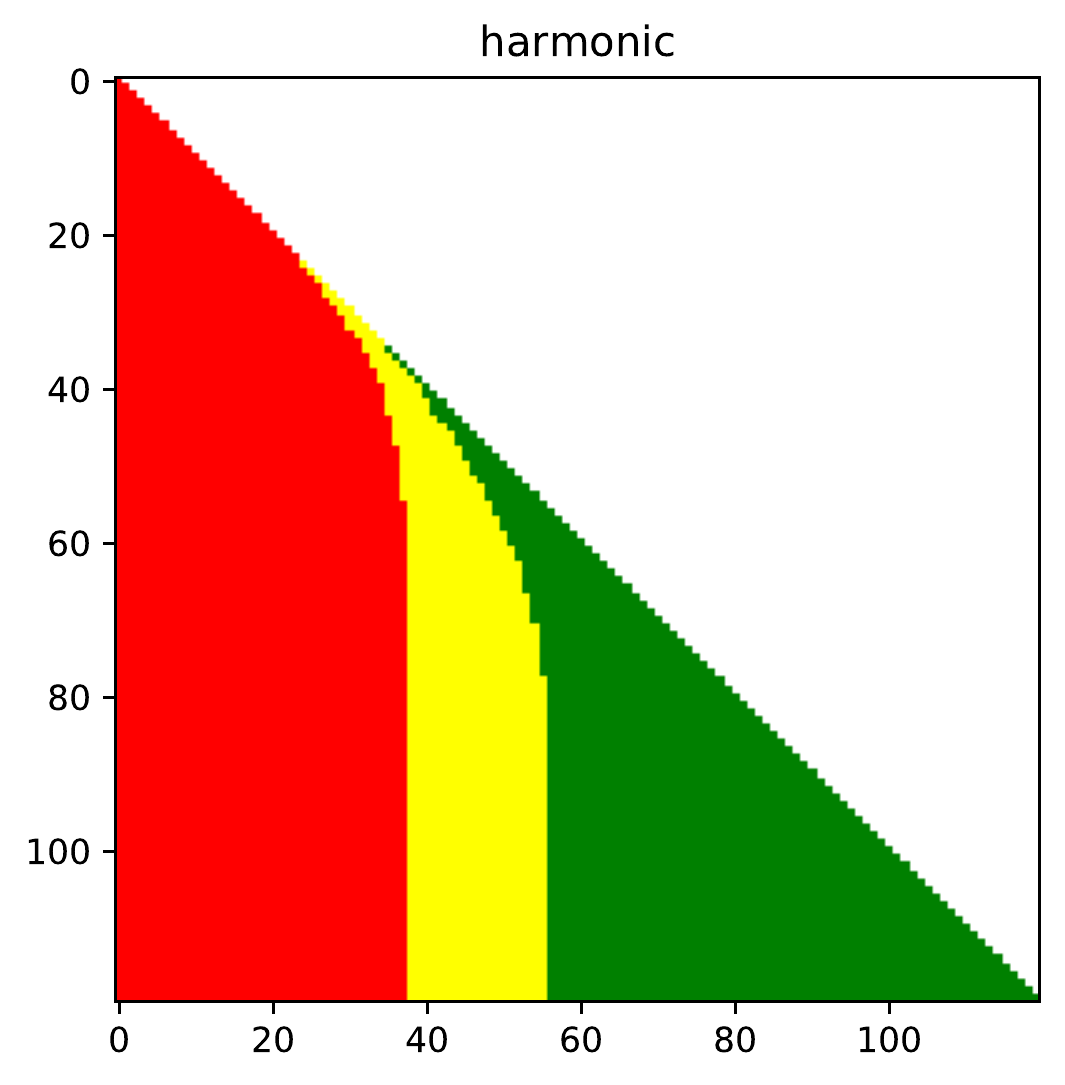}
    \includegraphics[width=0.32\textwidth]{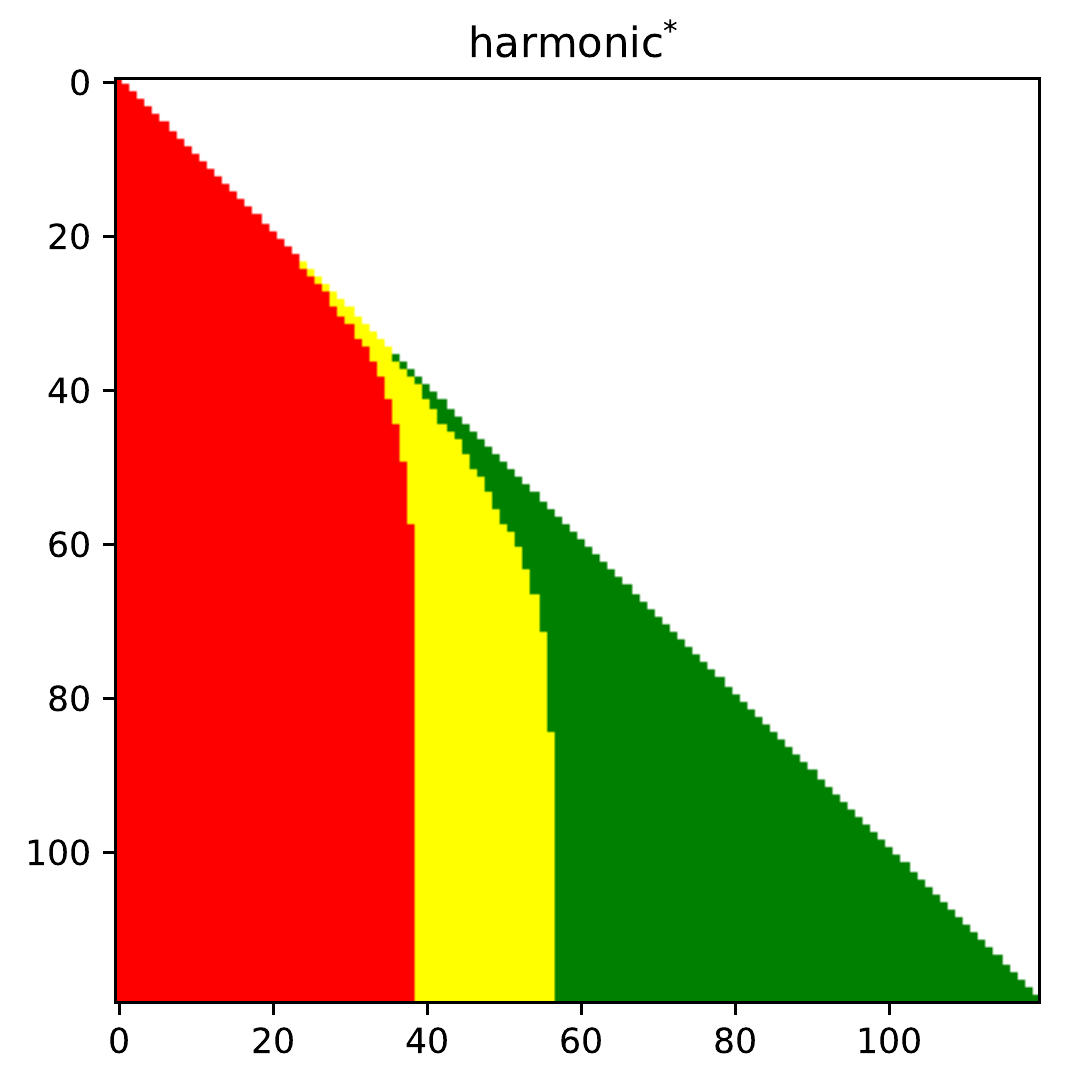}
    \includegraphics[width=0.32\textwidth]{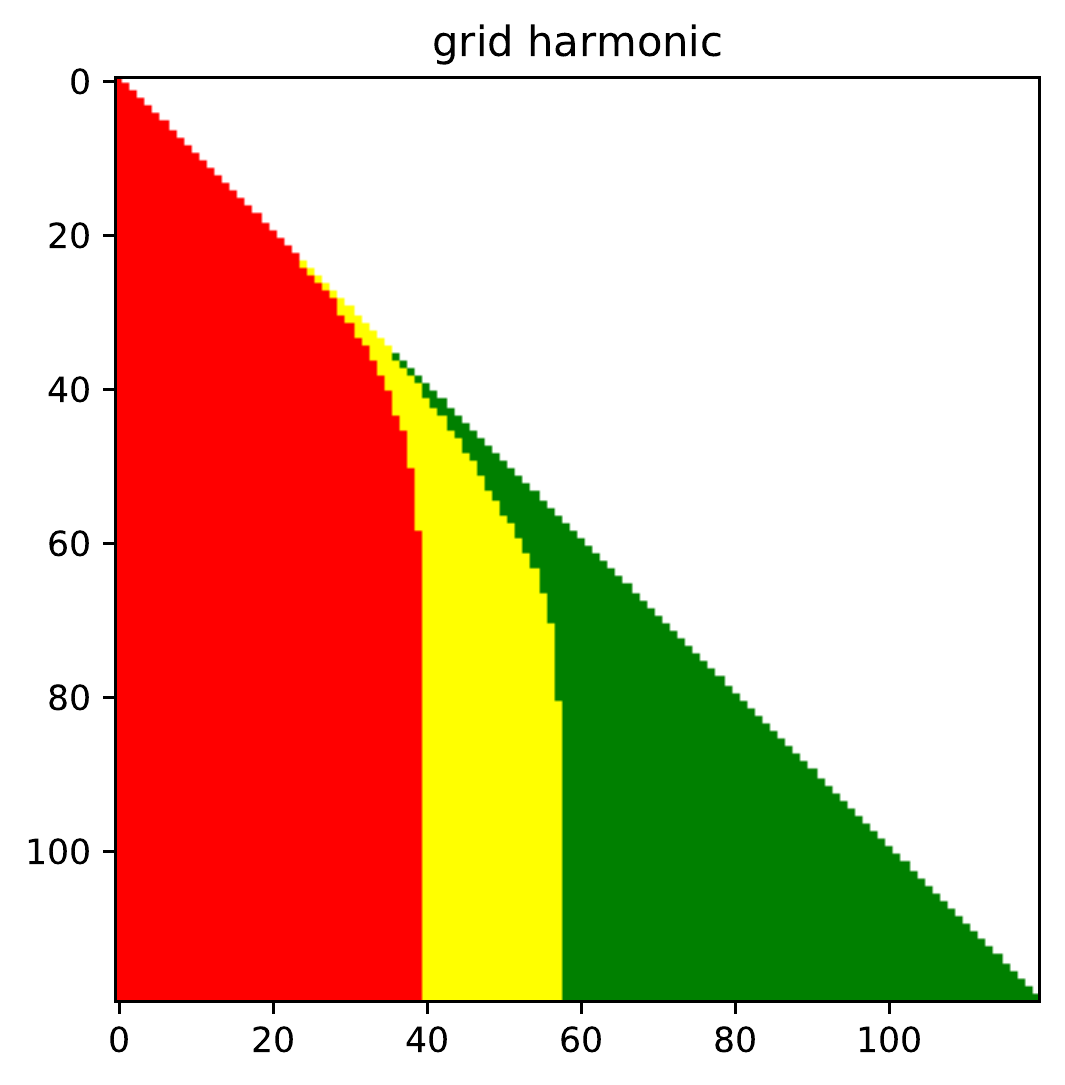}
  \end{center}
  \caption{An analogue of Figure \ref{fig:DM} with correlation $0$.}
  \label{fig:DM3}
\end{figure}

In this section we report some additional simulation results.
Figure \ref{fig:cdf2} is an analogue of Figure \ref{fig:cdf} with correlations $0.5$ and $0$ in place of $0.9$,
and Figures \ref{fig:DM2} and \ref{fig:DM3} are analogues of Figure \ref{fig:DM} with correlations $0.5$ and $0$,
respectively.
One interesting phenomenon is that the performance of the Bonferroni method improves as we approach independence.
The performance of the Bonferroni method also typically improves
when there are fewer observations from the alternative hypothesis:
see Figure~\ref{fig:cdf3}, where we have $0.1\%$ of observations from the alternative distribution in the left panel
(which coincides with Figure~\ref{fig:cdf})
and $1\%$ of observations from the alternative distribution in the right panel.

For other values of parameters (correlation, signal strength, signal sparsity, number of p-values) that we tried,
the relative performance of the four methods that are our main object of study
(Hommel, grid harmonic, harmonic, and harmonic${}^*$)
is qualitatively similar to the figures presented here and in Section \ref{sec:9}.

\begin{figure}
  \begin{center}
    \includegraphics[width=0.55\textwidth]{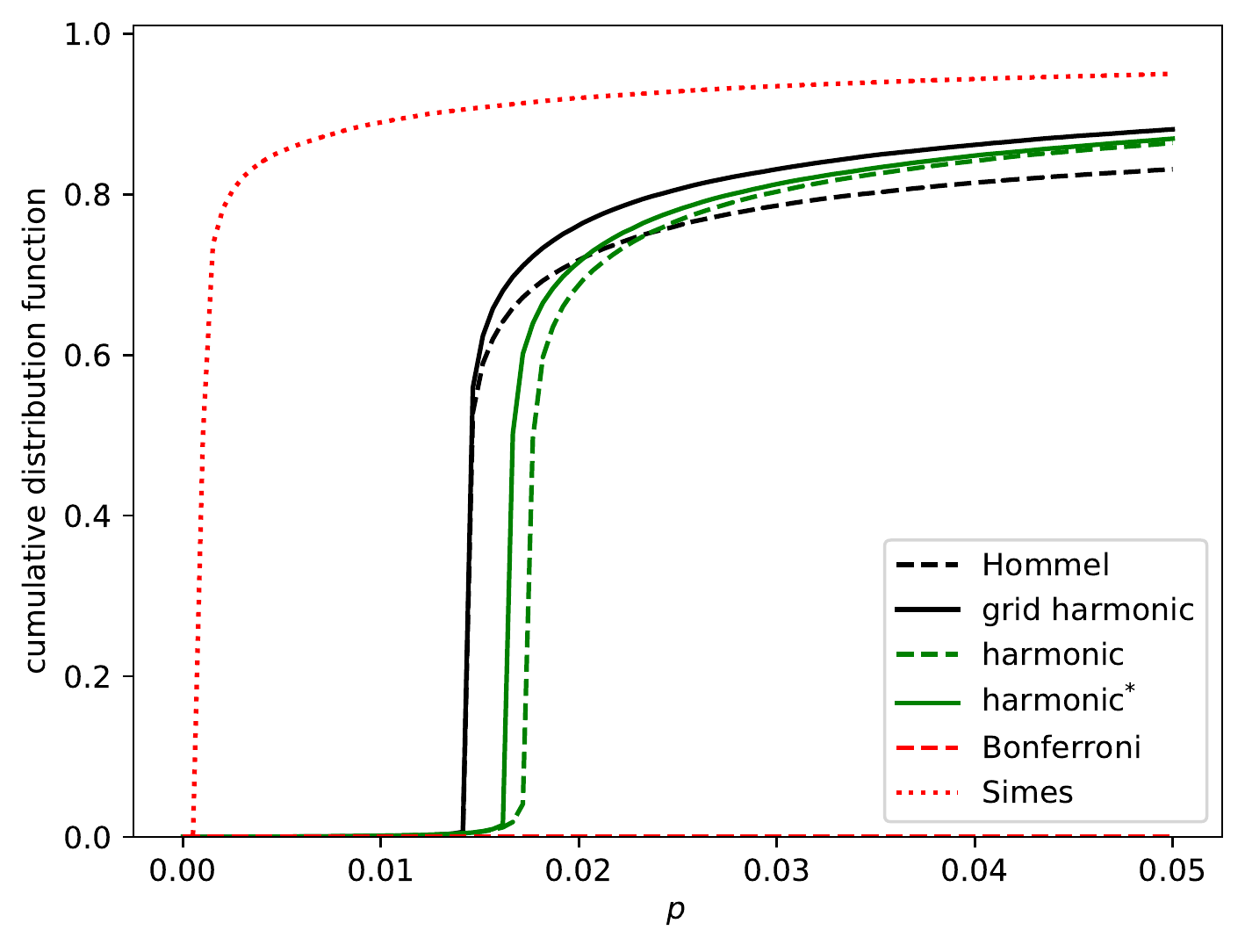}
  \end{center}
  \caption{An analogue of Figure~\ref{fig:cdf} for discrete p-values, as described in text.}
  \label{fig:cdf-discrete}
\end{figure}

Figures~\ref{fig:cdf-discrete} and~\ref{fig:epsilon} illustrate some specifics of merging discrete p-values.
Figure~\ref{fig:cdf-discrete} is produced in the same way as Figure~\ref{fig:cdf},
except that each input p-value $p$ is replaced by $\lceil D p\rceil/D$, where we take $D:=10^4$.
Now the Bonferroni function performs poorly; the corresponding curve is barely visible and coincides with the horizontal axis
(our definition \eqref{eq:Bonferroni} gives a combined p-value of 1).
We show only the most interesting part of the plot, for $p\in[0,0.05]$.
For small values of $p$ Hommel's p-merging function is now better than the harmonic and even harmonic$^*$.

\begin{figure}
  \begin{center}
    \includegraphics[width=0.6\textwidth]{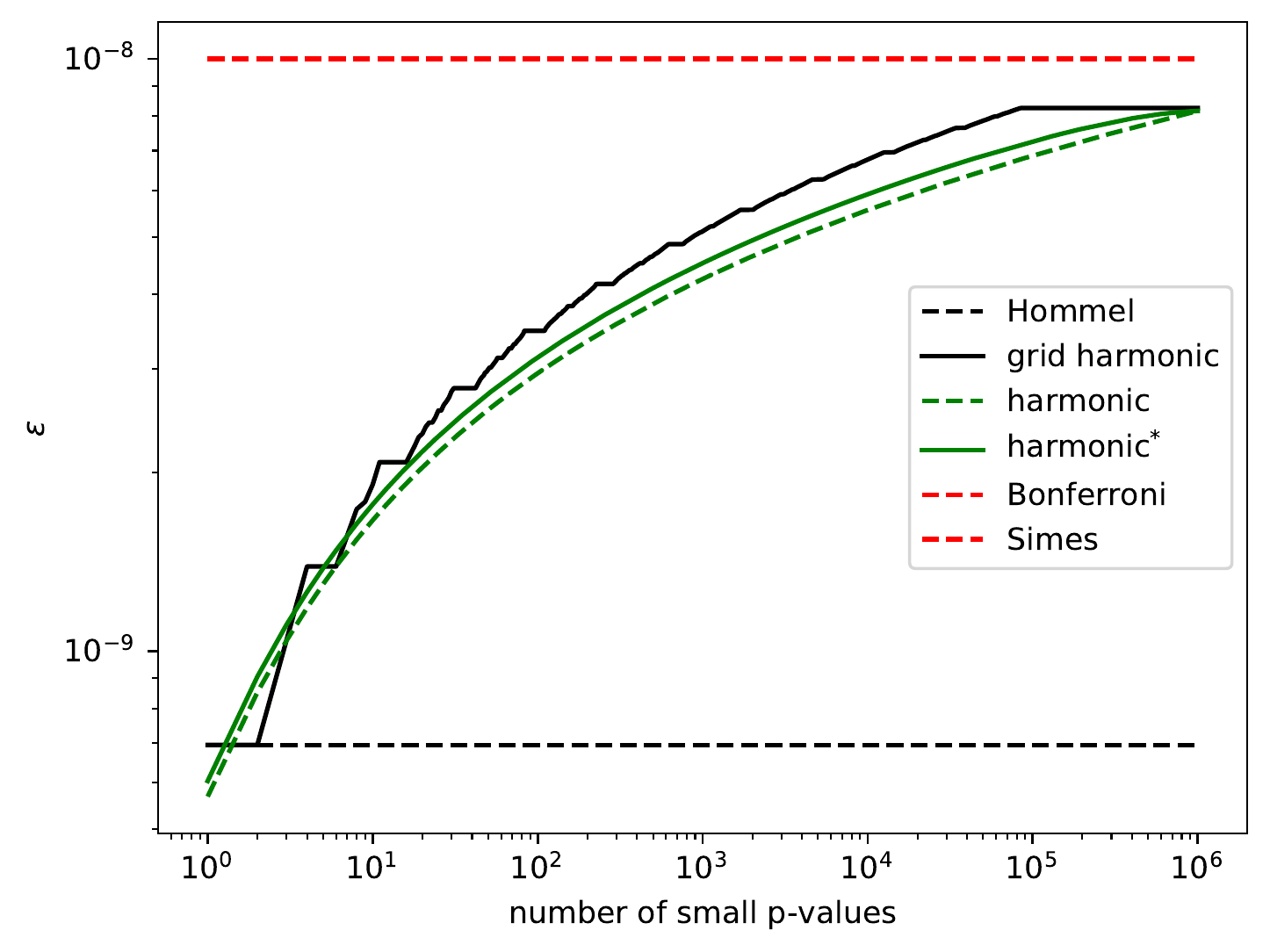}
  \end{center}
  \caption{The smallest p-value leading to the combined p-value of $1\%$,
    as described in text, for various merging methods.}
  \label{fig:epsilon}
\end{figure}

In Figure~\ref{fig:epsilon} we again consider a set of $K=10^6$ p-values generated by a test
(e.g., a rank test)
that produces p-values divisible by $\epsilon>0$.
A number $K_1\in\{1,\dots,K\}$ of these p-values are ``small''
(intuitively, correspond to a global null hypothesis being violated),
and the remaining $K_0:=K-K_1$ p-values are 1.
The small p-values are $\epsilon,2\epsilon,\dots,K_1\epsilon$.
The question that we ask in this toy scenario is:
how small should $\epsilon$ be in order for the combined p-value
to be highly statistically significant?

Figure~\ref{fig:epsilon} gives the borderline values of $\epsilon$
(leading to the combined p-value of $1\%$) as function of $K_1$ for six merging methods.
In this situation the Simes and Bonferroni methods produce the same borderline $\epsilon$ of $10^{-8}$ for all $K_1$.
These are the best results (in this context the higher the better),
while Hommel's method produces the worst result, $6.94\times10^{-10}$.
The graphs for the remaining merging methods are instructive in that,
whereas the grid harmonic method usually produces better results than harmonic and harmonic${}^*$,
the shape of its graph is much less regular.
While the discreteness of the grid harmonic calibrator \eqref{eq:Hkf}
is not noticeable in our previous figures,
in this combination with discrete p-values it becomes obvious.
In the middle of the plot, $K_1:=10^3$,
the borderline values of $\epsilon$ are $5.12\times10^{-9}$ for the grid harmonic method,
$4.25\times10^{-9}$ for harmonic, and $4.52\times10^{-9}$ for harmonic${}^*$.
\end{document}